\documentclass{article}
\usepackage{amsmath}
\usepackage{color}
\usepackage{amsthm,amssymb}
\usepackage{amscd}
\usepackage{graphicx}
\usepackage{ dsfont }
\newtheorem{theorem}{Theorem}[section]
\newtheorem{lemma}{Lemma}[section]
\newtheorem{definition}{Definition}[section]
\newtheorem{proposition}{Proposition}[section]
\newtheorem{corollary}{Corollary}[section]
\newtheorem{remark}{Remark}[section]
\newcommand{\D}{{\bf D}}

\newcommand{\ep}{\varepsilon}

\newcommand{\lan}{\lambda}
\newcommand{\Dim}{\mbox{\rm Dim}}
\newcommand{\1}{\mathds{1}}
\newcommand{\N}{{\mathbb N}}
\newcommand{\C}{{\mathbb C}}
\newcommand{\R}{{\mathbb R}}
\newcommand{\Z}{{\mathbb Z}}
\newcommand{\p}{\partial}

\newcommand{\overs}{{\overline{s}}}
\newcommand{\unders}{{\underline{s}}}

\newcommand{\diam}{\text{\rm diam\,}}

\newcommand{\Pibf}{{\bf\Pi}}

\newcommand{\musymb}{{\widehat\mu}}
\newcommand{\mhat}{\widehat{\bf m}}

\newcommand{\medsymb}{\musymb}
\newcommand{\inv}{G^{-1}}
\newcommand{\Wsimbolico}{{W}_{\sigma}(\widehat{P},\ell_n,w)}

\begin{document}

\title{Targets, local weak $\sigma$-Gibbs measures and a generalized Bowen dimension formula}

\author{Mar\'\i a Victoria Meli\'an P\'erez$^1$}
\date{}
\maketitle


\footnotetext[1]{Research supported by Grants MTM2013-46374-P and MTM2015-69323-REDT
Ministerio de Econom'a y Competitividad, Spain. }

\begin{abstract}
For a dynamical system, we study the set of points $\cal W$ whose orbit  approximates any chosen point at certain specified rates.
Our basic setting is  that of  left shift acting on topological Markov chains endowed with a local weak Gibbs measure. Our rates of recurrence are so fast that the corresponding set $\cal W$ has   measure zero, but  we obtain a generalized Bowen formula for  Carath{\'e}odory dimension.  For the  case of Markov transformations with countable partition and big image (BI) property a
Bowen-type formula is  obtained for the Hausdorff dimension of those exceptional sets. In particular,  we apply our general results to  Gauss and Luroth maps, a not Bernoulli modification of Gauss map and some inner functions. Since we  
only require  the existence of weak Gibbs measures we can deal with non 
H\"older  potentials and  we  can  also consider  intermittent systems as the Manneville-Pomeau map.

\end{abstract}

\medskip

\section{Introduction}

Let  $(X,d)$  be a locally complete separable
metric space endowed with a finite measure $\mu$ over the
Borel sets. 
The classical recurrence theorem of Poincar\'e 
says that
if  a measurable transformation $T:X \longrightarrow X$ preserves the measure $\mu$,
then $\mu$-almost every point of $X$ is recurrent, in the sense that
$$
\liminf_{n\to \infty} d(T^n(x),x)=0 \,.  
$$
Here $T^n$ denotes the $n$-th fold composition $T^n=T\circ T \circ \cdots \circ T$.

It is natural to ask if the orbit $\{T^n(x)\}$ of the point $x$
returns regularly not only to every neighborhood of $x$ itself as
Poincar\'e's theorem asserts, but whether it also visits every
neighborhood of a previously chosen point $y \in X$. Under the
additional hypothesis of ergodicity it is easy to check (see  e.g. \cite{FMP1}) that for
any $y\in X$, we have that
\begin{equation} \label{recurrence}
\liminf_{n\to \infty}  d(T^n(x),y) = 0 \,, \qquad \hbox{for $\mu$-almost all $x\in X$.}
\end{equation}
Recall that $\mu$ is ergodic if the only $T$-invariants sets (up to sets of $\mu$-measure zero) are trivial, i.e. they have zero $\mu$-measure or their complement have zero $\mu$-measure.

In this paper we will consider quantitative versions of (\ref{recurrence}). We are interested in studying the size of the set of
points of $X$ such that $\liminf_{n\to\infty} d(T^n(x),y)/r_n =0$, where $\{r_n\}$ is a given sequence of positive numbers and $y$
is a previously chosen point in $X$.
So, we study the size of the set
\begin{equation} \label{Wradios}
{\cal W}_T(X,r_n,y) = \{x\in X: \ d(T^n (x),y)< r_n \
\mbox{for infinitely many} \ n\}\,.
\end{equation}
If $d(T^k (x),y)< r_k$, then  the orbit of $x$ ``hits" at time $k$ the target-ball of center $y$ and radius $r_k$. We will refer to ${\cal W}_T(X,r_n,y) $ as the target-ball set.

\newpage

\noindent {\it Aim of this work}

\medskip

If $\mu$ is ergodic and the sequence $\{r_n\}$ is bounded below by some positive number,   it follows from (\ref{recurrence}) that the set   ${\cal W}_T(X, r_n,y) $ has full $\mu$-measure. In \cite{FMP2},  
under so called  eventual quasi-independence property, we proved  that 
if $\sum_n\mu(B(y,r_n))=\infty$ then the set ${\cal W}_T(X, r_n,y)$ has positive $\mu$-measure. 
  Now, we would like to study in more detail the case ${\cal W}_T(X,r_n,y)$ being  an  exceptional set, i.e with  zero $\mu$-measure.
Notice that by the direct part of Borel-Cantelli lemma we know that  if
$\sum_n\mu(B(y,r_n))<\infty$, then ${\cal W}_T(X,r_n,y) $ has zero $\mu$-measure.

\medskip

In this paper 
we consider the generic symbolic setting consisting of the left-shift acting on a topological Markov  chain endowed with  a local weak Gibbs measure, and we obtain  a generalized Bowen formula  for a  Caratheodory dimension  of the target-ball set.

This setting is particularly interesting due to the fact that
hyperbolic systems  are modeled  by countable Markov chains (via Markov partitions). And so one can expect this study to
provide a useful framework  when  considering similar questions in such systems.  In particular,  when the transformation is conformal enough, the Hausdorff dimension  for the target-ball set would follow from  the corresponding results for topological Markov chains.
Following this aproach, we also consider in this paper  the more concrete setting of Markov transformations with the ACIP measure (absolutely continuous with respect to Lebesgue measure invariant probability). 
  We use  the Carath{\'e}odory dimension results  in the symbolic setting to establish  Hausdorff dimension results for Markov transformations.  
  
Since  in our symbolic setting we just ask for {\it weak } Gibbs measures we are  able to deal with hyperbolic system with less regularity, and in particular we get   Hausdorff dimension results for the target-ball set in the case of  intermittent systems.  In a forthcomming work \cite{M} we consider the non-additive thermodynamic formalisms for countable Markov shifts and  this  allows the study of more general target sets.

\medskip

\noindent {\it Some previous results}

\medskip

If the set $T^{-n} (B(y,r_n))$ is a countable union ${\cal W}_n$ of pairwise disjoint balls, then the points in the target ball are those 
$x$ such that
$d(x,c_i)\leq R_i$ for infinitely many $i$ with $c_i$ the center and $R_i$ the radius of a ball in ${\cal W}:=\cup_n {\cal W}_n$. These points are called  {\it  well approximable points} by the system ${\cal W}$ in diophantine approximation theory. 
Having  a good distribution by size of the balls of the system one gets Hausdorff dimension results for the well approximable set. This notion was developed by 
Baker and Schmidt, refining some ideas of Besicovitch \cite{Be}, when    they   introduced   in \cite{BS} the concept of {\it regular system of intervals}. In that work   they  proved a generalization of the Jarnik-Besicovitch theorem  considering  approximation by algebraic numbers. 
In \cite{MP}  we extended the idea of regular system  to  {\it well-distributed system of balls} to obtain Hausdorff dimension results of the same kind.   Similar diophantine approximation results were obtained using a different  approach in \cite{V}. 
We would like to  remark  that our Cantor-like constructions  in this paper  (also  in \cite{MV}) are inspired, in a broad sense,
by the idea of regular systems. See section 1.3 of this introduction for a brief description.

\medskip

 The  Hausdorff dimension  of the target-ball set for different kind of  transformations has been studied for several authors. 
In \cite{HV1} and \cite{HV2}  Hill and Velani studied the case  of $T$ being an expanding rational map of the Riemann sphere acting on its Julia set. Urbanski  in   \cite{U} considered the case of conformal {\it countable} iterated function systems. In both cases the value of the Hausdorff dimension of the target-ball set was expressed as the zero of a Bowen formula. In section 1.1 we  explain  the symbolic counterpart of these results.

 We considered  in \cite{FMPa} the  general framework of expanding systems in metric spaces getting lower and upper  bound estimates on the dimension of target-ball and target-block sets. In this setting  we have a sequence of partitions $\{{\cal P}_n\}$  of the space with ${\cal P}_{n+1}$ finner that ${\cal P}_{n}$, and so 
the  target-block set is defined  by changing    the role of the  balls centered at $y$  by  elements of the  partitions containing the point $y$.
In particular we proved  Hausdorff dimension estimates for the target-ball/block sets for the Gauss map.  Afterwards,  Wang and Zhang got in \cite{WZ} the Hausdorff dimension for the  block-target set for the Gauss map. 
 In \cite{LWWX} Li, Wang,  Wu and Xu also studied  the Gauss map and obtained the Hausdorff dimension for a more general version of the target-ball set. They consider in (\ref{Wradios}) a sequence $\{y_n\}$ instead of $y$ and $\phi(n,x)=e^{-(f(x)+\cdots+f(T^{n-1}(x))}$ with $f$ a positive  continuos function instead of $r_n$. Previosly, in \cite{WW} Wang and Wu   
 got the Hausdorff  dimension for the case $y=0$. 

More recently, the dimension size of the recurrent version of the set (\ref{Wradios}), i.e  with $y=x$,  has been sudied in \cite{TW} by Tan and Wang for $\beta$-expansions,  and by Seuret and Wang in \cite{SW} for conformal iterated function system.

\

Before we describe in detail our results,  we would like to summarize  the main novelties contained in the paper:

\begin{enumerate}
\item By considering the target-ball set in the setting of topological Markov chains endowed with a  local weak Gibbs measure,  we found a common  susbtrate implicit in the examples of hyperbolic systems mentioned above.

\item   Using  the symbolic results, we have got the Hausdorff dimension of the target-ball set  for Markov transformations having the BIP property.

\item  By considering {\it weak} Gibbs measures in the symbolic setting we computed the Hausdorff dimension for the  target-ball set of some intermittent systems.

\item A generalized Hungerford lemma for Cantor-like sets on topological Markov chains is stablished.

\end{enumerate}

 \subsection{Target problem  for  Markov chains}

Our  dynamic setting can be described as  follows:

\smallskip

\noindent {\it Space, distance and transformation}
\smallskip

Given   a countable alphabet ${\cal I}$
 we denote by $\Sigma^{\cal I}$  the space of all infinite words formed with this alphabet. In $\Sigma^{\cal I}$  we define the distance
$$
d((i_k),(j_k))= 2^{-m} \quad \mbox{ with } \qquad m=\min\{n\in \N:  \ i_n\ne j_n\}, 
$$
where $\inf \varnothing:=\infty$. 
 The left shift $\sigma: \Sigma^{\cal I} \longrightarrow \Sigma^{\cal I}$ is the
continuous map defined by
$
\sigma (i_0,i_1,i_2,\cdots) =(i_1,i_2,\cdots)\,.$
Moreover, given a ${\cal I}\times {\cal I}$ transition matrix  $A=(a_{i,j})$ with entries $0$ and $1$,  we consider  the $\sigma$-invariant subset $\Sigma_A^{\cal I}\subset \Sigma^{\cal I}$ defined by
$$
\Sigma_A^{\cal I}=\{(i_0\,,\,i_1\,,\,\dots)\in\Sigma^{\cal I}: \ a_{i_k,i_{k+1}}=1 \ \hbox{for all $k=0,1,\ldots$}\}\,.
$$
The triple $(\Sigma_A^{\cal I},d, \sigma)$ is called a (one-sided) {\it topological Markov chain}.  Along this paper we will assume  that our topological Markov chain is topologically mixing (see definition in section 2.1).

 The $n$-{\it cylinders} in $\Sigma_A^{\cal I}$ (which generate the topology)  are the sets
$$
C_{i_0i_1\dots i_n} = \{(k_0\,,k_1\,,\,\dots)\in \Sigma_A^{\cal I}: \ k_{s}=i_s \ \hbox{for all $s=0,1,\dots,n$}\}\,, 
$$
and  for $w=(w_0,w_1,\ldots )\in \Sigma_A^{\cal I}$,  we denote  the $n$-cylinder $C_{w_0w_1\dots w_n}$ by  $C(n,w)$ . 

\medskip

\noindent {\it Measure}

\smallskip

We  endow our topological Markov chain $\Sigma_A^{\cal I}$ with  a $\sigma$-invariant local weak Gibbs measure $\widehat\mu$ associated to some potential $\phi:\Sigma_A^{\cal I}\longrightarrow \R$,  (see definition 3.1).  Considering weak Gibbs measures (instead of  Gibbs) will allow us to deal with non H\" older potentials. The  ``local" term just means that we add a constant depending on the first symbol in the usual definition of the  weak Gibbs measure of a cylinder. This will be useful for working with  Ruelle-Perron-Frobenius (RPF) measures  in the infinite countable case with BI property. 

\medskip

\noindent {\it Target-ball set}
\smallskip

Given  a point $w\in \Sigma_A^{\cal I}$, a $N$-cylinder ${\widehat P}$, and  a sequence $\{\ell_n\}\subset \N$   our taget-ball set is 
$$
W_{\sigma}({\widehat P}, 2^{-(\ell_n+1)},w)=\{z\in{\widehat P}: \ \sigma^k(z) \in
C(\ell_k,w) \ \hbox{for infinitely many }k\}.
$$
If  $\sigma^k(z) \in C(\ell_k,w)$, then $z_{k+j}=w_j$ for $0\leq j\leq \ell_k$, 
and therefore $d(\sigma^k(z),w)\leq 2^{-(\ell_k+1)}$. 
For simplicity we write $\Wsimbolico$ instead of $W_{\sigma}({\widehat P}, 2^{-(\ell_n+1)},w)$

\

Theorem A for finite alphabet and theorem B  for infinite countable alphabet, below, give Bowen formulas for the $\widehat\mu$-dimension (defined by Carath{\'e}odory's construction, see section 2) of these target-ball sets.

\medskip

Our approach to study these sets is based 
on the thermodynamic
formalism 
for countable Markov shifts. Thermodynamic formalism goes back to
the classical work of Ruelle, Sinai and Bowen (\cite{Ru1}, \cite{Si}, \cite{Bo}) in which they  translate  the ideas of  Gibbsian statistical  mechanics to the field of  dynamical systems.
For infinite countable alphabet,  the thermodynamic formalism   has been studied by many authors, (see e.g   \cite{Gu}, \cite{MU2},  \cite{MU3}, \cite{Sa}, \cite{Y}). 
 In this paper we follow  Sarig's approach, see \cite{SaNotes}, \cite{Sa2015},  and   we  use the  Gurevich pressure and  Sarig's results on the existence of equilibrium and Gibbs measures based on his study of  Ruelle's operator.  
\smallskip

For the sake of  clarity of exposition next we split the discussion into  finite and infinite countable alphabet.
We will assume that 
$$
s:=\lim_{n\to\infty}  -\frac 1n \log { {\musymb(C(\ell_n,w))}}<\infty.
$$
If $\widehat\mu$ is  weak  Gibbs but not Gibbs, then we  also require  $\limsup_{n\to\infty} {\ell_n}/{n} <\infty$.
If $s=\infty$ then the target-ball set has zero $\widehat\mu$-dimension; if $\lim_{n\to\infty} {\ell_n}/{n}=\infty$ and $\widehat\mu$ is ergodic  with  $\phi\in L^1(\widehat\mu)$ and $\phi$ with  enough regularity  (see section \ref{special}),
then $s=\infty$  for $w$  $\widehat\mu$-a.e.
We recall that $\phi$ denotes the potential of our measure $\widehat\mu$.

\subsubsection{Finite alphabet}

We have the following  formula for the $\widehat\mu$-dimension of the target-ball set (see theorem \ref{introA}).

\medskip

{\noindent  \bf Theorem A. }{\it  If there exists a mixing weak $\sigma$-Gibbs measure with continuous potential $t\phi$ for all $0<t\leq 1$, then 

 $$
\Dim_{\widehat\mu}(W_{\sigma}(\widehat P,\ell_n,w)))=T,
$$
with $T$ the unique root of the equation
\begin{equation}\label{B}
P_{top}(t\phi)-P_{top}(\phi)t=st.
\end{equation}
  Here $P_{top}(.)$ denote  the classical topological pressure.
 }

 \
 
 As stated we only need mixing weak $\sigma$-Gibbs measures for the potentials $t\phi$.  But, if $\phi$ satisfies   Walter's condition (see section 3), then from  classical results of Bowen and Ruelle we known that there is an exact (whence ergodic and strongly mixing)  $\sigma$-Gibbs measure with potential $t\phi$, for all $0<t\leq 1$. Notice that for any measurable set $E\subset \Sigma_A^{\cal I}$ we have that $\Dim_{\widehat\mu}(E)\leq 1$,  and if $\widehat\mu(E)>0$, then $\Dim_{\widehat\mu}(E)=1$ .

 \

We can rewrite (\ref{B}) as a  Bowen equation, more precisely as  
$$P_{top}(t\psi)=0 \qquad \mbox{ with } \qquad \psi=\phi-P_{top}(\phi)-s. 
$$
 Bowen,  in his study of quasi-circles \cite{Bo2},   first  described  what we now call Bowen's formula, the Hausdorff dimension of a set  as the solution of a pressure equation.   Bowen's result was extended  by  Ruelle \cite{Ru3} to  the repeller $J$ of  a $C^{1+\ep}$ map $f$ which is conformal and topologically mixing  on $J$. The Hausdorff dimension of the conformal repeller $J$ is the unique root $t_0$ of the  equation $P_{top}(-t\phi)=0$  with $\phi:J\longrightarrow \R$ defined by $\phi(x)=\log\|D_xf\|$, and moreover the $t_0$-dimensional Hausdorff measure of $J$ is positive and finite and equivalent to the Gibbs measure with potential $-t_0\phi$.

\smallskip
 
In \cite{HV2}   Hill and Velani   studied the target-ball set for $f$ an expanding rational map of the Riemann sphere acting on its  Julia set $J$. They proved  that the Hausdorff dimension of  $W_f(J, e^{-n\tau}, y)$ was the unique root of  the Bowen equation 
$P_{top}(-t\log |f^{\prime}|)=t\tau$.

The symbolic version of Hill and Velani  result  corresponds   to considering, in theorem A,  the Markov chain associated (via Markov partition) to the Julia set $J$, and $\widehat\mu$ as 
 the $\sigma$-Gibbs measure for the potential $\phi=-t_0\log|f^\prime\circ\pi|$, with $t_0$ is the Hausdorff dimension of $J$. (Here $\pi$ denotes the standard projection from the symbolic model to the Julia set). From the thermodynamic formalism we know that the 
 H\"older continuity of the potential $\phi$ implies the existence of $\sigma$-Gibbs measures with potentials $t\phi$.
To derive   a Hausdorff dimension formula for $W_f(J, e^{-n\tau}, y)$, via our symbolic modeling and Carath{\'e}odory dimension, conformality is, of course, crucial, but we will not pursue this matter here.

 \medskip
 
 We would like to remark that some  recent  results  show the existence of weak $\sigma$-Gibbs measures  for  non-H\"older potentials, see e.g  \cite{FFY},  \cite{Hu},  \cite{Yu1}, \cite {Yu2}, \cite{Ke}.  Consideration of  weak  $\sigma$-Gibbs  measures  will  allow us  to deal with hyperbolic systems with less regularity.  As an example, in section  5.2.3,  we consider some non-H\"older potentials studied by Hu in \cite{Hu}.

\subsubsection{Infinite countable alphabet}

Sarig proved in \cite{Sa2}  that under  the {\it big images and preimages} (BIP)  property,  the RPF measure (coming from his generalization of  Ruelle's Perron Frobenius theorem)  is a $\sigma$-Gibbs measure for any potential with enough regularity. A closely related result was simultaneously obtained by Mauldin and Urbanski in \cite{MU2}.
We recall  that the BIP property  is the following  condition on the transition matrix of   $\Sigma_A^{\cal I}$: there exist a finite set of symbols ${\cal I}_0\subset{\cal I}$ such that  for each symbol $i\in{{\cal I}}$ there exist $k,\ell\in{\cal I}_0$ such that $C_{ki\ell}\neq \varnothing$.

The existence of local  weak $\sigma$-Gibbs measures for potentials $t\phi$ (with $0<t\leq 1$)  is an important tool  in our estimation from below of the $\widehat\mu$-dimension  of the target-ball set. We use  these measures to construct Cantor-like sets  inside  $\Wsimbolico$ with large $\widehat\mu$-dimension. Hence,  a more precise result is obtained if we assume BIP property.  Getting the upper bound of the $\widehat\mu$-dimension  of the target-ball set is easier  (see proposition \ref{cotasupformalismo}).

In theorem 5.4  we state our $\widehat\mu$-dimension result under the weaker condition {\it Big images} (BI) property. This condition is equivalent to the condition 
$\inf\{\widehat\mu(\sigma(C_i))\, :\,  i\in{\cal I}\}>0.
$
Our more general result for the lower bound of $\Dim_\medsymb \, \Wsimbolico $ is corollary 5.1.

\medskip

If the BIP property holds,  and $\widehat\mu$ is a $\sigma$-Gibbs measure with   potential $\phi$ such that
$
\sum_{n= 1}^\infty V_n(\phi)<\infty$  
with 
$$V_n(\phi):=\sup\{|\phi(w)-\phi(w')| :  w=(i_0,i_1,\ldots) ,w'=(j_0,j_1,\ldots)\in \Sigma_A^{\cal I }, \, i_k=j_k, \, 0\le k\le n-1\},$$
and $\sup \phi<\infty$, 
then we have  the following (see theorem 5.6) :

\medskip

{\noindent  \bf Theorem B.}
{\it 
 If there exists $0<t_1\leq 1$ such that
 $$
\infty> P_{G}(t_1\phi)-P_{G}(\phi)t_1>st_1 \quad \mbox{ and }\quad 
-\sum_{i\in{\cal I}}\widehat\mu(C_i)^{t_1}\log\widehat\mu(C_i)<\infty \;  ,
$$
then
\begin{equation}
 \begin{aligned}
 \Dim_{\widehat\mu}(W_{\sigma}(\widehat P,\ell_n,w)))&=\sup\{t\geq t_1: P_{G}(t\phi)-P_{G}(\phi)t>st\}  \notag \\
 &=\inf\{ t>0\, : P_{G}(t\phi)-P_{G}(\phi)t<st\} .
\end{aligned}
\end{equation}
And  if moreover $\phi$ is weakly H\"older continuous, then  
$$ \Dim_{\widehat\mu}(W_{\sigma}(\widehat P,\ell_n,w)))=T
$$
 with
$P_{G}(T\phi)-P_{G}(\phi)T=sT.
$
Here 
$P_G(.)$ denotes the Gurevich pressure
}

\

The function $t\longrightarrow P_G(t\phi)-P_G(\phi)t=P_G(t[\phi-P_G(\phi)])$ is not necessarily continuous, this was first pointed out by Mauldin and Urbanski in \cite{MU} (for their pressure); see also \cite{MU3}. Hence  the {`}sup' above is not necessarily a maximum.

\medskip

In  \cite{U} Urbanski   studied
the target-ball problem   for the limit set of a conformal {\it countable } iterated function system, and  got a result like Bowen's 
 for the Hausdorff dimension. We recall that an iterated function system is a collection of injective contractions (all with the same contractive constant). The symbolic version of his result corresponds to considering, in theorem B,  Markov chains with the {\it Bernoulli } property (i.e. all the entries in the transition matrix are $1$),   and a potential with strong regularity properties.  Notice that BIP property  is weaker than Bernoulli property.

  \medskip

In section 5.1.2 we give an alternative approach to get a lower estimate of the  $\widehat\mu$-dimension of  $W_{\sigma}(\widehat P,\ell_n,w))$,  in this case we  do not require the existence of new measures, but  we ask for the initial measure $\widehat\mu$ to be mixing. We get the  lower bound 
$$
\Dim_{\widehat\mu}(W_{\sigma}(\widehat P,\ell_n,w)))\geq \frac{P_G(\phi)-\int \phi \, d\widehat\mu}{P_G(\phi)-\int \phi \, d\widehat\mu +s}.
$$
We can think of this value as a lower bound for the root $T$ in theorems A and  B.
We would like to mention that  in 
\cite{FMPa} we obtained  a related bound for the target-ball set in the general setting of  expanding maps in metric spaces.

\

For example,  we have the following result that corresponds to the  symbolic version of  the Gauss map (see  remark  \ref{Gaussmu}).

\medskip

{\noindent  \bf Theorem C.}
{\it 
Let $\widehat\mu$ be the RPF probability  for the potential $\phi:\Sigma^{\N}\longrightarrow\R$ defined by
$$
\phi(w)=2\log \left |
\pi(w)
\right | \qquad \mbox{ with } \qquad  \pi(w)=\lim_{j\to \infty}\frac{1}{ w_0+\dfrac 1{w_1+\dfrac 1{\ddots +w_j}}}
$$
Then 
$$\Dim_{\widehat\mu}(W_{\sigma}(\widehat P,\ell_n,w)))=T \, \geq  \, 
\frac{h_{\widehat\mu}}{h_{\widehat\mu} +s}=
\frac{\pi^2}{\pi^2+(6\log2)s}
$$
where $1/2<T\leq 1$ is the unique solution of $P_{G}(t\phi)=st$. Here $h_{\widehat\mu}$ denotes the entropy of ${\widehat\mu}$.
}

\subsection {Markov transformations and Intermittent systems}

Next we turn to the target-ball problem
for $f$ a Markov transformation. In this case   our space is the interval $[0,1]$,  distance is  euclidean, and the $f$-invariant probability is the ACIP measure $\mu$. Our objective is to get estimates for the Hausdorff dimension of target sets for $f$.
Recall that any Markov transformation $f$ has a symbolic representation (via the Markov partition ${\cal P}_0$).
Our approach is to use this representation and the Carath{\'e}odory dimension results  described  in  the previous section.

\medskip

 Let 
$(\Sigma_{A}^{\cal I},\sigma)$  be the symbolic representation of $f$,  let $\pi$  denote  the standard projection from $\Sigma_{A}^{\cal I}$ to the interval $[0,1]$,  and   $\widehat\mu=\mu\circ\pi$ be the corresponding measure.  We will assume that the BI property holds,  then  $\widehat\mu$ is a local $\sigma$-Gibbs measure with potential $-\log|f^{\prime}\circ \pi|$ 
; see proposition 5.4.  

The collection $\{ {\cal P}_n:=\bigvee_{j=0}^n f^{-j}({\cal P}_0)\}$ 
of partitions of the interval $[0,1]$ give us a {\it grid } of $[0,1]$; we use  $P(n,z)$  to denote the $n$-block in ${\cal P}_n$ which contains the point $z$.
There is a  direct correspondence  between the target-ball set  in $(\Sigma_{A}^{\cal I},\sigma)$  and 
 the  target-{\it block} set  for $f$, i.e the set of points in $[0,1]$ whose orbit by $f$ hits at time $k$ (for infinitely many $k$) the $\ell_k$-block of the partition ${\cal P}_{\ell_k}$  which contains the point $x:=\pi(w)$. In particular for $f(x)=10x (\mbox{mod } 1)$ or $f$  the Gauss map,  the study of the size of the target-block set is,  in fact, a  very classical problem in diophantine approximation, the study of 
 exceptional sets  with  prescribed string of digits in his decimal  or continued fraction expansions. (See section 6.3 for  results on continued fraction and Luroth expansions).

 Our dimension results for the target-ball set for $f$ come from consider an appropriate  target-block set inside.
The  $\widehat\mu$-dimension results in $(\Sigma_{A}^{\cal I},\sigma)$  will translate into {\it grid}-dimension results in the interval $[0,1]$ for $f$. The  {grid}-dimension is defined as the Hausdorff dimension but  in  the coverings  only intervals of the grid intervene. 
Grid and Hausdorff dimensions  coincide for a set $A$  (see proposition 2.1)  if for all $\gamma>0$ and for all $z\in A$ 
\begin{equation} \label{gridvsH}
\frac{\diam(P(n,z))}{ \diam(P({n-1},z))^{1+\gamma}}\geq C>0 
\end{equation}
For Markov transformations with finite alphabet,  condition (\ref{gridvsH}) holds (with $\gamma=0$) for {\it all } points in $[0,1]$,  and  we have  equality between both dimensions, but  this is not true for countable infinite alphabet. 

For our Hausdorff dimension
results   we require that 
condition 
(\ref{gridvsH}) holds on the Cantor-like sets  constructed to get  the lower bounds of  the dimension.  
We get this  property by requising, in the symbolic framework,  for a similar condition  involving 
the RPF measure with potential  $-t \log|f^{\prime}\circ \pi|$ (in the case  this potential is positive recurrent) instead of the diameter. 
 Essentially, we  use that there is  a fixed proportion (in measure) of {\it good } points (see definitions  3.2 and \ref{unifgood}) in a collection   ${\cal D}=\{ C(0,\sigma^{p_i}(w))\}$ of $0$-cylinders  with $\{ p_i\}$ an increasing sequence in $\N$ and $w$ the target-center. Good points come from  uniform convergence  in the Birkhoff's ergodic theorem
for the potential of the measure.

If the collection $\cal D$ is finite, then  we obtain  this fixed proportion of { good } points in $\cal D$  by using that the RPF measure is mixing.  Therefore  the  ergodicity of the ACIP and the recurrence theorem of Poincar\'e (see (\ref{recurrence})) give us   the desire proportion of good points for  $w$ $\lan$-a.e.

However,
it is interesting to understand the  mixing properties   behind this proportional distribution of good points in any  (countable) collection of $0$-cylinders  and   for any 
  local {\it  weak} $\sigma$-Gibbs measure  (and not only for a local $\sigma$-Gibbs  measure  with potential  $-t \log|f^{\prime}\circ \pi|$).  
  Surely in forthcoming situations (more general than Markov transformations)  it will be  useful to go from grid to  Hausdorff dimension. 
With this purpose, in section 3.2 we define a  new mixing property 
which implies 
having a fixed proportion of good points in a collection  of  $0$-cylinder (see theorem 3.1). This property is a generalization of a mixing condition already used in \cite{FMP2}.   In particular,  if the potential  $-t \log|f^{\prime}\circ \pi|$  is positive recurrent and BIP property holds, then the RPF measure is a local $\sigma$-Gibbs measure which satisfies this mixing property in the collection of {\it all } $0$-cylinders.  In fact,  the RPF measure satisfies a stronger mixing property which is called  exponentially continued fraction mixing.

\medskip

Section 6 contains our results on the Hausdorff dimension of  targets sets for Markov transformations with BI property. A very precise result for Markov transformation with BIP property is stated in section 6.2. Recall that BIP property is weaker than Bernoulli property.
In particular, in section 6.3  we consider 
 some concrete examples such as: the Gauss transformation, a not Bernoulli modification of the Gauss transformation, the Luroth map and some inner (analytic) functions.

\medskip

\noindent For  example, we consider  
 the following  (not Bernoulli) modification of the Gauss map $\phi(x)=\frac 1x -\lfloor{\frac 1x}\rfloor$
$$
f(x)=
\begin{cases}
\phi(x)\,, \quad &\text{if } \frac 12<x\leq 1
\,, \vspace{3pt}
\\
\left(1-\frac{1}{\lfloor{\frac 1x}\rfloor}\right)\phi(x)+\frac{1}{\lfloor{\frac 1x}\rfloor}\,, \quad &\text{if } 0<x\leq \frac 12 \,.
\end{cases}
$$
The initial partition for $f$ is, as in the Gauss map,  ${\cal P}_0=\{P^0_i:=(1/(i+1),1/i) : i\in\N\setminus\{0\}\}$,  but
in this case 
$f(P^0_1)=(0,1)$  and  $f(P^0_i)=(1/i,1)$ for $i\geq 2$.

\medskip

Let $\{r_n\}$ be a sequence of radii such that 
$
u:=-\lim_{n\to\infty}\frac 1n\log{r_n}<\infty, 
$
then  (see theorem 6.5)

\medskip

{\noindent \bf Theorem D.} {\it 
Let $1/2<T\leq 1$ be  the unique solution of  $P_G(-t\log|{f^{\prime}\circ\pi}|)=tu$, then for $\lan$-a.e  $x\in (0,1)$
$$
\Dim \left\{y\in [0,1]: \ \liminf_{n\to\infty} 
\frac{|f^n(y)-x| }{r_n}=0 \right\} = T  \ge \frac{\int \log|f'|\, d\mu}{\int \log|f'|\, d\mu +u}
$$
}

Also  we get dimension results for  target problems in the case of  non-uniformly expanding transformations. 
In particular,  for the Manneville-Pomeau transformation $F(x)=x+x^{1+\alpha} \mod 1$ with $0<\alpha<1$. If we  
define
 $$
  \widetilde u=
 \begin{cases}
u,\qquad \qquad \quad \text{ if } x\neq 0 \\
(1-\alpha)u, \qquad \text{ if } x=0.
 \end{cases}
 $$ 
with $
u:=-\lim_{n\to\infty}\frac 1n\log{r_n}<\infty, 
$
then we get the following (see theorem 7.2):

\medskip

 {\noindent \bf Theorem E. }{\it  Let $\alpha<T\leq 1$  be 
 the unique solution of $P_{top}(-t\log|{F^{\prime}}|)=t\widetilde u,$ then 
 $$
\Dim \left\{y\in [0,1]: \ \liminf_{n\to\infty} 
\frac{|F(y)-x| }{r_n}=0 \right\} = T \geq h_{\mu}/(h_{\mu}+\widetilde u)
$$
with  $h_{\mu}$  the entropy of  the ACIP measure of $F$.
}

\medskip

It is interesting to notice that we have bigger  Hausdorff dimension  for the case $x=0$ (the point with $T^{\prime}(x)=1$).

\subsection{A Hungerford lemma for Cantor-like sets}

The kind of Cantor-like sets that we have managed to deal with in this paper  have the peculiarity that the ratio between the $\widehat\mu$-measure of a parent and his children increase wildly. This, in general,  would imply zero dimension. However, if we add a nice property which guarantees that 
the sets of the same generation with the same parent (i.e brothers) are separated enough, then we  get positive dimension.
We think that the study of the dimension of  Cantor-like sets with this characteristic and the relation with the thermodynamic formalisms is interesting by itself, and it would be useful in other contexts. For this reason, in   
 section 4  we present  a general description, as pattern subsets in the symbolic space, of sets with this structure  and we prove some  results (with  the flavor of the classical Hungerford lemma, see e.g \cite{Pom}) on lower estimates of their dimension. See corollaries \ref{G1} and \ref{G2} for the Gibbs case, the simplest one.

\subsection {Outline of the paper.} 

\medskip

\rm  In section 2 we introduce (following Carath{\'e}odory approach) the different dimensions  that we use along this paper and we compare them.   In section 3 we include several results on thermodynamic formalisms;  we also define a mixing property for local  weak $\sigma$-Gibbs measures which implies a fixed proportion of good points in $0$-cylinders. In section 4 we introduce a family of Cantor-like sets in the symbolic  space and state some Hungerford  dimension results.  The target problem for the shift transformation is the content of section 5. In sections 6 and 7 
we consider  the target problems (ball and block) for Markov transformations and intermittent  systems respectively.

\medskip

\noindent {\it  A few words about notation.} \rm 
For two sequences $\{a_n\}$ and $\{b_n\}$ we write  $a_n \asymp b_n$ 
 if for some fixed positive absolute constant $c$ we have that $c^{-1}b_n\leq a_n\leq c\, b_n$ for $n\ge 1$. 
 If $\lan$ and $\mu$ are two measures on $X$, we will say that $\lan$ and $\mu$ are comparable on a subset $Y\subset X$ if for some positive constant $C=C(Y)$ we have that $C^{-1}\mu(Z)\le \lan(Z)\le C\mu(Z)$ for any Borel subset $Z$ of $Y$. 
By convention, if $A$ denote the empty subset in the set $[0,1]$, then $\sup(A)=0$ and $\inf(A)=1$.

\section{Carath{\'e}odory dimensions}

Let $(X,d)$ be a separable metric space and let denote by $\diam(.)$  the diameter. Next we recall the Carath{\'e}odory's construction of Borel measures (see  e.g. \cite{M}).

Let   ${\cal F}$ be  a family of subsets of $X$, and $\varphi$   a non-negative function  on  ${\cal F}$ verifying:
\begin{itemize}
\item[(i)]  For every $\ep>0$ there are $F_1,F_2,\ldots \in {\cal F}$ such that $\diam(F_i)\leq \ep$ and $X=\cup_{i=1}^{\infty} F_i$.
\item [(ii)]  For every $\ep>0$ there is $F\in {\cal F}$ such that $\diam(F)\leq \ep$ and $\varphi(F)\leq \ep$.
\end{itemize}
We will say that $({\cal F}, \varphi)$ is a {\it Carath{\'e}odory's pair}.

For $0<\ep\leq \infty$ and $A\subset X$ we define the  regular Borel measure 
\begin{equation}\label{Caratheodory}
 {M}_{{\cal F},\varphi}(A)=\lim_{\ep\to 0}{
M}_{{\cal F},\varphi,\, \ep}(A)
\end{equation}
with
$$
{M}_{{\cal F},\varphi,\, \ep}(A)=\inf\{\sum_i\varphi(F_i)\, : \, A\subset \cup_i F_i \, , \, \diam{F_i}\leq{\ep}\, , \, F_i\in{\cal F}\}
$$
In  particular, we can consider a Carath{\'e}odory's pair $({\cal F}, \varphi)$ with $\varphi(.)= \psi(.)^{\alpha}$ with $\psi$ a non-negative function on ${\cal F}$ and $0\leq \alpha<\infty$. 
It is not difficult to check that  if the function $\psi$ verifies that
\begin{equation}\label{paradefdim}
\sup\{ \psi(F) \, : \, F\in {\cal F} \, ,\, \diam(F)\leq \ep\}\to 0 \quad \mbox{ when} \quad \ep\to 0
\end{equation}
then 
the regular Borel measures $M_{{\cal F},\psi(.)^{\alpha}}$  satisfies for all $A\subset X$
\begin{align}
{M}_{{\cal F},{\psi}^{\alpha}}(A)<\infty &\implies {M}_{{\cal F},{\psi}^{\beta}}(A)=0 \,\, \, \quad \text{ for }\beta>\alpha, \notag \\
{M}_{{\cal F},{\psi}^{\alpha}}(A)>0 \, &\implies {M}_{{\cal F},{\psi}^{\beta}}(A)=\infty \quad \text{ for }\beta<\alpha,\notag
\end{align}
Hence   we can define the {\it  $({\cal F}, \varphi)$-dimension }as  
$$
{\mbox{\rm Dim}}_{{\cal F},\psi}(A):=\inf\{\alpha \, : \,
{M}_{{\cal F},{\psi}^{\alpha}}(A)=0\}=\sup\{\alpha \, : \, {M}_{{\cal F},{\psi}^{\alpha}}(A)>0\}\,.
$$
If  $\cal{F}$  is    the collection of all open balls in $X$ we will refer to ${\mbox{\rm Dim}}_{{\cal F},\psi}(A)$ as the { \it $\varphi$-spherical dimension}.
If ${\cal F}=\{ F : F\subset X\}$ and  $\varphi(.)=(\diam(.))^{\alpha}$ for some $0\leq \alpha<\infty$, the measure $ {M}_{{\cal F},\varphi}$ is called the $\alpha$-dimensional Hausdorff measure and  we will denote  by ${H}_{({\diam})^{\alpha}}$; moreover, 
${\mbox{\rm Dim}}_{{\cal F},\diam}$ is the usual Hausdorff  dimension, for simplicity we will denote it by ${\mbox{\rm Dim}}$.

\subsection{Topological Markov chains and $\widehat\mu$ dimension. }\label{Cardim}

Let $(\Sigma_A^{\cal I},d,\sigma)$ be a  (one-sided) {topological Markov chain}, see introduction. 
Recall  that,  
 for $w=(w_0,w_1,\ldots )\in \Sigma_A^{\cal I} $, 
  we  denote the $n$-cylinder $C_{w_0w_1\dots w_n}$  by $C(n,w)$. Also we will use 
 $w|_n$ to indicate  the finite sequence $(w_0,w_1,\ldots,w_n)$, and 
$\Sigma_A^{\cal I }|_n:=\{ w|_n \, :\, w\in \Sigma_A^{\cal I }\}$.

A topological Markov chain $\Sigma_A^{\cal I}$ is {\it topologically mixing} if for every $a,b\in{\cal  I}$ there exists $N=N(a,b)$ such that: for all $n\geq N$ there exists $w|_n\in \Sigma_A^{\cal I }|_n$ with $w_0=a$ and $w_n=b$. Along this paper we will assume this property.

A topological Markov chain $\Sigma_A^{\cal I}$ satisfies the {\it big images} (BI) property if there 
exist a finite set of symbols ${\cal I}_0\subset{\cal I}$ such that:
$
\text{for all symbol}\;  i\in{{\cal I}}\; \text{there exists}\; k\in{\cal I}_0\;  \text{such that}\; C_{ik}\neq \varnothing.
$
This condition is called the big images property because 
if $m$ is some finite measure supported in $\Sigma_A^{\cal I}$, then is equivalent to the condition 
$\inf\{m(\sigma(C_i))\, :\,  i\in{\cal I}\}>0
$
 
A topological Markov chain  satisfies the {\it big images and preimages {\rm (BIP)} property} if there exist a finite set of symbols ${\cal I}_0\subset{\cal I}$ such that:  for all symbol $i\in{{\cal I}}$ there exist $k,\ell\in{\cal I}_0$ such that $C_{ki\ell}\neq \varnothing.$

\medskip

\begin{definition}
Given a finite  atomless   measure $\widehat\mu$   in the topological Markov chain  $(\Sigma_A^{\cal I},d)$, we define the $\widehat\mu$-dimension of any set $E\subset \Sigma_A^{\cal I}$ as 
the $({\cal F},\widehat\mu)$-(Carath{\'e}odory)-dimension with 
${\cal F}$ the set of all cylinders in $\Sigma_A^{\cal I}$.
For simplicity we will write ${\mbox{\rm Dim}}_{\widehat\mu}(E) $ instead of ${\mbox{\rm Dim}}_{{\cal F},\widehat\mu}(E)$.
\end{definition}

 Notice that condition 
(\ref{paradefdim}) holds due to the atomless. 
We will use ${\mbox{\rm Dim}}_{\widehat\mu}(E) $  to study the size of sets of zero $\widehat\mu$-measure.

\subsection{Grid dimension vs Hausdorff dimension  in $[0,1]$}

Besides of the usual Hausdorff dimension in the interval $[0,1]$, which we denote by  $\Dim$, we will also consider the {\it grid dimension } defined following Carath{\'e}odory's construction.

\begin{definition} A grid inside $[0,1]$ is a sequence 
 $\Pibf=\{{\cal P}_n\}$ of collections   $ {\cal P}_n$ of subsets of $[0,1]$ such that:
\begin{itemize}
\item[\rm (i)] For all $n$, if $P, P'\in  {\cal P}_n$ with $P\neq P'$, then $P\cap P«=\varnothing$.
\item[\rm (ii)] For all $P_n\in {\cal P}_n$ there exists a unique
$P_{n-1}\in{\cal P}_{n-1}$ such that
$P_n\subset P_{n-1}$
\item[\rm (iii)]
$\displaystyle\sup_{P\in{\cal P}_n} \diam(P) \to 0$ as $n\to\infty.$
\end{itemize}
\end{definition}
A grid define the sets
$$X_{\Pibf}:=\bigcap_n\bigcup_{P\in {\cal P}_n}\mbox{cl}(P)\qquad  \mbox{ and } \qquad  X_{\Pibf}^{1-1}:=\bigcap_n\bigcup_{P\in {\cal P}_n}P .
$$
Given a grid $\Pibf=\{{\cal P}_n\}$,
we consider the collection 
$$
{\cal F}:=\{ \mbox{cl}(P)\cap X_{\Pibf}: P\in {\cal P}_n \mbox{ for some } n\}.
$$
The grid dimension in $X_{\Pibf}$ is defined as the $({\cal F},\diam)$-dimension; we  will  use ${\mbox{\rm Dim}}_{\Pibf}$ to denote it.

\begin{remark}  Notice that  for all $E\subset X_{\Pibf}$, 
$ {\mbox{\rm Dim}}(E) \leq {\mbox{\rm Dim}}_{\Pibf}(E) \leq 1$

\end{remark}

\begin{proposition} \label{parareg1}
Let $E\subset X_{\Pibf}$
such that
\begin{equation}\label{ngrande}
\#\{P\in{\cal P}_0 : \mbox{cl}(P)\cap E\neq\varnothing\}< \infty.
\end{equation}
and let us  suppose that  
there exist $\gamma\geq 0$ and $C>0$
such that  for all $P_n\in {\cal P}_n$ with $ \mbox{cl}(P_n)\cap E\neq \varnothing$
\begin{equation}\label{tau}
\frac{\diam(P_{n-1})^{1+\gamma}}{\diam(P_n)}\leq C\, ,
\quad \mbox{where }P_n\subset P_{n-1}\in {\cal P}_{n-1}.
\end{equation}
Then,  for  $\gamma/(1+\gamma)<\alpha\leq 1$

\begin{equation} \label{Hausmess1}
{M}_{{\cal F},{\diam}^{\alpha}}(E)\leq C_0\, 
H_{{\diam}^{\alpha-\gamma(1-\alpha)}}(E)
\end{equation}
with $C_0$ a positive constant.  Here $H_{{\diam}^{\eta}}$ denotes the $\eta$-dimensional Hausdorff measure.

\noindent In particular,
$$
 \Dim_{\Pibf}(E)\leq  \Dim(E)+\gamma\, ( 1-\Dim_{\Pibf}(E))
$$
\end{proposition}

\begin{remark}\label{comparacioncongrid} Notice that  if $E$ satisfies (\ref{ngrande}), and for all $\gamma>0$ there exists $C$ such that (\ref{tau}) holds, then
 $\Dim_{\Pibf}(E)=  \Dim(E)$.
\end{remark}

\begin{proof}Let  $I$ be an interval in $[0,1]$ with small radius and such that $I\cap E\neq\varnothing$. Then  there exists  a collection   $\{P_i=P(n_i,x_i)\}$   such that $I\cap E\subset \cup_i \mbox{cl}(P_i)\cap X_{\Pibf}$ and 
$$
\diam(\mbox{cl}(P(n_i,x_i))\cap X_{\Pibf})\leq \diam(I) \quad \hbox{ but } \quad \diam(I)<\diam(\mbox{cl}(P(n_i-1,x_i))\cap X_{\Pibf})\, ,
$$
and
$$
\sum_i\diam(\mbox{cl}(P_i)\cap X_{\Pibf})\leq 3\, \diam(I)\, 
$$
Notice that from (\ref{ngrande}) and (\ref{tau}) we have that there exists $K>0$ such that for all for all $P_n\in {\cal P}_n$ with $ \mbox{cl}(P_n)\cap E\neq \varnothing$
$$
\diam(\mbox{cl}(P_n)\cap X_{\Pibf})\geq \left(\frac{1}{C^{1/\gamma}}\right)^{(1+\gamma)^n-1}\diam(\mbox{cl}(P_0)\cap X_{\Pibf})^{(1+\gamma)^n}\geq C^{1/\gamma} K^{(1+\gamma)^n}.
$$
Hence, if $I$ has small radius then by the above inequality we have that the generations $n_i$ are large. Also from (\ref{tau}) we have that
$$
\diam(\mbox{cl}(P_i)\cap X_{\Pibf})=\diam(\mbox{cl}(P(n_i,x_i)\cap X_{\Pibf} ))\geq C'\diam(\mbox{cl}(P(n_i-1,x_i))\cap X_{\Pibf}  )^{1+\gamma}>C''\diam(I)^{1+\gamma}.
$$
Hence, for $0\leq\alpha\leq 1$
$$
\sum_{i}(\diam(\mbox{cl}(P_i)\cap X_{\Pibf}))^{\alpha}=\sum_{i}\frac{\diam(\mbox{cl}(P_i)\cap X_{\Pibf})}{(\diam(\mbox{cl}(P_i)\cap X_{\Pibf}))^{1-\alpha}}\,\leq C\,
\frac{\diam(I) }{\diam(I)^{(1+\gamma)(1-\alpha)}} = C \diam(I)^{\alpha-\gamma(1-\alpha)} \,,
$$
and we get the inequality (\ref{Hausmess1}).
\end{proof}

\subsection{Shift modeled transformations in $[0,1]$, grid and $\widehat\mu$ dimensions }\label{SMT}

Let $\lan$ denote the Lebesgue measure in the interval $[0,1]$ and let $f:[0,1]\longrightarrow [0,1]$ be a map with the property that 
there exists a  finite or numerable family ${\cal P}_0=\{P^0_i\}_{i\in {\cal I}}$ of disjoint
open intervals in $[0,1]$ such that:

\begin{itemize}
\item[(a)] $\lan([0,1]\setminus \cup_j P^0_j)=0$.

\item[(b)] 
If $P^0_i, P^0_j \in {\cal P}_0$ and $f(P^0_i)\cap P^0_j\neq\varnothing$, then $P^0_j\subset f(P^0_i)$.

\item[(c)]  For each $j$, the map $f:P^0_j\longrightarrow f(P^0_j)$ is continuous and injective.

\item[(d)] $\sup_{P\in {\cal P}_{n}} \diam(P)\to 0$ as $n\to\infty$ with 
$$
{\cal P}_{n}=\bigcup_{P^0_i\in{\cal
P}_0}\{(f\big|_{P^0_i})^{-1}(P_j)\, :\, P_j\in{\cal P}_{n-1} \,, \
P_j \subset f(P^0_i)\}=\bigvee_{j=0}^n f^{-j}({\cal P}_0).
$$
\end{itemize}
It is clear  that $\bf\Pi=\{{\cal P}_{n}\}$ is a grid in $[0,1]$, and moreover
from (a) and (b) it follows that $\lan(X_{\bf\Pi})=\lan(X_{\Pibf}^{1-1})=1$. From (c) we have that the elements of ${\cal P}_{n}$ are open intervals.

The  transformation $f$ 
 can be modeled by the left  shift acting on a topological Markov chain: Let ${\cal I}$ be the finite or numerable set indexing the initial partition
${\cal P}_0=\{P^0_i\}_{i\in {\cal I}}$ and let $A=(a_{i,j})$ be the ${\cal I}\times {\cal I}$ matrix with entries $0$ and $1$ defined by
$$
a_{i,j}=
\begin{cases}
1,\, \text{ if } f(P^0_i)\cap P^0_j\neq \varnothing \\
0, \, \text{ otherwise}
\end{cases}
$$

\medskip
\noindent The map $\pi: \Sigma_A^{\cal I} \longrightarrow X_{\Pibf} $ defined by  $\pi((i_0,i_1,\ldots ))=x$, where $x$ is the unique point such that
$$
\{x\}:= \bigcap_{n=0}^{\infty} \text{cl}(f^{-n}(P^0_{i_n}))= \bigcap_{n=0}^{\infty}\text{cl}(P_{i_0}\cap f^{-1}(P^0_{i_1}) \cap \cdots \cap f^{-n}(P^0_{i_n})),
$$
is continuous and $f\circ\pi=\pi\circ\sigma.$
If  $P_{i_0i_1\ldots i_n}$ denote the interval  in ${\cal P}_n$ defined by
$$
P_{i_0i_1\ldots i_n}=\{x\in[0,1] : x\in P^0_{i_0}, f(x)\in P^0_{i_0}, \cdots, f^n(x)\in P^0_{i_n}\},
$$
then
$
\pi(C_{i_0i_1\ldots i_n})=\mbox{cl}(P_{i_0i_1\ldots i_n})\cap X_{\Pibf}$.   Notice also, that the map $\pi$ is injective in $\pi^{-1}(X_{\Pibf}^{1-1})$.
 For any $x\in X_{\Pibf}^{1-1}$ there is a uniquely determinated block sequence $\{P_n\}$ with 
$P_n\in{\cal P}_n$ and $P_{n+1}\subset P_n$ such that  $\cap_n\mbox{cl}(P_n)=\{x\}$, we will denote $P_n$ by $P(n,x)$. If $x\in X_{\Pibf}\setminus X_{\Pibf}^{1-1}$, the sequence $\{P_n\}$ is not uniquely determinated  by $x$,  from now on, for each $x$ we choose $\{P_n\}$ and we denote $P_n$ by $P(n,x)$, i.e if  there exits $w,w'$ such that $\pi(w)=\pi(w')=x$ we choose one of them, say $w=(w_0,w_1,\ldots)$, and define $P(n,x)=P_{w_0w_1\cdots w_{n}}$.

\medskip

 \noindent We will refer to $(\Sigma_A^{\cal I},d,\sigma)$
as the {\it  symbolic representation } of $f$.

\medskip

Let $\widehat\mu$  be a 
finite atomless measure in the topological Markov chain $(\Sigma_A^{\cal I},d)$ such that  $\widehat\mu\circ\pi\asymp \lan $ in  $\mbox{cl}(P^0_i )\cap X_{\Pibf}$ for each interval $P^0_i \in {\cal P}_0$.

 Notice that $X_{\Pibf} \setminus X_{\Pibf}^{1-1}\subset  \bigcup_{n\geq N} {\cal B}_n$, with ${\cal B}_n$ the set of points in $X_{\Pibf}$ belonging to the boundary of some interval of the family ${\cal P}_n$, and also $\#(\pi^{-1}(x))\leq 2$. The following lemma is an easy consequence of these facts and   we will not include the proof.

\begin{lemma} \label{zero0}
Let ${\cal F}_1=\{ \mbox{cl}(P)\cap X_{\Pibf}: P\in {\cal P}_n \mbox{ for some } n\}$, ${\cal F}_2=\{C\subset \Sigma_A^{\cal I} :\, C \mbox{ a cylinder }\}$, and $\alpha>0$. Then 
$$
M_{{\cal F}_1,\diam^{\alpha}}(X_{\Pibf} \setminus X_{\Pibf}^{1-1})=M_{{\cal F}_2,\widehat\mu^{\alpha}}(\Sigma_A^{\cal I} \setminus \pi^{-1}(X_{\Pibf}^{1-1}))=0
$$
\end{lemma}

We have the following relation between  grid and $\widehat\mu$-dimensions.

\begin{lemma}{\label{gridmu}} $
\Dim_{\Pibf}(\pi(\Sigma))=\Dim_{\widehat\mu}(\Sigma) 
\quad $ for all $\quad \Sigma\subset \Sigma_A^{\cal I}$.
\end{lemma}

\begin{proof}
The blocks $P_{i_0i_1\ldots i_n}$  in ${\cal P}_n$ are intervals contained in the initial intervals  of ${\cal P}_0$ and since 
$\widehat\mu\circ\pi\asymp \lan $ in  $\mbox{cl}(P^0_i )\cap X_{\Pibf}$ for each interval $P^0_i \in {\cal P}_0$  we have that
\begin{equation}\label{misma0}
 \diam(P_{i_0i_1\ldots i_n})=\lan(P_{i_0i_1\ldots i_n})=\lan(\mbox{cl}(P_{i_0i_1\ldots i_n})\cap X_{\Pibf})\asymp \widehat\mu\circ\pi(\mbox{cl}(P_{i_0i_1\ldots i_n})\cap X_{\Pibf})
\end{equation}
with constants depending on the initial block $P_{i_0}^0$. We recall that $\lan(X_{\Pibf})=1$.

Let ${\cal F}_1=\{ \mbox{cl}(P)\cap X_{\Pibf}: P\in {\cal P}_n \mbox{ for some } n\}$ and ${\cal F}_2=\{C\subset \Sigma_A^{\cal I} :\, C \mbox{ a cylinder }\}$.  First, notice that from lemma \ref{zero0}, 
since $(\pi(\Sigma)\cap\pi(C_i))\setminus \pi(\Sigma\cap C_i)  \subset X_{\Pibf} \setminus X_{\Pibf}^{1-1}$,
 we have that
\begin{equation}\label{nulo10}
M_{{\cal F}_1,\diam^{\alpha}}(\pi(\Sigma)\cap\pi(C_i))=M_{{\cal F}_1,\diam^{\alpha}}(\pi(\Sigma\cap C_i))
\end{equation}
and 
\begin{equation}\label{nulo20}
M_{{\cal F}_2,\widehat\mu^{\alpha}}(\Sigma\cap C_i)=M_{{\cal F}_2,\widehat\mu^{\alpha}}(\Sigma\cap C_i\cap\pi^{-1}(X_{\Pibf}^{1-1}))
\end{equation}
Next,  we will see that 

\begin{equation}\label{demo0}
M_{{\cal F}_2,\widehat\mu^{\alpha}}(\Sigma\cap C_i)=0 \iff
M_{{\cal F}_1,\diam^{\alpha}}(\pi(\Sigma)\cap \mbox{cl}(P^0_i))=0 
\end{equation}

\noindent ($\Longrightarrow$) 
If $\Sigma\cap C_i\subset \cup_{F_2\in{\cal D}_2}F_2\subset C_i$ with ${\cal D}_2\subset {\cal F}_2$, then 
$$\pi(\Sigma\cap C_i)\subset  \cup_{F_1\in{\cal D}_1}F_1\subset \mbox{cl}(P^0_i))\cap X_{\Pibf}$$ with ${\cal D}_1=\{\pi(F_2): F_2\in{\cal D}_2\}\subset {\cal F}_1$. From (\ref{misma0})
we have that $\diam (F_1)\asymp \widehat\mu(F_2)$ for $F_1=\pi(F_2)$, and therefore if  $M_{{\cal F}_2,\widehat\mu^{\alpha}}(\Sigma\cap C_i)=0$ then $M_{{\cal F}_1,\diam^{\alpha}}(\pi(\Sigma\cap C_i))=0$.
By (\ref{nulo10}) we get $M_{{\cal F}_1,\diam^{\alpha}}(\pi(\Sigma)\cap\mbox{cl}(P^0_i))=0.$

\medskip
\noindent ($\Longleftarrow$) 
If  $\pi(\Sigma)\cap \mbox{cl}(P^0_i) \subset \bigcup_{F_1\in {\cal D}_1} F_1\subset \mbox{cl}(P^0_i))$ with  ${\cal D}_1\subset {\cal F}_1$, then 
$$
\pi^{-1} \left[   \pi(\Sigma)\cap \mbox{cl}(P^0_i)\cap  X_{\Pibf}^{1-1}  \right]   \subset  \bigcup_{F_1\in {\cal D}_1} \pi^{-1}(F_1\cap X_{\Pibf}^{1-1})
$$
But $\pi^{-1} \left[  \mbox{cl}(P_{i_0i_1\cdots i_n})\cap X_{\Pibf}^{1-1}  \right] \subset C_{i_0i_1\cdots i_n}$ and so from (\ref{misma0}) we get that 
$$
\pi^{-1} \left[   \pi(\Sigma)\cap \mbox{cl}(P^0_i)\cap  X_{\Pibf}^{1-1}  \right]   \subset  \bigcup_{F_2\in {\cal D}_2} F_2
\quad \mbox{ with  } {\cal D}_2\subset {\cal F}_2 
$$
and
$$
\sum_{F_1\in {\cal D}_1}\diam^{\alpha}(F_1)=\sum_{F_2\in {\cal D}_2}\widehat\mu^{\alpha}(F_2).
$$
Notice also that 
$$
\Sigma\cap C_i\cap\pi^{-1}(X_{\Pibf}^{1-1})\subset \pi^{-1} \left[   \pi(\Sigma)\cap \mbox{cl}(P^0_i)\cap  X_{\Pibf}^{1-1}  \right] 
$$
Therefore if  $M_{{\cal F}_1,\diam^{\alpha}}(\pi(\Sigma)\cap \mbox{cl}(P^0_i))=0$ then $M_{{\cal F}_2,\widehat\mu^{\alpha}}(\Sigma\cap C_i\cap\pi^{-1}(X_{\Pibf}^{1-1}))=0$, and from (\ref{nulo20}) $M_{{\cal F}_2,\widehat\mu^{\alpha}}(\Sigma\cap C_i)=0$.

\medskip

\noindent Finally, notice that the statement follows from (\ref{demo0}) since 
$$M_{{\cal F}_1,\diam^{\alpha}}(\pi(\Sigma))=0 \iff  M_{{\cal F}_1,\diam^{\alpha}}(\pi(\Sigma)\cap\mbox{cl}(P^0_i))=0  \qquad  \mbox{for all }  P^0_i \in{\cal P}^0
$$ 
$$M_{{\cal F}_2,\widehat\mu^{\alpha}}(\Sigma)=0 \iff M_{{\cal F}_2,\widehat\mu^{\alpha}}(\Sigma\cap C_i)=0  \qquad  \mbox{for all cylinder  }  C_i
$$
\end{proof}

\section{Thermodynamic formalism for countable Markov shifts }\label{TForm}

 Let $\Sigma_A^{\cal I}$ be a is  topologically mixing Markov chain.
The basic notion of the thermodynamic formalism is the {\it topological pressure} which was introduced by Ruelle \cite{Ru1} and Walters \cite{Wa} for continuous transformations acting on a compact space.
 In the case of finite alphabet, for  any continuous function $\phi:\Sigma_A^{\cal I}\longrightarrow\R$
\begin{equation}\label{topsimb}
P_{top}(\phi)=\lim_{n\to\infty} \frac 1n
\log\sum_{w|_n\in \Sigma^{\cal I}_A|_n} \exp
 \left[
 \sup_{z\in C(n,w) }\sum_{j=0}^{n-1}\phi(\sigma^j(z)) 
 \right]
\end{equation}
For topologically mixing countable Markov chains,  O. Sarig   introduced  in \cite{Sa}  the Gurevich pressure for some appropriate potentials. First let us recall some  regularity conditions for the potentials.

A potential $\phi:\Sigma_A^{\cal I}\longrightarrow\R$ is said to have {\it summable variations}  if
$
\sum_{n\ge 2} V_n(\phi) < \infty
$
where the variation $V_n(\phi)$ is defined by
$$
V_n(\phi):=\sup\{|\phi(w)-\phi(w')| :  w=(i_0,i_1,\ldots) ,w'=(j_0,j_1,\ldots)\in \Sigma_A^{\cal I }, \, i_k=j_k, \, 0\le k\le n-1\}.
$$
The potential  $\phi$ satisfies the {\it Walter condition }
 if for every $k\geq 1$, 
$$
\sup_{n\geq 1}[\mbox{\rm V}_{n+k}(\sum_{j=0}^{n-1}\phi\circ \sigma^j)]<\infty \quad \mbox{ and } \quad \sup_{n\geq 1}[\mbox{\rm V}_{n+k}(\sum_{j=0}^{n-1}\phi\circ \sigma^j)]\to 0 \mbox{ as }  k\to\infty ;
$$ 
the  second condition implies the first one when the alphabet is finite.
Notice that  Walter condition  is weaker than summable variations.

 Moreover, we will say that $\phi$ is  {\it weakly   H\"older continuos} iff there exists $A>0$ and $\theta\in(0,1)$ such that for all $n\geq 2$, $V_n(\phi)\leq A\, \theta^n$. Weak H\"older continuity is stronger than summable variation.

\medskip

For  a potential $\phi$ with summable variations 
O. Sarig  defined  the Gurevich pressure, $P_G(\phi)$, as
$$
P_{\bf G}(\phi)=\lim_{n\to\infty}\frac 1n
\log\sum_{\substack{{w=(i_0,i_1,\dots)\in \Sigma_{A}^{\cal I}} \\ i_0=i\,, \sigma^n(w)=w} }\exp
 \left[
\sum_{j=0}^{n-1}\phi(\sigma^j(w))
 \right]
 \qquad \mbox{ with } i\in\cal I
$$
This pressure   is a generalization of the Gurevich entropy $h_G(\sigma)$ \cite{Gu}, since $P_{G}(0)=h_G(\sigma)$.
This limit does not depend on $i$, it is never $-\infty$ and has the following convexity property for potentials $\phi$,$\psi$ of summable variations: $P_G(t\phi+(1-t)\psi)\leq tP_G(\phi)+(1-t)P_G(\psi)$ for all $0\leq t\leq 1$.

\subsection{Weak Gibbs measures}\label{WG}
We  will use the following  definition of a local weak Gibbs measure. 

\begin{definition} \label{mGibbs} A probability  $\widehat\mu$ on $\Sigma_A^{\cal I}$ is called a local weak Gibbs measure for the potential  $\phi:\Sigma_A^{\cal I}\longrightarrow\R$ if there exists a constant $P$ and a sequence $\{K_n\}$ of positive numbers with $1\leq K_n\leq K_{n+1}$ such that $\lim_{n\to\infty}\dfrac{1}{n}\log K_n=0$
and for $\widehat\mu$-a.e. $z\in C_{i_0i_1\ldots i_{n-1}}$
\begin{equation}\label{Gibbs}
\frac{1}{c(i_0)K_n}\leq \frac{\widehat\mu(C_{i_0i_1\ldots i_{n-1}})}{\exp(-nP+\sum_{j=0}^{n-1}\phi\circ\sigma^j(z))}\leq c(i_0) K_n
\end{equation}
with $c(i_0)\geq 1$ a constant depending on $i_0$.  If  the measure $\widehat\mu$ is $\sigma$-invariant, we will say that $\widehat\mu$ is a  local weak $\sigma$-Gibbs measure.

If  $\sup_{i\in {\cal I}}c(i)<\infty$, then $\widehat\mu$ is  called a weak Gibbs measure. If $\sup K_n<\infty$  and {\rm (\ref{Gibbs})} holds for all $z\in C_{i_0i_1\ldots i_{n-1}}$, then  $\widehat\mu$ is called a local Gibbs measure. A local Gibbs measure is  a Gibbs measure if 
 $\sup_{i\in {\cal I} }c(i)<\infty$. 

 \end{definition}
 
 \begin{remark}\label{summablevariation}{ If the alphabet ${\cal I}$ is infinite numerable but 
$\phi$ has summable variation then the constant $P$ is the Gurevich pressure $P_G(\phi)$. Moreover, if $\sum_{n\geq 1} V_n(\phi)<\infty $, then (\ref{Gibbs}) holds for all $z\in C_{i_0i_1\ldots i_{n-1}}$ with an appropriate  new constant $c(i_0)$.} 
 \end{remark}

We will refers later to the following  property  of   a local weak  $\sigma$-Gibbs  measure

\begin{remark}\label{wghitting}
 If $\musymb$ is a  local weak $\sigma$-Gibbs measure, then for  $m>n$ 
$$
\frac{1}{s_{n,m}}\frac{\musymb(\sigma^n(C(m,z)))}{\musymb(\sigma^n(C(n,z)))}
\leq \frac{\musymb(C(m,z))}{\musymb(C(n,z))}\leq s_{n,m}\frac{\musymb(\sigma^n(C(m,z)))}{\musymb(\sigma^n(C(n,z)))}
$$
with 
$\quad s_{n,m}(z):={c(z_0)^2}{c(z_n)^2}K_1 K_{m-n+1} K_{n+1} K_{m+1}\leq {c(z_0)^2}{c(z_n)^2}K_{m+1}^4.
$

\noindent For $m=n-1$ we have
$$
\frac{1}{s_{n,n-1}}\frac{1}{\musymb(\sigma^n(C(n,z)))}
\leq \frac{\musymb(C(n-1,z))}{\musymb(C(n,z))}\leq s_{n,n-1}\frac{1}{\musymb(\sigma^n(C(n,z)))}
$$
with 
$\quad s_{n,n-1}(z):={c(z_0)^2}{c(z_n)}K_1 K_{n} K_{n+1} \leq {c(z_0)^2}{c(z_n)}K_{n+1}^3
$
\end{remark}
\begin{proof} From Definition \ref{mGibbs} we have that for  $\widehat\mu$-a.e. $\widetilde z\in C(m,z)$
\begin{equation}\label{primeraG}
\frac{1}{c(z_0)^2 K_{m+1} K_{n+1}}\leq \frac{\musymb(C(m,z))}{\musymb(C(n,z))}\frac{1}{e^{-(m-n)P}\exp(\sum_{j=n+1}^{m}\phi\circ\sigma^j(\widetilde z))}\leq c(z_0)^2 K_{m+1} K_{n+1}
\end{equation}
and for $\widehat\mu$-a.e. $w\in \sigma^n(C(m,z))$
\begin{equation}\label{segundaG}
\frac{1}{c(z_n)^2 K_{m-n+1} K_{1}}\leq \frac{\musymb(\sigma^n(C(m,z)))}{\musymb(\sigma^n(C(n,z)))}\frac{1}{e^{-(m-n)P}\exp(\sum_{j=1}^{m-n}\phi\circ\sigma^j(w))}\leq c(z_n)^2 K_{m-n+1} K_{1}
\end{equation}
If $W$ denotes the set of points in $\sigma^n(C(m,z))$ verifying (\ref{segundaG}), then $\sigma^{-n}(W)$ is a subset of  $C(m,z)$ with  measure $\widehat\mu(\sigma^{-n}(W))=\widehat\mu(W)>0$. Therefore, there exist  $\widetilde z, w$ with  $\sigma^n(\widetilde z)=w$ verifying the inequalities (\ref{primeraG}) and (\ref{segundaG}). The first result follows directly from these inequalities. The proof for the case $m=n-1$ is similar.
\end{proof}

{
The following equality is called the {\it variational  principle of topological pressure}  
$$
P_{top}(\phi)=\sup_{\mu\in{\cal M}}\left\{ h_{\mu}(\sigma)+\int \phi\, d\mu\right\}
$$
where ${\cal M}$ is the set of $\sigma$-invariant probability measures in $\Sigma_A^{\cal I}$ and $h_{\mu}(\sigma)$ denotes the entropy of $\sigma$ with respect to $\mu$. A measure $\mu\in{\cal M}$ is called an {\it equilibrium measure for} $\phi$  iff 
$$
 h_{\mu}(\sigma)+\int \phi\, d\mu=\sup_{\mu\in{\cal M}}\left\{ h_{\mu}(\sigma)+\int \phi\, d\mu\right\}.
$$
For finite alphabet, classical  results  of Bowen and Ruelle (see \cite{Bo}, \cite{Ru1}) shows that for any $\phi:\Sigma_A^{\cal I}\longrightarrow\R$ with the Walters property,  there is a unique equilibrium measure and this measure is a $\sigma$-Gibbs measure. 
The existence of   equilibrium and  weak Gibbs measures 
has been also studied for non H\"older potentials. See \cite{FFY},  \cite{Hu},  \cite{Yu1}, \cite {Yu2}, \cite{Ke}.

For infinite countable alphabet 
 Sarig obtained a variational  principle for the Gurevich  pressure (for  $\phi$ with summable variations and $\sup \phi<\infty$) and studied the existence   of equilibrium and Gibbs measures. See also Mauldin and Urbanski approach \cite{MU3}. {We remark that in the  non-compact case one can have equilibrium measures which are not Gibbs and also Gibbs measures which are not equilibrium measures. }

Sarig 
proved a  generalized version of Ruelle-Perron-Frobenius theorem (see \cite{Sa},  \cite{Sa2}, \cite{SaNotes}),  involving the modes of recurrence of the potential. In particular,  for 
 $(\Sigma_A^{\cal I},\sigma)$   a topologically mixing Markov chain   and
$\phi:\Sigma_A^{\cal I}\longrightarrow\R$ a potential with summable variations and  finite Gurevich  pressure, he proved that   $\phi$ is positive recurrent iff there are $\lan>0$, a positive continuous function $h$, and a conservative measure $\nu$ which is finite on cylinders, such that $L_{\phi}h=\lan h, L^{*}{\nu}=\lan \nu$, and $\int h\, d\nu=1$. In this case $\lan=\exp{P_G(\phi)}$, and for every cylinder $C$, 
$$
\lan^{-n}L_{\phi}^n (\1_{C})(x)\to h(x)\frac{\nu(C)}{\int h\, d\nu} \quad \mbox{as } n\to\infty.
$$
uniformly in $x$ on compact sets. Here $L_{\phi}(f):=\sum_{\sigma(y)=x}e^{\phi(y)}f(y)$ is the Ruelle operator. We refers to the measure $dm=hd\nu$ as the Ruelle-Perron-Frobenius (RPF) measure.
Moreover, he proved that if $\bf m$ is the RPF measure of a potential $\phi$, which is   positive recurrent with summable variations and such that $P_G(\phi)<\infty$ and $\sup\phi<\infty$, and the entropy of $\bf m$ is finite, then $\bf m$ is an equilibrium measure which is exact (whence ergodic and strong mixing).

We recall that 
 a potential  $\phi:\Sigma_A^{\cal I}\longrightarrow\R$ with summable variations and  finite Gurevich  pressure $P:=P_G(\phi)$ is   {\it positive recurrent} iff  for some (hence all)  $i\in {\cal I}$
$$
\sum_n e^{-nP}Z_n(\phi,i)=\infty\quad  \mbox{ and } \quad \sum_n ne^{-nP}Z^*_n(\phi,i)<\infty \quad \mbox{with}
$$
$$
Z_n(\phi,i)=\sum_{\substack{{w=(i_0,i_1,\dots)\in \Sigma_{A}^{\cal I}} \\ i_0=i,\, \sigma^n(w)=w} } \exp
\left[{\sum_{j=0}^{n-1} \phi\circ\sigma^j(w)} \right],  \quad
Z^*_n(\phi,i)=\sum_{\substack{{w=(i_0,i_1,\dots)\in \Sigma_{A}^{\cal I}} \\ i_0=i\,, i_1,\ldots i_{n-1}\neq i, \, \sigma^n(w)=w} } \exp
\left[{\sum_{j=0}^{n-1} \phi\circ\sigma^j(w)} \right]
$$
Notice that
$$
P_G(\phi)=\lim_{n\to\infty}\frac 1n \log Z_n(\phi,i)
$$
In \cite{Sa3}  Sarig   proved that  if $(\Sigma_A^{\cal I},\sigma)$ is a topologically mixing Markov chain  and
$\phi:\Sigma_A^{\cal I}\longrightarrow\R$ satisfies the Walters condition
 then $\phi$ has a $\sigma$-Gibbs measure $($with $P=P_G(\phi))$ iff $\Sigma_A^{\cal I}$ has the BIP property and $\phi$ has finite Gurevich pressure and V$_1 (\phi)<\infty$. He proved that if $\Sigma_A^{\cal I}$ has the BIP property, then any potential $\phi$ with the Walters property and such that $V_1(\phi)<\infty$ and $P_G(\phi)<\infty$, is positive recurrent. A related result on the existence of Gibbs measures was obtained in \cite{MU2} by Mauldin and Urbanski.

\subsection{Some mixing properties} \label{SMB}

Along this section we will assume that  $\widehat\mu$ is a  local weak $\sigma$-Gibbs measure  with potential $\phi$ such that    $\phi\in L^1(\widehat\mu)$.  In the case of infinite  alphabet  we  will assume that $\sum_{n\geq 1}V_n(\phi)<\infty$ so that remark  \ref{summablevariation} holds.  In next two section to avoid  misunderstanding  we will use $\bf P$ (the bold style) to denote  the  constant (pressure) in the definition of local weak $\sigma$-Gibbs measure.

\begin{definition}Given $\ep>0$ and $M\in\N$ we denote by $Good(M,\ep)$ the set of  points $z\in \Sigma_A^{\cal I}$ such that 
\begin{equation}\label{Ergodic}
e^{-(j+1)[{\bf P}- \int\phi d\,\widehat\mu+\ep]}<\widehat \mu(C(j,z))<  
e^{-(j+1)[{\bf P}- \int\phi d\,\widehat\mu-\ep]}\,, \qquad \mbox{
for all } j\geq M\,.
\end{equation}
\end{definition}
If $\widehat\mu$ is ergodic, as
a consequence of Birkhoff's ergodic theorem  and remark \ref{summablevariation}, we have that \begin{equation}\label{Birk}
\lim_{n\to\infty}\frac1n \log{\widehat\mu(C(n,z))}=-{\bf P}+\int\phi \, d\widehat\mu \quad \mbox{for} \quad  z \quad  \widehat\mu- a.e.
\end{equation}
Hence by Egoroff's theorem we get that

\begin{lemma} \label{egorof}
 If $\widehat\mu$ is ergodic, then
given $\ep>0$ and any $0$-cylinder $P_1$
$$
\widehat\mu(P_1\cap Good(M,\ep))\to \widehat\mu(P_1) \quad  \mbox{  as   }   \quad M\to\infty
$$
\end{lemma}

\begin{proposition} \label{SMsinmalos} Let $P_1,P_2$   be  two $0$-cylinders, and  given $\ep>0$
let ${\cal S}_{N}(M,\ep)$  denote the collection
of cylinders  $C(N,z)$  with $z\in Good(M,\ep)$ verifying
$$
C(N,z) \subset P_1\,, \qquad
\sigma^N(C(N,z))=P_2 \,.
$$
If $\widehat\mu$ is mixing then, for all $M$  large enough (depending on $P_1$ and $\ep$) and $N$ large enough (depending on $P_1$ and $P_2$),
\begin{equation} \label{Sn}
\widehat\mu({\cal S}_{N}(M,\ep)) := \mu \big(\bigcup_{S\in{\cal S}_{N}(M,\ep)} S
\big) \ge \frac 12 \, \widehat \mu (P_1) \, \widehat\mu  (P_2) \,.
\end{equation}
\end{proposition}

\begin{proof} Let ${\cal P}$ denote the set of all $0$-cylinders. We have that  
\begin{align}
\widehat\mu (P_1) & = \widehat\mu ({\cal S}_{N}(M,\ep)) + \sum_{P\in{\cal P}\setminus
\{P_2\}} \sum_{\substack{ C(N,z)\subset P_1 \ s.t. \ z\in\, Good(M,\ep) \\ \sigma^N(C(N,z)) = P}} \widehat\mu(C(N,z)) + \widehat\mu (P_1 \setminus Good(M,\ep)) \notag \\
& \le\widehat\mu ({\cal S}_{N}(M,\ep)) + \sum_{P\in{\cal P}\setminus \{P_2\}}\widehat \mu (P_1 \cap \sigma^{-N} (P)) + \widehat \mu (P_1 \setminus Good(M,\ep)) \notag \\
& = \widehat\mu ({\cal S}_{N}(M,\ep)) +\widehat \mu (P_1) - \widehat\mu (P_1 \cap \sigma^{-N} (P_2))
+ \widehat \mu (P_1 \setminus Good(M,\ep)) \notag \,.
\end{align}
But $\lim_{M\to\infty}\widehat \mu(Good(M,\ep)) = 1$  by lemma
\ref{egorof}  and because $\widehat\mu$ is mixing
$$
\lim_{N\to\infty} \widehat\mu (P_1 \cap \sigma^{-N} (P_2)) = \widehat\mu(P_1) \,
\widehat\mu(P_2) .
$$
 Hence, we get (\ref{Sn}).
\end{proof}

\begin{remark} \label{goodlocal}
\rm 
If $\widehat\mu$ is mixing and there exist $m_0\in \N$ and  $0<\eta< 1$ such that
$$
 \frac{\widehat\mu(P_1\cap Good(m_0,\ep))}{\widehat\mu(P_1)}\geq 1-\eta,
$$
 then for all $0$-cylinder $P_2$  with $\widehat\mu(P_2)>2\eta$ 
 $$
\widehat\mu({\cal S}_{N}(m_0,\ep)) \geq \frac{1}{2} \,  \widehat\mu (P_1) \, \widehat\mu  (P_2) \,.
$$
for all $N$ large enough (depending on $P_1$ and $P_2$).
\end{remark}
\begin{proof} Take in the above proof $N$ large so that
$$
\frac{ \widehat\mu (P_1 \cap \sigma^{-N} (P_2))}{ \widehat\mu(P_1) \,\widehat\mu(P_2)} \geq \frac 12 +\frac{\eta}{\widehat\mu(P_2)}.
$$
\end{proof}

\begin{definition}\label{unifgood}
A collection ${\cal P}$ of  \, $0$-cylinders  is  $\ep$-uniformly $\widehat\mu$-good iff  there exist $m_0$  and $0<\eta<1$ such that 
$$
 \frac{\widehat\mu(P\cap Good(m_0,\ep))}{\widehat\mu(P)}\geq 1-\eta,  \quad \text{ for all } \quad P\in {\cal P}.
$$
\end{definition}
\begin{remark}\label{aproxfinito}\rm
Of course if $\widehat\mu$ is ergodic  any finite collection   is   $\ep$-uniformly $\widehat\mu$-good, since
$\widehat\mu(P\cap \text{Good\,}(M,\ep))\to \widehat\mu(P)$ when $M\to\infty$.
\end{remark}

From remark \ref{goodlocal} follows that:
\begin{corollary} \label{muybuenos}
If $\widehat\mu$ is mixing and ${\cal P}$ is  $\ep$-uniformly $\mu$-good, then  for all $P_1\in{\cal P}$ and all $0$-cylinder $P_2$ with $\mu(P_2)>2\eta$ 
$$
\widehat\mu({\cal S}_{N}(m_0,\ep)) \geq \frac{1}{2} \, \widehat \mu (P_1) \, \widehat\mu  (P_2)
$$
for all  $N$ large enough (depending on $P_1$ and $P_2$).
\end{corollary}

\subsection{Summable uniform rate of  mixing}\label{SUM}

In this section we look for setting up a better mixing property which allow us to get that an infinite countable collection of $0$-cylinders is $\ep$-uniformly $\widehat\mu$-good.

\begin{definition}\label{propgauss}
We will say that $\widehat\mu$ has a summable uniform rate of mixing $\Psi$ (of order $3$) in a collection ${\cal P}$ of $0$-cylinders  iff
there exist a decreasing function $\Psi:\N\longrightarrow \R$  and a natural number $\ell_0$ such that
\begin{itemize}
\item[\rm (1)]  For all $P_1\in  {\cal P}$ and for all pair of $0$-cylinders $P_2, P_3$, we have for $k\geq 0$ 
$$
|\widehat\mu(P_1\cap \sigma^{-k}(P_2)\cap \sigma^{-(k+\ell)}(P_3))-\widehat\mu(P_1\cap \sigma^{-k}(P_2))\,\widehat\mu(P_3)|\leq\Psi(\ell)\,\widehat\mu(P_1\cap \sigma^{-k}(P_2))\, \widehat\mu(P_3)\, \, \mbox{ for   } \ell\geq \ell_0,
$$
and  for $0<\ell<\ell_0$ there exists a positive constant $C$ such that
$$
\widehat\mu(P_1\cap \sigma^{-k}(P_2)\cap \sigma^{-(k+\ell)}(P_3))\leq C\,\widehat \mu(P_1)\, \widehat\mu(P_2\cap \sigma^{-\ell}(P_3))
$$
\item[\rm (2)]$ \sum_{\ell}\Psi(\ell)<\infty$.
\end{itemize}
\end{definition}

\begin{remark}\label{qi}
Condition $(1)$ implies (taking $k=0$ and $P_1=P_2$) that 
$$
|\widehat\mu(P_1\cap \sigma^{-\ell}(P_3))-\widehat\mu(P_1)\widehat\mu(P_3)|\leq\Psi(\ell)\widehat\mu(P_1)\widehat\mu(P_3) \quad \mbox{ for } \ell\geq \ell_0.
$$
Also, by addition over all the of  $0$-cylinders $P_3$ we have that  
$$\widehat\mu(P_1\cap \sigma^{-k}(P_2))\leq C\,\widehat \mu(P_1)\, \widehat\mu(P_2).
$$
\end{remark}

\begin{remark} \label{fracmixing} We recall (see \cite{MU})  that ${\cal P}$ is called continued fraction mixing (with respect to $\widehat\mu$) if  for all  $k$-cylinder $Q$ with  $\sigma^k(Q)\in {\cal P}$  and for all  Borel set  $A$ we have  that 
$$
\widehat\mu(Q\cap \sigma^{-(k+1)}(A))\leq c\, \widehat\mu(Q)\widehat\mu(A)\quad  \mbox { for some constant } c, 
$$
and for $\ell $ large enough, say $\ell\geq \ell_1$,
$$
|\widehat\mu(Q\cap \sigma^{-(k+\ell)}(A))-\widehat\mu(Q)\widehat\mu(A)|\leq\Psi(\ell)\widehat\mu(Q)\widehat\mu(A), \quad \mbox{ with } \Psi(\ell)\to 0 \mbox{  as }\ell\to\infty.
$$
Continued fraction mixing  implies
 condition $(1)$.  If 
 $\Psi(\ell)=\theta^{\ell}$  for some $0<\theta<1$, this property is usually called exponentially continued fraction mixing, and we also have  condition $(2)$. 
\end{remark}

\begin{theorem} \label{mejoregorof}
Let $\widehat\mu$
be a   local weak $\sigma$-Gibbs measure with potential $\phi$ such  that $\phi\in L^2(\widehat\mu)$, and $V_1(\phi)<\infty$.
If  $\widehat\mu$ has summable uniform rate of mixing of order $3$ in a  collection ${\cal P}$ of $0$-cylinders and the constants of localness of the measure $\widehat\mu$ in ${\cal P}$ are upper bounded  then, given $\ep>0$, for all $M$ large enough,
\begin{equation} \label{e_nlocal}
 \frac{\widehat\mu(P\cap Good(M,\ep))}{ \widehat\mu(P)}
\geq  1-\ \frac{c}{\ep^2}\sum_{n\geq M}\frac{1}{n^2}\,,
\quad \hbox{ for all $0$-cylinder $P$ in $\cal P$ }
\end{equation}
with $c$ a positive constant.
\end{theorem}

\begin{corollary}\label{comogauss}
${\cal P}$ is $\ep$-uniformly $\widehat\mu$-good.
\end{corollary}

\begin{proof}
Notice  that 
$$
C_{j}\cap (Good(M,\ep))^c\subset \bigcup_{n\geq M} \left\{z\in C_{j} : \left|-\log\widehat\mu(C(n,z))-(n+1)\left [{\bf P}-\int\phi \,d\widehat\mu\right]\right|\geq (n+1)\ep\right\}
$$
And, 
by the definition   of a local weak $\sigma$-Gibbs measure, we  have for $M$ large  (depending on $\ep$ and $\sup\{c(j) : j\in{\cal J}\}$)  that 
$$
C_{j}\cap (Good(M,\ep))^c\subset \bigcup_{n\geq M} \left\{z\in C_{j} : \left|\sum_{k=0}^{n}\phi\circ\sigma^k(z)-(n+1)\int\phi\, d\widehat\mu\right|\geq (n+1)\ep/2\right\}
$$
Let $Z_k(z)=\phi\circ\sigma^k(z)$.
In order to get (\ref{e_nlocal}) we  prove that for $\ep$ small and all $0$-cylinder $P\in{\cal P}$
\begin{equation}\label{conZ}
\widehat\mu\left(\bigcup_{n\geq M}\{z\in P: |\sum_{k=0}^nZ_k(z)-\sum_{k=0}^nE[Z_k]|>(n+1)\ep\}\right)\leq c\, \frac{\widehat\mu(P)}{\ep^2}\sum_{n\geq M}\frac{1}{n^2}
\end{equation}
The argument of the proof follows the same line that the proof of the strong law of large numbers.
We will prove that for $m\leq n$
\begin{equation}\label{variance}
E\left[\left(\sum_{k=m}^n Z_k-\sum_{k=m}^n E[Z_k]\right)^2\chi_P\right]\leq  c\, \widehat\mu(P)(n-m+1)
\end{equation}
where $\chi_P$ denote the characteristic function on $P$.
Then, by Chebyshev's inequality we have that for $t>0$
$$
\mu\{z\in P : |\sum_{k=m}^n Z_k(z)-\sum_{k=m}^n E[Z_k]|>t\ep\}<
\frac{ E[\left(\sum_{k=m}^n Z_k-\sum_{k=m}^n E[Z_k]\right)^2\chi_P]
}
{t^2\ep^2} \leq c\,  \widehat\mu(P) \frac{(n-m+1)}{t^2\ep^2} 
$$
The above inequality implies
(\ref{conZ}).  Just notice that the set  whose  measure we are  considering in (\ref{conZ})
is  a subset of the union on $n\geq M$ of the sets
$$
A_{n^2}=\{z\in P : |\sum_{k=0}^{n^2}Z_k(z)-\sum_{k=0}^{n^2}E[Z_k]|>(n^2+1)\frac{\ep}{2}\}
$$
and 
$$
\bigcup_{n^2<\ell<(n+1)^2}\{z\in P : |\sum_{k=n^2+1}^{\ell}Z_k(z)-\sum_{k=n^2+1}^{\ell}E[Z_k]|>(n^2+1)\frac{\ep}{2}\}.
$$
By computation we have that
\begin{align}
&E\left[\left(\sum_{k=m}^n Z_k-\sum_{k=m}^n E[Z_k]\right)^2\chi_P\right]=\sum_{k=m}^n\left( E[Z_k^2\chi_P]-2E[Z_k]E[Z_k\chi_P] +(E[Z_k])^2\widehat\mu(P)\right) \notag \\
&+2 \sum_{m\leq k<j\leq n}\left( E[Z_kZ_j\chi_P]-E[Z_k]E[Z_j\chi_P]-E[Z_j]E[Z_k\chi_P]+E[Z_k]E[Z_j]\widehat\mu(P)\right) \notag
\end{align}
Since   
 $\phi\in L^2(\widehat\mu)$ and $V_1(\phi)<\infty$,
 by using Remark \ref{qi} we get that $|E[Z_k\chi_P]|\leq c\, \widehat\mu(P)$ and $E[Z_k^2\chi_P]\leq c\,\widehat\mu(P)$.
Hence,
\begin{equation}\label{unaestimacion}
\sum_{k=m}^n\left( E[Z_k^2\chi_P]-2E[Z_k]E[Z_k\chi_P] + (E[Z_k])^2\widehat\mu(P)\right)\leq  c\,\widehat\mu(P)(n-m+1)
\end{equation}
By using again  that: $\phi\in L^1(\widehat\mu)$, $V_1(\phi)<\infty$, Remark \ref{qi} and the $\sigma$-invariance of $\widehat\mu$,   we obtain for $k<j$ and  $j\geq \ell_0$  that 
$$
E[Z_k]E[Z_j]\widehat\mu(P)-E[Z_k]E[Z_j\chi_P]=E[Z_k]\sum_{C}\left(\widehat\mu(P)\int_{\sigma^{-j}(C)}Z_j d\widehat\mu-\int_{P\cap \sigma^{-j}(C)}Z_j d\widehat\mu \right)\leq c\, \Psi(j)\widehat\mu(P)
$$
where the sum is over all $0$-cylinders $C$. Notice that from  $k<j<\ell_0$ we have 
$$
E[Z_k]E[Z_j]\widehat\mu(P)-E[Z_k]E[Z_j\chi_P]\leq |E[Z_k]|[|E[Z_j]|+|E[Z_j\chi_P]|]\leq c \, \widehat\mu(P)
$$
Therefore by property (2) in  Definition \ref{propgauss} 
\begin{equation}\label{otraestimacion}
\sum_{m\leq k<j\leq n}(E[Z_k]E[Z_j]\mu(P)-E[Z_k]E[Z_j\chi_P])\leq c\, \widehat\mu(P)\sum_{k=m}^n\sum_{j=k+1}^n \Psi(j)\le c\,\widehat\mu(P)(n-m+1)
\end{equation}
In a similar way, but  by using now property (1) in Definition \ref{propgauss}, we get that for $j-k\geq\ell_0$
\begin{align}
&E[Z_kZ_j\chi_P]-E[Z_k\chi_P]E[Z_j]=\notag \\
&=\sum_{C,C^\prime}
\int_{P\cap \sigma^{-k}(C)\cap \sigma^{-j}(C^\prime)}Z_j Z_k \, d\widehat\mu-\left(\sum_{C}
\int_{P\cap \sigma^{-k}(C)}Z_k d\widehat\mu\right)\left( 
\sum_{C^\prime}
\int_{\sigma^{-j}(C^\prime)} Z_j d\widehat\mu\right) \leq c\, \Psi(j-k)\widehat\mu(P) \notag 
\end{align}
where the sum is over  all $0$-cylinders $C, C^\prime$. For $k<j<k+\ell_0$, by using property (1) of Definition \ref{propgauss}, we have that
$$
E[Z_kZ_j\chi_P]-E[Z_k\chi_P]E[Z_j]\leq c\, \widehat\mu(P) E[|Z_0Z_{j-k}|]+|E[Z_k\chi_P]E[Z_j]|\leq c' \, \widehat\mu(P) 
$$
Therefore by property (2) of Definition \ref{propgauss}
\begin{equation}\label{ultima}
\sum_{m\leq k<j\leq n}(E[Z_kZ_j\chi_P]-E[Z_k\chi_P]E[Z_j])\leq c\, \widehat\mu(P)\sum_{k=m}^n\sum_{j=k+1}^n \Psi(j-k)\le c\,\widehat\mu(P)(n-m+1)
\end{equation}
The inequality (\ref{variance}) follows from (\ref{unaestimacion}), (\ref{otraestimacion}) and (\ref{ultima}).
\end{proof}

\section{Pattern subsets in the symbolic space}

We want to define a subset of $\Sigma_A^{\cal I}$ with certain  regular structure. 

\begin{definition} Let $A$ be a set of $m$-cylinders and $B$ be a set of $n$-cylinders with $m>n$. We will say that $A$ is finer than $B$ iff  each $m$-cylinder in $A$ is contained in a $n$-cylinder in $B$. We will write $A<B$.
\end{definition}

\begin{definition}Let us consider two sequences   $\{\widetilde{\cal J}_j\}$ and  $\{{\cal J}_j\}$ with $ \widetilde{\cal J}_j$ a set of $\widetilde{d}_j $-cylinders and  $\{{\cal J}_j\}$ a set of ${d}_j $-cylinders such that  $\widetilde{\cal J}_0={\cal J}_0=\{J_0\}$ and  moreover:
\begin{itemize}
\item[$(i)$] For each $j$,  $d_{j-1}<\widetilde{d}_j  \leq d_j $ and  ${\cal J}_{j}<\widetilde{\cal J}_j<{\cal J}_{j-1}$.
\item[$(ii)$] For each  $\widetilde{J}_j \in \widetilde{\cal J}_j$ there exists a unique ${J}_j \in {\cal J}_j$ such that  $J_j \subseteq \widetilde{J}_j$.
\end{itemize}

We define the set  ${\cal Z}$ by
$$
{\cal Z}= \bigcap_{j=0}^\infty \bigcup_{J_j\in{\cal J}_j} J_j=  \bigcap_{j=0}^\infty \bigcup_{\widetilde{J}_j\in\widetilde{\cal J}_j}\widetilde{ J}_j \;.
$$
We will refer to ${\cal Z}$ as the  set with pattern $(\widetilde{\cal J}_j,{\cal J}_j)$.
\end{definition}

\subsection{The $\widehat\mu$- dimension of pattern sets. A Hungerford lemma.}

Let $\widehat\mu$ be an atomless local weak Gibbs measure for the potential $\phi:\Sigma_A^{\cal I}\longrightarrow\R$. In this section we collect two results  on the $\widehat\mu$-dimension of  a set   ${\cal Z}$  with pattern $({\cal J}_j, \widetilde{\cal J}_j)$.  

{
 We will assume  that  for all $z\in{\cal Z}$ the constants  $c(z_n)$ in Remark \ref{wghitting} have an absolute upper bound in the times $d_j$, more precisely
\begin{equation} \label{condloc2}
\sup\{ c(z_{d_j}): z\in{\cal Z}\, , \, j\in\N\}\leq c<\infty
\end{equation}
Notice also that since ${\cal Z}\subset J_0$, the set ${\cal Z}$ is contained in a unique $0$-cylinder, say $C(0,z)$.
The importance of the above condition is that 
implies the following: consider the sequence  $\{\widehat{s}_{m}=(c(z_0)c)^2 K^4_{m+1}\}$ then
\begin{equation}\label{lims}
 1\leq \widehat{s}_{m}\leq   \widehat{s}_{m+1} \qquad \mbox{and} \qquad 
\lim_{m\to\infty}\frac 1m \log  \widehat{s}_{m}=0
\end{equation}
and 
$$
s_{n,m}(z)\leq  \widehat{s}_{m} \quad \mbox{ for all } z\in{\cal Z} \quad \mbox{ and for all } n=d_j<m
$$
where $s_{n,m}(z)$ is defined in Remark \ref{wghitting}.

\noindent Notice that if $\widehat\mu$ is a weak Gibbs measure (not local) then the condition (\ref{condloc2}) always holds. Also, for  $\widehat\mu$ being local if  for example 
$$
\#\{ \sigma^{d_{j}}(J_{j}) :  J_{j}\in  {\cal J}_j \}<\infty,
$$
then we also have (\ref{condloc2}).
}

\medskip

Hence, from (\ref{condloc2}) and Remark \ref{wghitting}  we have that for all $m>d_j$
\begin{equation}\label{estrella}
\frac{1}{ \widehat{s}_{m} }\, \frac{\musymb(\sigma^{d_j}(C(m,z)))}{\musymb(\sigma^{d_j}(C(d_j,z)))}
\leq \frac{\musymb(C(m,z))}{\musymb(C(d_j,z))}\leq  \widehat{s}_{m}\,  \frac{\musymb(\sigma^{d_j}(C(m,z)))}{\musymb(\sigma^{d_j}(C(d_j,z)))}
\end{equation}

\begin{theorem}\label{azulesyrojosmu}
Let ${\cal Z}$  a set with pattern $({\cal J}_j, \widetilde{\cal J}_j)$ verifying (\ref{condloc2}). 
  Let $\{\alpha_j\}$,
$\{\beta_j\}$,
 $\{\gamma_j\}$,  and $\{\delta_j\}$ be sequences of positive numbers  such that:
\begin{equation} \label{prop1}
e^{-\alpha_j}\leq \frac{{\widehat\mu}(\sigma^{d_{j-1}}(\widetilde J_j))}{{\widehat\mu}(\sigma^{d_{j-1}}({J}_{j-1}))} \leq e^{-\beta_j}\ , 
 \qquad e^{-\gamma_j}\leq \frac{{\widehat\mu}(\sigma^{d_{j-1}}(J_j))}{{\widehat\mu}(\sigma^{d_{j-1}}(\widetilde{J}_{j}))}
\end{equation}
\begin{equation} \label{prop3}
{\widehat\mu}(\sigma^{d_{j-1}}(\widetilde{\cal J}_j \cap J_{j-1})):=\displaystyle\sum_{\widetilde{J}_j\in\widetilde{\cal J}_j, \,
\widetilde {J}_j\subset J_{j-1}} {\widehat\mu}(\sigma^{d_{j-1}}(\widetilde{ J}_j))\ge  \delta_j\, {\widehat\mu} (\sigma^{d_{j-1}}(J_{j-1}))\, ,
\end{equation}
and let assume that for all $\ep>0$ there exists $c>0$ such that
\begin{equation}\label{arreglo}
\widehat s_m \leq  \frac{c} {\widehat\mu(C(m,z))^{\ep}} \qquad \mbox{ for all  } \quad  d_{j-1}< m <\widetilde d_j \quad \mbox{ and }\quad  z\in{\cal Z}.
\end{equation}
Then
$$
\Dim_{\widehat\mu} ({\cal Z}) \ge D^-:=\liminf_{j\to\infty}\frac{(\beta_1+\cdots +\beta_j)-\log[1/(\delta_1\ldots\delta_{j+1})]}{(\alpha_1+\ldots +\alpha_j)+(\gamma_1+\ldots+\gamma_j)+ \log[ \widehat s_{d_1}  \cdots \,  \widehat s_{d_{j-1}}]} 
$$
\end{theorem}

\begin{remark}
Notice that since $\delta_j\leq 1$ and $\beta_j \leq \alpha_j$,  then $D^-\leq 1$.
\end{remark}
\begin{corollary} \label{positivo} If 
\begin{equation}
\liminf_{j\to\infty}\frac{(\alpha_1+\ldots +\alpha_j)+(\gamma_1+\ldots+\gamma_j)}{{\widetilde d}_{j}}>0 \notag
\end{equation}
then 
$$
D^-:=\liminf_{j\to\infty}\frac{(\beta_1+\cdots +\beta_j)-\log[1/(\delta_1\ldots\delta_{j+1})]}{(\alpha_1+\ldots +\alpha_j)+(\gamma_1+\ldots+\gamma_j)} 
$$
\end{corollary}

If $\widehat\mu$ is  a weak Gibbs measure we can eliminate  $\sigma^{d_{j-1}}$ from the conditions  (\ref{prop1}) and   (\ref{prop3}) and the conclusion of theorem  (\ref{azulesyrojosmu}) holds.  Moreover, for  $\widehat\mu$   a Gibbs measure the  we can state the following  corollary. The proofs  are similar  with the obvious changes. 

\begin{corollary}\label{G1} If  $\widehat\mu$  is  a Gibbs measure
$$
 e^{-\alpha_j}\leq \frac{{\widehat\mu}(\widetilde J_j)}{{\widehat\mu}({J}_{j-1})} \leq e^{-\beta_j}\ , 
 \qquad e^{-\gamma_j}\leq \frac{{\widehat\mu}(J_j)}{{\widehat\mu}(\widetilde{J}_{j})},
\qquad
\displaystyle\sum_{\widetilde{J}_j\in\widetilde{\cal J}_j, \,
\widetilde {J}_j\subset J_{j-1}} {\widehat\mu}(\widetilde{ J}_j)\ge  \delta_j\, {\widehat\mu} (J_{j-1})\, ,
$$
and 
\begin{equation}\label{unaG}
\liminf_{j\to\infty}[(\alpha_1+\ldots +\alpha_j)+(\gamma_1+\ldots+\gamma_j)]/{j^{\eta}}>0 \quad  \mbox{ for some } \eta>1.
\end{equation}
Then
$$
\Dim_{\widehat\mu} ({\cal Z}) \ge \liminf_{j\to\infty}\frac{(\beta_1+\cdots +\beta_j)-\log[1/(\delta_1\ldots\delta_{j+1})]}{(\alpha_1+\ldots +\alpha_j)+(\gamma_1+\ldots+\gamma_j)} 
$$
\end{corollary}

\medskip

\begin{remark} Notice that if $\widetilde{\cal J}_j={\cal J}_j$ i.e. $\gamma_j=0$  and $\alpha_j=\alpha$, $\beta_j=\beta$ and $\delta_j=\delta$  for all $j$, then (\ref{unaG}) holds and  therefore  $\Dim_{\widehat\mu} ({\cal Z}) \ge (\beta+\log\delta)/\alpha$.  This case corresponds to the  classical Hungerford lemma.
\end{remark}

\begin{proof} [Proof of theorem \ref{azulesyrojosmu}]

We construct a probability measure $\nu$ supported on  ${\cal Z}$
in the following way: We define $\nu(J_0)=1$ and for each set
$J_j\in {\cal J}_j$ we write
$$
\nu(J_j)=\nu(\widetilde{J}_j) = \frac {{\widehat\mu}(\sigma^{d_{j-1}}(\widetilde{J}_j))}
{
\sum_{\widetilde{J}\in\widetilde{\cal J}_j, \, {J}\subset J_{j-1}}
 {\widehat\mu}(\sigma^{d_{j-1}}(\widetilde{J}))
 }
\, \nu(J_{j-1})
$$
where $J_{j-1}$ and $\widetilde{J}_j$ denote the unique sets in
${\cal J}_{j-1}$ and $\widetilde{\cal J}_{j}$ respectively, such
that $J_j \subseteq \widetilde{J}_j \subset J_{j-1}$. 
As usual, for
any cylinder  $B$, the $\nu$-measure of $B$ is defined by
$$
\nu(B) = \nu(B\cap {\cal Z}) = \inf  \sum_{U\in {\cal U}}  \nu(U) \,
$$
where the infimum is taken over all the coverings $\cal U$ of $B\cap {\cal Z}$
with sets in $\bigcup {\cal J}_j$.

We will show that for all $\Lambda_1 < D^-$
 there exists a positive constant $c$ such that for all $z\in{\cal Z}$ and $m$ large enough,
\begin{equation} \label{frostmangrid}
\nu(C(m,z))\leq c\,(\widehat\mu(C(m,z))^{\Lambda_1}
\end{equation}
and therefore, for all covering ${\cal U}$ we have that $\sum_{U\in{\cal U}}(\widehat\mu(U))^{\Lambda_1} >0$ and we get  the lower bound for the $\widehat\mu$-dimension.

To prove (\ref{frostmangrid}) let us suppose first that $C(m,z)=J_j$ for some $J_j\in{\cal
J}_j$. From property (\ref{prop3}) we have that

\begin{equation} \label{la1}
{\nu(J_j)}\leq    \frac {1}{\delta_j \, \delta_{j-1} \cdots \delta_1} \, \frac {{{\widehat\mu}(\sigma^{d_{j-1}}(\widetilde{J}_j))}\, \cdots \, {\widehat\mu}(\sigma^{d_0}(\widetilde{J}_1))}{{\widehat\mu}(\sigma^{d_{j-1}}(J_{j-1}))\cdots{\widehat\mu}(\sigma^{d_0}(J_0))} 
\end{equation}
Then by  (\ref{prop1})
\begin{equation} \label{la2}
\nu(J_j)\leq \frac{e^{-(\beta_j +\beta_{j-1}+ \cdots +\beta_1)}}{\delta_j \,\delta_{j-1} \cdots \delta_1   }\, .
\end{equation}
Also we have from (\ref{estrella})  and  (\ref{prop1}) that
$$
\frac{\widehat\mu(J_j)}{\widehat\mu(J_{j-1})}\geq \frac{1}{\widehat s_{d_j}}\frac{\widehat\mu(\sigma^{d_{j-1}}(J_j))}{\widehat\mu(\sigma^{d_{j-1}}(J_{j-1}))} \frac{\widehat\mu(\sigma^{d_{j-1}}(\widetilde{J}_j))}{\widehat\mu(\sigma^{d_{j-1}}(\widetilde{J}_{j}))}\geq \frac{1}{\widehat s_{d_j}} \, e^{-(\alpha_j+\gamma_j)}
$$
Wi fix $\ep>0$ so that $\frac{\Lambda_1+\ep}{1-\ep}<D^-$
and  by using (\ref{arreglo})
\begin{equation}\label{cotainf}
\widehat\mu(J_j)^{1-\ep}\geq \frac {e^{-[(\gamma_j +\alpha_j )+(\gamma_{j-1}+\alpha_{j-1})+ \cdots+ (\gamma_1+\alpha_1)]}
}
{ \widehat s_{d_{j-1}}\cdots \,  \widehat s_{d_1}} \, \widehat\mu(J_0)
\end{equation}
In  order to get  (\ref{frostmangrid})  for any $C(m,z)$ with $z\in{\cal Z}$, we will get an stronger condition  for the cylinders $\widetilde{J}_j$ and $J_j$, more precisely,  for $j$ large 
\begin{equation} \label{nuj}
\nu(\widetilde{J}_j)=\nu({J}_j) \le  c \,  {\delta_{j+1}} \widehat\mu(J_j)^{\Lambda_1+\ep} \,
\end{equation}
for some  positive constant $c$.
Since
$0<({\Lambda_1+\ep})/({1-\ep})<D^-$
the inequality (\ref{nuj}) follows from (\ref{la2}) and  (\ref{cotainf}).

Now,  let us suppose that $C(m,z)\not= J_j$ for all $j$ and  for
all $J_j\in{\cal J}_j$, i.e $d_j<m<d_{j+1}$. Since $z\in{\cal Z}$ there exist
$J_j\in{\cal J}_j$ and $J_{j+1}\in{\cal J}_{j+1}$ such that
$J_{j+1}\subset C(m,z)\subset J_j $.

If $C(m,z)\subset \widetilde {J}_{j+1}$, then from the
definition of $\nu$ and (\ref{nuj}) for $\widetilde{J}_{j+1}$ we get
$$
\nu(C(m,z))=\nu(\widetilde{J}_{j+1})=\nu(J_{j+1})\leq c
\,(\widehat\mu(J_{j+1}))^{\Lambda_1}\leq c\, (\widehat\mu(C(m,z)))^{\Lambda_1}\,.
$$
 Otherwise $C(m,z)$ contains sets of the family ${\widetilde{\cal
J}}_{j+1}$, then $d_j<m<\widetilde d_{j+1}$ and we have that
\begin{align} \label{nocantor}
\nu(C(m,z)) & =
\sum_{\substack{{\tilde{J}}_{j+1} \in {\widetilde{\cal J}}_{j+1} \\
{\widetilde{J}}_{j+1} \subseteq C(m,z)}}  \nu (J_{j+1}) \,=
\sum_{\substack{{\tilde{J}}_{j+1} \in {\widetilde{\cal J}}_{j+1} \\
{\widetilde{J}}_{j+1} \subseteq C(m,z)}} \frac{{\widehat\mu}(\sigma^{d_{j}}(\widetilde{J}_{j+1}))}
{{\widehat\mu}(\sigma^{d_{j}}(\widetilde{\cal J}_{j+1} \cap J_{j}))}\,
\nu (J_{j}) \notag \\
& = \frac{\nu (J_{j}) }{{{\widehat\mu}(\sigma^{d_{j}}(\widetilde{\cal J}_{j+1} \cap J_{j}))}}\,\sum_{\substack{{\tilde{J}}_{j+1}
\in {\widetilde{\cal J}}_{j+1} \\
{\widetilde{J}}_{j+1} \subseteq
C(m,z)}}{\widehat\mu}(\sigma^{d_{j}}(\widetilde{J}_{j+1})) \leq \frac{\nu (J_{j})
}{{{\widehat\mu}(\sigma^{d_{j}}(\widetilde{\cal J}_{j+1} \cap J_{j}))}}\ 
{\widehat\mu}(\sigma^{d_{j}}(C(m,z))
 \,.
\end{align}
Using  property (\ref{prop3}) we obtain that
\begin{equation} \label{cota2}
\nu(C(m,z)) \le \frac {\nu(J_j)}{\delta_{j+1}\,{\widehat\mu}(\sigma^{d_{j}}(J_j)} \, {\widehat\mu}(\sigma^{d_{j}}(C(m,z)).
\end{equation}
And by using (\ref{estrella})  and (\ref{arreglo}) 
$$
\nu(C(m,z))  \le \, \frac{ \widehat s_{m} {\nu(J_j)}}{\delta_{j+1}\,{\widehat\mu}(J_j)} \, {\widehat\mu}(C(m,z)) \le \, \frac{ c\, {\nu(J_j)}}{\delta_{j+1}\,{\widehat\mu}(J_j)} \,  {\widehat\mu}(C(m,z))^{1-\ep}
$$
Hence, by (\ref{nuj}) we have that for some $c'>0$
$$
\nu(C(m,z)) \le  \, c' \,  \frac{1}{{\widehat\mu}(J_j)^{1-\Lambda_1-\ep}}\,  {\widehat\mu}(C(m,z))^{1-\ep}
$$
and since 
${\widehat\mu}(J_j)\geq {\widehat\mu}(C(m,z))$  we get (\ref{frostmangrid}).

\end{proof}

The last result in this section is also a kind of Hungerford lemma.  The motivation  is the following:  for some $\cal Z$ we will be able to  construct  (by using thermodynamic formalism)  an alternative measure ${\mhat}$ such that ${\mhat}(\widetilde J_j)\leq c \, \widehat\mu(J_j)^t$, and we can use $\mhat$ (instead of $\widehat\mu$) to equidistribute the mass in the standard measure with support
 in $\cal Z$. The definition of $t$ in the next theorem  is the way of expressing  the above relation between  $\mhat$ and $\widehat\mu$ in the case of $\widehat\mu$ being a local weak Gibbs measure.

\begin{theorem}\label{azulesyrojos3n}
 Let   $\{\gamma_j\}$
 be sequence of positive numbers  such that
\begin{equation}\label{cotas}
\frac{{\widehat\mu}(\sigma^{d_{j-1}}(J_j))}{{\widehat\mu}(\sigma^{d_{j-1}}(\widetilde{J}_{j}))}\geq e^{-\gamma_j}.
\end{equation}
Let us suppose that  there exists a finite  measure $\bf m$ on $\Sigma_A^{\cal I}$  such that  
\begin{equation}\label{muchos}
{\mhat}(\sigma^{d_{j-1}}(\widetilde{\cal J}_j \cap J_{j-1})):=\displaystyle\sum_{\widetilde{J}_j\in\widetilde{\cal J}_j, \,
\widetilde {J}_j\subset J_{j-1}} {\mhat}(\sigma^{d_{j-1}}(\widetilde{ J}_j)) \ge\overline \delta_j\, {\mhat} (\sigma^{d_{j-1}}(J_{j-1}))\, ,
\end{equation}
for some sequence $\{\overline \delta_j\}$ with $0<\overline \delta_j\leq 1$. And moreover,   for some $0< t\leq 1$, if   $C(m,z)\subset J_{j-1}$ with $z\in{\cal Z}$   then 
\begin{equation} \label{prop0f}
\frac{{\mhat}(\sigma^{d_{j-1}}(C(m,z)))}{{\mhat}(\sigma^{d_{j-1}}(J_{j-1}))}\leq c\, \overline \delta_j  \left({\eta_m}\frac{\musymb(\sigma^{d_{j-1}}(C(m,z)))}{\musymb(\sigma^{d_{j-1}}(J_{j-1}))}\right)^t \quad 
 \text {for  }
 d_{j-1}\leq  m\leq\widetilde d_{j}
\end{equation}
 with  $c$ a positive constant and
 \begin{equation} \label{etam}
 \eta_m=
 \begin{cases}
   {e^{-{\overline\gamma}_j}}/{\widehat s_{d_{j}}} \quad \mbox{for } m= \widetilde d_{j} \\
   \medskip
  {1}/{\widehat s_m} \quad \qquad \mbox{for } d_{j-1}\leq m<\widetilde d_{j} 
  \end{cases}
 \end{equation}
Then
$$\Dim_{\widehat\mu} ({\cal Z}) \ge t .
$$
\end{theorem}

If $\widehat\mu$ is weak Gibbs  we can  delete $\sigma^{d_{j-1}}$ from (\ref{cotas}), (\ref{muchos})  and (\ref{prop0f}) and   the conclusion of theorem \ref{azulesyrojos3n} holds. For $\widehat\mu$ Gibbs we have the following corollary. Again the proofs are similar  with the obvious changes.

\begin{corollary} \label{G2}Let $\widehat\mu$ be a Gibbs measure such that
$$
\frac{{\widehat\mu}(J_j)}{{\widehat\mu}(\widetilde{J}_{j})}\geq e^{-\gamma_j}.
$$
Let us suppose that,  there is  a $\sigma$-Gibbs measure
$\mhat$ for the potential $t\phi$, for some $0\leq t\leq 1$,  
$$
\displaystyle\sum_{\widetilde{J}_j\in\widetilde{\cal J}_j, \,
\widetilde {J}_j\subset J_{j-1}} {\mhat}(\widetilde{ J}_j)\ge\overline \delta\, {\mhat} (J_{j-1})\, \quad \mbox{ for some constant } \quad \overline \delta>0,
$$
and
\begin{equation} \label{compintro}
\limsup_{j\to\infty}\frac{\gamma_j}{\widetilde{d_j}-d_{j-1}} <\frac 1t \left[ P_G(t\phi)-tP_G(\phi)\right]
\end{equation}
Then
$$\Dim_{\widehat\mu} ({\cal Z}) \ge t .
$$
\end{corollary}

\begin{remark} Recall that if $\Sigma_A^{\cal I}$ has the BIP property  (in particular, if the alphabet is finite) and the potential $\phi$  satisfies the Walter condition (this condition is weaker than H\" older), has finite Gurevich pressure and $V_1(\phi)<\infty$, then for all $0<t\leq 1$, the RPF measure for $t\phi$  is a $\sigma$-Gibbs measure
\end{remark}

\begin{remark} Notice that from (\ref{compintro}) follows that
$
\mhat(\widetilde{J}_{j}) \leq C\, {\widehat\mu}(J_j)^t .
$
\end{remark}

\begin{proof} [Proof of theorem \ref{azulesyrojos3n}]
We proceed as in the proof of  theorem \ref{azulesyrojosmu} but we define the measure $\nu$ by:
$\nu(J_0)=1$ and for each set
$J_j\in {\cal J}_j$ we write
$$
\nu(J_j)=\nu(\widetilde{J}_j) = \frac {{\mhat}(\sigma^{d_{j-1}}(\widetilde{J}_j))}
{
{\mhat}(\sigma^{d_{j-1}}(\widetilde{\cal J}_j \cap J_{j-1}))
 }
\, \nu(J_{j-1})
$$
where $J_{j-1}$ and $\widetilde{J}_j$ denote the unique sets in
${\cal J}_{j-1}$ and $\widetilde{\cal J}_{j}$ respectively, such
that $J_j \subseteq \widetilde{J}_j \subset J_{j-1}$. 
From (\ref{muchos}) and (\ref{etam}) we get  the following estimate  instead of (\ref{la1}) 
$$
{\nu(J_j)}\leq  \left ( \frac{e^{-{\overline\gamma}_j}}{\widehat s_{ d_{j}}}\frac {{ {\widehat\mu}(\sigma^{d_{j-1}}(\widetilde{J}_j))}}
{{\widehat\mu}(\sigma^{d_{j-1}}(J_{j-1}))}\,   \cdots \,  \frac{e^{-{\overline\gamma}_1}}{\widehat s_{d_{1}}}\frac {{ {\widehat\mu}(\sigma^{d_{0}}(\widetilde{J}_1))}}
{{\widehat\mu}(\sigma^{d_{0}}(J_{0}))} \right )^{t}
$$
and  therefore for some $c'>0$
$$
{\nu(J_j)}\leq   \left (\frac{1} {\widehat s_{ d_{j}}}\frac {{ {\widehat\mu}(\sigma^{d_{j-1}}({J}_j))}}
{{\widehat\mu}(\sigma^{d_{j-1}}(J_{j-1}))}\,   \cdots \,  \frac{1} {\widehat s_{ d_{1}}}\frac {{ {\widehat\mu}(\sigma^{d_{0}}({J}_1))}}
{{\widehat\mu}(\sigma^{d_{0}}(J_{0}))} \right )^{t}\leq 
\left (\frac {{ {\widehat\mu}({J}_j)}}
{{\widehat\mu}(J_{j-1})}\,   \cdots \,  \frac {{ {\widehat\mu}({J}_1)}}
{{\widehat\mu}(J_{0})} \right )^{t}
\leq c' {{\widehat\mu}(({J}_j))}^{t}
$$
Also, instead of (\ref{cota2}) for $d_j<m<{\widetilde{d}_{j+1}}$ we have 
$$
\nu(C(m,z)) \le  \frac{\nu(J_j)}{\overline\delta_{j+1}{\bf m}(\sigma^{d_j}(J_j))} {\bf m}(\sigma^{d_j}(C(m,z)))
$$
Hence by using (\ref{etam})
$$
\nu(C(m,z))\leq c\,
\left ( \frac {1}{\widehat s_m}  \frac {{{\widehat\mu}(\sigma^{d_{j}}(C(m,z)))}}{{\widehat\mu}(\sigma^{d_{j}}(J_{j})))} \right )^{t} \nu(J_j) \leq  \, c^{\prime\prime} \left ( \frac {{{\widehat\mu}(C(m,z))}}{
{\widehat\mu}(J_{j}))} \right )^{t} \widehat\mu(J_j)^{t}
= c^{\prime\prime}\,  {\widehat\mu}(C(m,z))^{t}
$$
The rest of the proof is similar.
\end{proof}

\section{Target-ball set for the shift transformation}\label{balltarget}

Let  $(\Sigma_A^{\cal I},d,\sigma, \musymb)$ be a topologically mixing topological Markov chain endowed with an atomless local weak $\sigma$-Gibbs measure $\musymb$. In the case of  infinite countable alphabet  we will assume that the potential $\phi$ has  summable variations and so the  constant $P<\infty$ is the Gurevich pressure $P_G(\phi)$.
Given  a point $w\in \Sigma_A^{\cal I}$, a $N$-cylinder ${\widehat P}$, and  a sequence $\{\ell_n\}\subset \N$   we want to estimate the size of the set
$$
\Wsimbolico=\{z\in{\widehat P}: \ \sigma^k(z) \in
C(\ell_k,w) \ \hbox{for infinitely many }k\}.
$$
In some subsections some  extra properties on $\phi$ and $\{\ell_n\}$ are required, we will indicate them at the beginning of each subsection.

Recall that  if $\widehat\mu$ is ergodic, $\phi\in L^1(\widehat\mu)$ and $\sum_{n\geq 1}V_n(\phi)<\infty$, then  (see (\ref{Birk})) for $w$  $\widehat\mu$-a.e  
\begin{equation}\label{geq0}
0\geq \lim_{n\to\infty}\frac 1n\log \widehat\mu(C(n,w) )=-P+\int\phi \, d\widehat\mu
\end{equation}
Therefore if  $P-\int\phi \, d\widehat\mu>0$ and  $\underbar v:=\liminf_{n\to\infty} \ell_n/n>0$, given $\ep>0$ small enough  we have that
$$
\sum_n\widehat\mu(C(\ell_n,w) )\leq c \sum_ne^{-n(\underbar v-\ep)[P-\int\phi \, d\widehat\mu-\ep]}<\infty
$$
and so, under these hypothesis,  the  Borel-Cantelli lemma implies that $\Wsimbolico$ has zero $\widehat\mu$-measure  for $w$ $\widehat\mu$ a.e.

\medskip

Later, in order to get a lower bound for the $\widehat\mu$- dimension of $\Wsimbolico$ we  will require the following   property on the center  $w$ of the target.

\begin{definition}
A point $w=(w_0,w_1,\ldots)\in \Sigma_A^{\cal I} $ is a $\musymb$-hitting point  if there exists an increasing sequence ${\cal I}(w)=\{p_j\}\subset \N$  such that
the constants  $c(w_n)$ in Remark \ref{wghitting} have an absolute upper bound in the times $p_j$, more precisely
$$
\sup\{ c(w_{p_j}):  \, j\in\N\}\leq c<\infty
$$
\end{definition}

\begin{remark}\label{R}
 If the alphabet is finite or $\musymb$ is a (not local) weak Gibbs measure, then  any point $w$ is a $\musymb$-hitting point. In any case,   by the recurrence theorem of Poincar\'e  we known that for $w=(w_0,w_1,\ldots)$ $\musymb$-a.e  there exists an increasing sequence $\{p_j\}\subset \N$  such that $w_{p_j}=w_0$ an so 
 the set  of  $\musymb$-hitting point has full $\musymb$-measure. In fact, if that recurrence sequence  exists, then $w$ is a hitting point for any local weak Gibbs measure.
\end{remark}

The importance of the above definition is indicated in the following lemma.
\begin{lemma}\label{chachisymb} Let $C$ be a $0$-cylinder and $w=(w_0,w_1,\ldots)\in \Sigma_A^{\cal I} $ be a $\musymb$-hitting point  with sequence ${\cal I}(w)=\{p_j\}$.  There  exists  a sequence 
$\{\widehat{s}_{m}\}$  with
$$
 1\leq \widehat{s}_{m}\leq   \widehat{s}_{m+1} \qquad \mbox{and} \qquad 
\lim_{m\to\infty}\frac 1m \log  \widehat{s}_{m}=0
$$
such that if  $C(n,z)\subset C$ with $\sigma^n( C(n,z))=C(0,\sigma^{p_{j}}(w))$  for some $j$
then 
$$
\frac{1}{\widehat{s}_{m}}\frac {\musymb (\sigma^{n}(C(m,z))}{\musymb(\sigma^n( C(n,z)))}\le 
 \frac {\musymb(C(m,z))}{\musymb(C(n,z))}\leq \widehat{s}_{m} \frac {\musymb (\sigma^{n}(C(m,z))}{\musymb(\sigma^n( C(n,z)))} \quad \mbox{ for all } m>n
$$
and for $m=n-1$
$$
\frac{1}{\widehat{s}_{n-1}}\frac {1}{\musymb(\sigma^n( C(n,z)))}\le 
 \frac {\musymb(C(n-1,z))}{\musymb(C(n,z))}\leq \widehat{s}_{n-1} \frac {1}{\musymb(\sigma^n( C(n,z)))}
 $$
 Moreover, if $\widehat\mu$ is a local $\sigma$-Gibbs measure (not weak), then the sequence $\{\widehat{s}_{m}\}$  is constant.
\end{lemma}

\begin{proof} Since $C(n,z)\subset C$ we have that $z_0$ is constant (is the symbol of $C$) and  since $\sigma^n( C(n,z))=C(0,\sigma^{p_{j}}(w))$  we have that $z_n=w_{p_j}$. Therefore  the constant $c(z_n)$ in Remark \ref{wghitting}  is equal to $c(w_{p_j})$, and  $c(w_{p_j})\leq c<\infty$ by  definition  of $\musymb$-hitting point.  We can take  $\widehat{s}_{m}=(c(z_0)c)^2 K^4_{m+1}$, then the result follows from Remark \ref{wghitting}.
\end{proof}

\subsection{Estimation of the $\musymb$-dimension}\label{mudim}

First, we will give an upper bound for the $\musymb$-dimension of $W_{\sigma}(\widehat{P},\ell_n,w)$.

\begin{proposition}\label{cotasupformalismo}
Let 
\begin{equation}\label{hipop}
\unders:=\liminf_{n\to\infty} \frac 1n \log \frac 1{\medsymb(C(\ell_n,w))}>0
\end{equation}
Then, 
$$
\Dim_\medsymb \, \Wsimbolico  \le 
\inf\{t>0: \, P_G(t\phi)-P_G(\phi)t<\unders t   \},
$$
\end{proposition}

\begin{proof}
For each $N\in\N$  we have the following covering of $\Wsimbolico$:
$$
\bigcup_{n=N}^\infty \{C(n+\ell_n,z): \ \sigma^n(z)=w, \,z_0=i \}\,.
$$
with $i$  such that $\widehat P\subset C_i$.
Moreover,  given $\ep>0$   and $t>0$, from (\ref{hipop}), remark  \ref{wghitting} and since $\sigma^n(z)=w$ we have for  $n$ large enough
\begin{equation}\label{hit0}
  \frac {\musymb(C(n+\ell_n,z)}{\musymb(C(n,z))}\leq s_{n,n+\ell_n} \frac {\musymb (C(\ell_n,w))}{\musymb(C(0,w))} 
  \leq   \frac{c(i)^2}{c(w_0)^2\musymb(C(0,w))} \, e^{-n({\underline s}-\ep/(2t))}
  \end{equation}
  
Therefore for $t>0$ and $N$ large 
$$
\sum_{n\ge N} \sum_{\sigma^n(z)=w, z_0=i} \medsymb(C(n+\ell_n,z))^t \le c\, 
  \sum_{n\ge N}e^{-n({\underline s}t-\ep/2)}
 \sum_{\sigma^n(z)=w, z_0=i}  \, {\musymb (C(n,z))}^t
  $$
with $c$ a positive constant.
But for $n$ large
$$
{\musymb (C(n,z))}^t\leq  c(i)^{2t}  \,  \exp[-(n+1)(tP_G(\phi)-\ep/4)+t\sum_{j=0}^n \phi\circ\sigma^j(z)]
$$
and so
$$
 \sum_{\sigma^n(z)=w, z_0=i}  \, {\musymb (C(n,z))}^t\leq  c(i)^{2t} e^{-(n+1)(tP_G(\phi)-\ep/4)}  \sum_{\sigma^n(z)=w, z_0=i} e^{\sum_{j=0}^n (t\phi)\circ\sigma^j(z)} 
 $$
 Since  $\Sigma_A^{\cal I}$ is topologically mixing   there exists $u=(u_0,u_1,\cdots u_{k_0})\in\Sigma_A^{\cal I}|_{k_0}$ (for some $k_0$) with $u_0=w_0$ and $u_{k_0}=i$, and therefore by using that $\phi$ has summable variations we have that
 $$
 \sum_{\sigma^n(z)=w, , z_0=i} e^{\sum_{j=0}^n (t\phi)\circ\sigma^j(z)} 
\leq c\,  \sum_{\sigma^{n+k_0}(v)=v, v_0=i} e^{\sum_{j=0}^{n+k_0-1} (t\phi)\circ\sigma^j(v)}=c\,Z_{n+k_0}(t\phi,i)
$$
for some  constant $c>0$. As $P_G(t\phi)<\infty$, for $n$ large, 
$$
 \sum_{\sigma^n(z)=w, z_0=i}  \, {\musymb (C(n,z))}^t\leq c\, c(i)^{2t} e^{-(n+1)(tP_G(\phi)-\ep/4)} e^{(n+k_0)(P_G(t\phi)+\ep/4)}=c'\, e^{n(P_G(t\phi)-tP_G(\phi)+\ep/2)}
$$
Therefore for $N$ large 
$$
\sum_{n\ge N} \sum_{\sigma^n(z)=w, z_0=i} \medsymb(C(n+\ell_n,z))^t \leq c''\, 
  \sum_{n\ge N}e^{-n({\underline s}t-P_G(t\phi)+tP_G(\phi)-\ep)}
$$
for some constant $c''>0$.
Therefore,   for $\ep>0$ and $t>0$ such that $\unders t-P_G(t\phi)+tP_G(\phi)-\ep>0$ we have that
$$
\sum_{n\ge N} \sum_{\sigma^n(z)=w,  z_0=i} \medsymb(C(n+\ell_n,z))^t  \to 0 \, \qquad \mbox{as } \quad N\to\infty
$$ 
and so  $M_{{\cal F},\medsymb^t} (\Wsimbolico)=0$. 
\end{proof}

\begin{remark}\label{s=infty}
In the case $\underline s=\infty$,  the same argument  (with the obvious minor changes)
gives that $ \Dim_\medsymb \, \Wsimbolico=0$, as it is implicitly stated in the proposition. 
\end{remark}

\subsubsection{Lower bound by using some related weak Gibbs measures}

In order to get the lower bound for the $\musymb$-dimension,  we will need  the existence of weak Gibbs measures for some related  new potentials, more precisely  for $t\phi$. 
In this section we will also require that $V_1(\phi)<\infty$ for
 the potential $\phi$ of the measure $\widehat\mu$,  and if $\widehat\mu$ is weak (i.e $\sup_n K_n=\infty$) then   we need the sequence $\{\ell_n\}$ verifies  $\limsup_{n\to\infty}{\ell_n}/{n}<\infty$.

The  following lemma  is a straightforward consequence of the definition of local weak Gibbs measures.

\begin{lemma} \label{cocientem}Let $t>0$ and suppose ${\mhat}$ is  a weak Gibbs measure  for the potential $t\phi$. Then there exist a sequence  $\{ \widehat{k}_{l}\}$ with
$$
 1\leq \widehat{k}_{l}\leq   \widehat{k}_{l+1} \qquad \mbox{and} \qquad 
\lim_{l\to\infty}\frac 1l \log  \widehat{k}_{l}=0
$$
such that 
for any $l$-cylinder $C(l,z)$
$$
\frac{\mhat(C(l,z))}{\mhat(C(0,z))}\leq \widehat c(z_0)^2 \, \widehat{k}_{l} \, e^{-l(P_G(t\phi)-tP_G(\phi))}
\left(\frac{\widehat\mu(C(l,z))}{\widehat\mu(C(0,z))} \right)^t
$$
Here $\widehat c(z_0)=c_{\widehat\mu}(z_0)^{t}c_{\mhat}(z_0)$ where $c_{\widehat\mu}(z_0)$ and $c_{\mhat}(z_0)$ denote   the localness constants of  the measures $\widehat\mu$ and $\mhat$.
\end{lemma}
\begin{proof} Since $\phi$ has summable variation,  the constant $P$ for the local weak Gibbs measure  $\widehat\mu$ is  the Gurevich pressure $P_G(\phi)$.
The potential of ${\mhat}$ has  also  summable variation  since $\phi$ has summable variation. So the constant $P$ for  ${\mhat}$  is   the Gurevich pressure $ P_G(t\phi)$.
Let us denote by $\{K_{n, {\mhat}}\}$ and  $\{K_{n, {\widehat\mu}}\}$ the sequences  in Definition \ref{WG} for  ${\mhat}$ and $\widehat\mu$ respectively. Then for some $x$
$$
\frac{\mhat(C(l,x))}{\mhat(C(0,x))} \leq c_{\mhat}(z_0)^2K_{1, {\mhat}}K_{l+1, {\mhat}} \, e^{-lP_G(t\phi)}  \, e^{t\sum_{j=1}^l \phi\circ\sigma^j(x)}
$$
and for some $z\in C(l,x)$
$$
\left(\frac{\widehat\mu(C(l,z))}{\widehat\mu(C(0,z))} \right)^t\geq \left(\frac{1}{c_{\widehat\mu}(z_0)^2K_{1, {\widehat\mu}}K_{l+1, {\widehat\mu}} }\right)^t\, e^{-ltP_G(\phi)}  \, e^{t\sum_{j=1}^l \phi\circ\sigma^j(z)}
$$
We get the result  by  using that $\sum_{n=1}^{\infty}V_n(\phi)<\infty$ .
\end{proof}

Later we will need the following result, the proof is similar to the previous one.

\begin{lemma} \label{cocientem2}Let $\gamma\geq 0$, $t>0$ and suppose ${\mhat}$ is  a weak Gibbs measure  for the potential $t\phi$. Then there exist a sequence  $\{ \widehat{k}_{l}\}$ with
$$
 1\leq \widehat{k}_{l}\leq   \widehat{k}_{l+1} \qquad \mbox{and} \qquad 
\lim_{l\to\infty}\frac 1l \log  \widehat{k}_{l}=0
$$
such that  
for any $l$-cylinder $C(l,z)$
$$
\frac{\mhat(C(l+1,z))}{\mhat(C(l,z))^{1+\gamma}}\leq \widehat c(z_0)^{2+\gamma}\widehat{k}_{l+1} \, e^{((l+1)\gamma-1)(P_G(t\phi)-tP_G(\phi))}
\left(\frac{\widehat\mu(C(l+1,z))}{\widehat\mu(C(l,z))^{1+\gamma}} \right)^t
$$
\end{lemma}

\

\begin{theorem} \label{cantorsimbolicoformalismo} 

Let $w$ be a $\musymb$-hitting point with sequence ${\cal I}(w)=\{p_j\}$ and let suppose
\begin{equation}\label{lineal}
 \overs:=\limsup_{n\to\infty} \frac 1n \, \log \frac 1{\musymb(C(\ell_n,w))}<\infty,
\end{equation}
there is a positive number $t$ 
such that
\begin{equation}\label{Presionys}
\infty >P_G(t\phi)-tP_G(\phi) \geq t\overs +\ep \qquad  \mbox{ for some } \ep>0
\end{equation}
and moreover, there exists a { mixing } local weak $\sigma$-Gibbs measure ${\mhat}$ for the potential $t\phi$
with local constants such that 
\begin{equation}\label{localm}
\sup\{ c_{\mhat}(w_{p_j}  ) : j\in\N\}<\infty, 
\end{equation}
and  $\phi\in L^1(\mhat)$.

\noindent Then, 
 there exists a  $(\widetilde{\cal J}_j,{\cal J}_j)$ pattern set ${\cal Z}$, with  $\widetilde{\cal J}_j$ a set of $\widetilde{d}_j$-cylinders and ${\cal J}_j$ a set of $d_j$-cylinders, 
 such that
 $$
 {\cal Z}\subset \{ z\in J_0: \sigma^{\widetilde{d}_j}(z)\in C(\ell_{\widetilde{d}_j},w) \mbox{ for } j\in\N\setminus\{0\}\}\subset \Wsimbolico, 
 $$
 and  verifying:
 
\noindent {\rm (i)} 
\begin{equation} \label{prop12f}
e^{-\gamma_j}\leq \frac{{\widehat\mu}(\sigma^{d_{j-1}}(J_j))}{{\widehat\mu}(\sigma^{d_{j-1}}(\widetilde{J}_{j}))}\qquad 
\mbox{ with } \qquad \gamma_j=\widetilde{d}_j (\overs+\ep/(2t)
\end{equation}
{\rm (ii)} There exists $\overline\delta>0 $ such that 
\begin{equation}\label{propp3f}
{\mhat} (\sigma^{d_{j-1}}(\widetilde{\cal J}_j \cap J_{j-1})):=\displaystyle\sum_{\widetilde{J}_j\in\widetilde{\cal J}_j, \,
\widetilde {J}_j\subset J_{j-1}} {\mhat} (\sigma^{d_{j-1}}(\widetilde{ J}_j))\ge \overline \delta \,  {\mhat}(\sigma^{d_{j-1}}(J_{j-1}))\,,
\end{equation}
\noindent {\rm (iii)} For $C(m,z)\subset J_{j-1}$ 
\begin{equation} \label{prop0f}
\frac{{\mhat}(\sigma^{d_{j-1}}(C(m,z)))}{{\mhat}(\sigma^{d_{j-1}}(J_{j-1}))}\leq c\, \overline \delta\left({\eta_m}\frac{\musymb(\sigma^{d_{j-1}}(C(m,z)))}{\musymb(\sigma^{d_{j-1}}(J_{j-1}))}\right)^t \quad 
 \text {for  }
 d_{j-1}\leq  m\leq\widetilde d_{j}
\end{equation}
 with $c>0$ a constant and
 \begin{equation}
 \eta_m=
 \begin{cases}
   {e^{-{\overline\gamma}_j}}/{\widehat s_{d_{j}}} \quad \mbox{for } m= \widetilde d_{j} \\
   \medskip
  {1}/{\widehat s_m} \quad \qquad \mbox{for } d_{j-1}\leq m<\widetilde d_{j} 
  \end{cases}
 \end{equation}
Here $\overline \delta$ is the constant in (ii) and $\widehat s_m$ is given by lemma \ref{chachisymb}.
\end{theorem}

The next corollary follows from above result and  theorem \ref{azulesyrojos3n}

\begin{corollary}\label{cantorsimbf} 
$$
 \Dim_\musymb (\Wsimbolico)  \geq  T^-,
$$ 
with $T^- $ the supremum of the set of $0<t\leq 1$ such that 
$$
 \infty>P_G(t\phi)-P_G(\phi)t>\overline st, 
 $$ and there exists a  mixing  local weak $\sigma$-Gibbs measure ${\mhat}$ for the potential $t\phi$
with local constants verifying 
(\ref{localm}), and  $\phi\in L^1(\mhat)$.
\end{corollary}

\begin{remark}\label{cotasbien} Notice that the convexity of the Gurevich pressure (see section \ref{TForm}) imply the convexity of the function $G:[0,1]\longrightarrow [0,+\infty]$ defined as $G(t):=P_G(t\phi)-P_G(\phi)t$, and since $G(1)=0$ we have that $G(t)$ is decreasing.
Hence, we have that for $\underline s>0$ 
$$
\sup\{t>0: \infty>G(t)>\overline s t\}\leq \inf\{t>0: G(t)<\underline st\}
$$
We recall that $ \Dim_\musymb (\Wsimbolico)  \leq  \inf\{t>0: G(t)<\underline st\}$.
\end{remark}

Later in section \ref{MarkovT} we will use the following remark to get our Hausdorff dimension results for Markov transformations with infinite countable alphabet. This result is not necessary in the study of the $\widehat\mu$-dimension of the target-ball set for the shift transformation. The  proof is given in section \ref{remarks}.

\begin{remark} \label{paradim} 
Under the hypothesis of theorem  \ref{cantorsimbolicoformalismo} let us also 
 suppose  that   
  the center of the target  $w$ verifies
$$
 \unders=\liminf_{n\to\infty}\frac 1n\log\frac{1}{ \widehat\mu(C(\ell_n,w))}>0 \quad \mbox{ and } \quad \liminf_{n\to\infty} \frac{\widehat\mu(C(n+1,w))}{\widehat\mu(C(n,w))}>0.
 $$
If $P_G(\phi)-\int\phi \, d\widehat{\bf m}>0$ and  the collection ${\cal P}=\{C(0,\sigma^{p_j}(w)): p_j\in {\cal I}(w)\}$  is $\ep$-uniformly $\widehat{\bf m}$-good  for all $\ep$ small enough, then 
there  exists  a $(\widetilde{\cal J}_j,{\cal J}_j)$ pattern set ${\cal Z}_{\ep}\subset \Wsimbolico$ verifying the conditions in theorem \ref{cantorsimbolicoformalismo} and with the following extra property: 

\noindent For all
$$
\gamma\geq \frac{3\ep}{tP_G(\phi)-t\int\phi \, d\widehat{\bf m}-\ep}>0
$$
there is a constant $c>0$ such that for all $z\in {\cal Z}$ 
$$
 \frac{\widehat\mu(\sigma^{\widetilde d_j}(C(n+1,z))}{\widehat\mu(\sigma^{\widetilde d_j}(C(n,z))^{1+\gamma}}\geq c \qquad \mbox{ for } \quad d_j\leq n<\widetilde d_{j+1} 
$$
By construction (also in  theorem \ref{cantorsimbolicoformalismo}) we have that 
$$ \sigma^{\widetilde d_{j+1}}(C(n,z))=C(n-\widetilde d_{j+1},w) \quad \mbox{for } \quad \widetilde d_{j+1}\leq n\leq d_{j+1}$$

Recall that if  the collection $\#{\cal P}$ is finite  then we always have that ${\cal P}$ is $\ep$-uniformly $\widehat{\bf m}$-good  (since $\widehat{\bf m}$ is mixing). For $\#{\cal P}$ infinite we know  by corollary  \ref{mejoregorof} that if $\widehat{\bf m}$  has summable uniform rate of mixing of order $3$ in ${\cal P}$ and $\phi\in L^2(\widehat{\bf m})$ then  ${\cal P}$ is $\ep$-uniformly $\widehat{\bf m}$-good.
\end{remark}

\

\begin{proof}[Proof of theorem \ref{cantorsimbolicoformalismo}]

Since $\Sigma_A^{\cal I}$  is topologically mixing,  for  $\widetilde{d}_0$ large enough $\sigma^{-\widetilde{d}_0}(C(0,w))\cap \widehat{P}\neq\varnothing$. Therefore,
there exists a $\widetilde{d}_0$-cylinder $\widetilde{J}_0\subset \widehat{P}$ such that  $\sigma^{\widetilde{d}_0}(\widetilde{J}_0 )=C(0,w)$. We  define $J_0:=\widetilde J_0$ and $\widetilde{\cal J}_0=\{\widetilde J_0\}$, ${\cal J}_0=\{J_0\}$.

To construct the family  $\widetilde{\cal J}_1$ we will apply proposition \ref{SMsinmalos} for the measure $\mhat$ with  $P_1=P_2=C(0,w)$.  We take $N_1$ be a  natural number large enough so that ${\cal S}_{N}(M,\ep)$ with $N=M=N_1$ satisfies  (\ref{Sn}). Hereafter, for all $n$-cylinder $C$, we will denote by 
 $\inv_{C}$   the  composition of the  $n$ branches of $\sigma^{-1}$ such that $C=\inv_{C}(\sigma^n(C))$.
We define $\widetilde{\cal J}_1$ as the family of cylinders  $\inv_{J_0}(S)$
with $S\in{\cal S}_{N_1}(N_1,\ep)$. 
Notice that  by proposition \ref{SMsinmalos}
$$
\mhat(\sigma^{d_0}(\widetilde{\cal {J}}_1\cap J_0)):=
\sum_{\widetilde{J}_1\in\widetilde{\cal J}_1, \,
\widetilde {J}_1\subset J_{0}} \mhat (\sigma^{d_{0}}(\widetilde{ J}_1))\ge \overline \delta  \, \mhat(\sigma^{d_{0}}(J_{0}) \quad \mbox{ with } \quad  \overline \delta =\frac 12\, \mhat(C(0,w))\,.
$$
Let  ${\cal I}(w)$ be  the sequence given by  the definition of $\widehat\mu$-hitting point. We choose $N_1$   so that there exists $p_i\in {\cal I}(w)$  for  $\widetilde{d}_1:=d_0+N_1$ such that
\begin{equation}\label{ell}
\ell_{\widetilde{d}_1} \le p_i <\ell_{\widetilde{d}_1+1}.
\end{equation}
We denote this natural number $p_i$  by $p(\widetilde{d}_1)$.

For  each set  $S$ in ${\cal S}_{N_1}(N_1,\ep)$ we have that $\sigma^{N_1}(S)=C(0,w)$, then we    take in each $S$ the subset $L:=\inv_{S}(C(p(\widetilde{d}_1),w))$ and we denote by ${\cal L}_1$ this family of sets. We define
the family ${\cal J}_1$   as  the collection  $\inv_{J_0}(L)$ with $L\in{\cal L}_1$.

Notice that by construction 
 for all $ J_1\in{ \cal J}_1$ there exists an unique $\widetilde J_1\in\widetilde{ \cal J}_1$ such that $J_1\subset \widetilde J_1$, 
 $$
 \sigma^{\widetilde{d}_1}(\widetilde J_1)=C(0,w) ,\quad  \sigma^{\widetilde{d}_1}(J_1)=C(p(\widetilde{d}_1),w) \subset C(\ell_{\widetilde{d}_1},w) , \quad\sigma^{d_1}(J_1)=C(0,\sigma^{p(\widetilde{d}_1)}(w)),
$$
with  $\widetilde{d}_1:=d_0+N_1$ and $d_1:=\widetilde{d}_1+p(\widetilde{d}_1).$

Since $\sigma^{d_0}(\widetilde J_1)=S$ for some $S\subset C(0,w)$ with  $\sigma^{N_1}(S)=C(0,w)$, and  $\sigma^{d_0}(J_1)=L\subset S$ for some $L$ with $\sigma^{N_1}(L)=C(p(\widetilde{d}_1),w)$, then we have from  lemma \ref{chachisymb} (we may assume that $0\in {\cal I}(w)$) and (\ref{ell}) that 
\begin{equation}\label{hitprop0}
\frac{\musymb(\sigma^{d_0}(J_1))}{\musymb(\sigma^{d_0}(\widetilde J_1))}= \frac{\musymb(L)}{\musymb(S)}\geq  \frac{1}{\widehat{s}_{N_1+p(\widetilde{d}_1)}} \frac{\musymb(C(p(\widetilde{d}_1),w))}{\musymb(C(0,w))} \geq \frac{1}{\widehat{s}_{N_1+\ell_{\widetilde{d}_1+1}}} \, {\musymb(C(\ell_{\widetilde{d}_1+1},w))}
\end{equation}
From (\ref{lineal}) we have for $N_1$ large enough that 
$$
\musymb(C(\ell_{\widetilde{d}_1+1},w))\ge \, e^{-\widetilde{d}_1 (\overs+\ep/t)}
$$
Recall that $\lim_{m\to\infty}(\log \widehat s_m)/m=0$ with  $\{\widehat s_m\}$ a constant sequence  if $\widehat\mu$ is not weak (i.e $\sup_nK_n<\infty$), and if $\widehat\mu$ is weak  (i.e $\sup_nK_n=\infty$) we are assuming that
$\limsup_{n\to\infty}{\ell_n}/{n}<\infty$.  So, 
 for $N_1$ large enough we also have 
$$
\frac{1}{\widehat{s}_{N_1+\ell_{\widetilde{d}_1+1}}} \geq e^{-\ep\widetilde{d}_1/(2t)}
$$
Hence
$$
\frac{\musymb(\sigma^{d_0}(J_1))}{\musymb(\sigma^{d_0}(\widetilde J_1))} \geq e^{-\gamma_1}
\qquad
\mbox{with } \qquad \gamma_1= \widetilde{d}_1 (\overs+\ep/(2t))
$$
From lemma \ref{cocientem}
 for $d_0\leq m\leq \widetilde{d}_1$ 
$$
\frac{{\mhat}(\sigma^{d_{0}}(C(m,z)))}{{\mhat}(\sigma^{d_{0}}(J_{0}))}\leq 
 \widehat c(w_0)^2 \, \widehat{k}_{m-d_0} \, e^{-(m-d_0)[P_G(t\phi)-tP_G(\phi)]}\left(\frac{\widehat\mu(\sigma^{d_{0}}(C(m,z)))}{\widehat\mu(\sigma^{d_{0}}(J_0))} \right)^t \notag
$$
By taking $d_0$ large enough  we have for $ m\geq d_0$ 
that
$$
 \widehat c(w_0)^2 \, \widehat{k}_{m}\, {\widehat s_m^t}\leq {\overline \delta}e^{m\ep/2}
$$
 and therefore (we recall that $\widehat{k}_{m-d_0}\leq k_m$) by using (\ref{Presionys}) we get for some $c>0$
 $$
 \begin{aligned}
\frac{{\mhat}(\sigma^{d_{0}}(C(m,z)))}{{\mhat}(\sigma^{d_{0}}(J_{0}))}&\leq  e^{d_0[P_G(t\phi)-tP_G(\phi)]}\
e^{-m[P_G(t\phi)-tP_G(\phi)-\ep/2]}\, {\overline \delta}
\left(\frac{1}{\widehat s_m} \frac{\widehat\mu(\sigma^{d_{0}}(C(m,z)))}{\widehat\mu(\sigma^{d_{0}}(J_0))} \right)^t \notag \\
&\leq c\, {\overline \delta}e^{-m[t\overline s+\ep/2]}\left(\frac{1}{\widehat s_m} \frac{\widehat\mu(\sigma^{d_{0}}(C(m,z)))}{\widehat\mu(\sigma^{d_{0}}(J_0))} \right)^t 
\end{aligned}
$$
Hence, since $ \gamma_1= \widetilde{d}_1 (\overs+\ep/(2t))$ we obtain (\ref{prop0f}) for $j=1$.
 
 \
 
Now, let us assume that we have already constructed the families
$\widetilde{\cal J}_j, \, {\cal J}_j$ 
and the numbers   $N_j$,  $\widetilde{d}_j$,
$p(\widetilde{d}_j)$  and  $d_j$ (for $j=1,\ldots,m$) , with $p(\widetilde{d}_j)\in {\cal I}(w)$, related by
$$
d_0=\widetilde{d}_0, \qquad \widetilde{d}_{j}=d_{j-1}+N_{j}
 \qquad   d_j=\widetilde{d}_j+p(\widetilde{d}_j)
$$
in such way that 
\begin{equation}\label{enW}
 \sigma^{\widetilde{d}_j}(\widetilde J_j)=C(0,w) ,\quad  \sigma^{\widetilde{d}_j}(J_j)=C(p(\widetilde{d}_j),w) \subset C(\ell_{\widetilde{d}_j},w) , \quad\sigma^{d_j}(J_j)=C(0,\sigma^{p(\widetilde{d}_j)}(w)),
\end{equation}
and  the properties (i),(ii) and (iii)  of  the theorem holds. 

To construct $\widetilde{\cal J}_{m+1}$ we choose  $N_{m+1}$ such that:  
\begin{itemize}
\item[(i)] proposition \ref{SMsinmalos} holds for the measure $\mhat$ with $N=M=N_{m+1}$, $P_1=C(0,\sigma^{{p(\widetilde{d}_m)}}(w))$ and $P_2=C(0,w)$,
 \item[(ii)] There exists
$p(\widetilde{d}_{m+1})\in {\cal I}(w)$  with   $\widetilde{d}_{m+1}:=d_m+N_{m+1}$ such that
$
\ell_{\widetilde{d}_{m+1}} \le p(\widetilde{d}_{m+1}) <\ell_{\widetilde{d}_{m+1}+1},
$
 \end{itemize}
 We define  $\widetilde{\cal J}_{m+1}$ as
$$
\widetilde{\cal J}_{m+1} = \bigcup_{J\in {\cal J}_m} \inv_J ({\cal S}_{N_{m+1}}(N_{m+1},\ep)) \,.
$$
and the family ${\cal J}_{m+1}$ as
$$
{\cal J}_{m+1} = \bigcup_{J\in {\cal J}_m} G_J^{-1} ({\cal L}_{m+1}) \, \quad  \text{where} \qquad {\cal L}_{m+1}= \bigcup_{S\in{\cal S}_{N_{m+1}}(N_{m+1},\ep)} G^{-1}_{S}(C(p(\widetilde{d}_{m+1}),w))
$$
In a similar way that in the initial step of the induction, one can  check that by taking $N_{m+1}$ large enough the properties (i),(ii) and (iii)  of  the theorem hold for $j=m+1$. Notice that we need $p(\widetilde{d}_j)\in {\cal I}(w)$ in order to get (\ref{hitprop0}) for $j+1$, and (\ref{localm}) to get (iii). The  pattern set ${\cal Z}$ is contained in $\Wsimbolico $ for (\ref{enW}).
\end{proof}

\subsubsection{Lower bound without using other weak Gibbs measures}

In this section we assume that   $\widehat\mu$ is a {\it  mixing } local  weak $\sigma$-Gibbs measure  with potential $\phi$ such that $\sum_{n\geq 1}V_n(\phi)<\infty$ and  $\phi\in L^1(\widehat\mu)$.  In particular, we have $0\leq P_G(\phi)-\int \phi \, d
{\widehat{\mu} }<\infty$ (see (\ref{geq0})).
If $\widehat\mu$ is really  weak (i.e $\sup_n K_n=\infty$) then   we ask  the sequence $\{\ell_n\}$ verifies  $\limsup_{n\to\infty}{\ell_n}/{n}<\infty$, this condition is not necessary if $\sup_n K_n<\infty$.
Moreover,     in the case of   infinite countable  alphabet and $\musymb$ weak  we will add the extra hypothesis $\sup\phi<\infty$ in order to have  for all $z$ with $z_0=i_0$ 
$$
{{\widehat\mu(C(n,z)}}\leq c(i_0) K_n  \exp[-nP_G(\phi)+n\sup\phi] ,
$$
and therefore for all $\ep>0$ there exists $n_0$ such that for all  $n\geq n_0$
\begin{equation}\label{controlsm}
\frac{1}
{{\widehat\mu(C(n,z))}}\geq \frac{1}{c(i_0)} 
\exp[n(P_G(\phi)-\sup\phi-\ep)]   \quad \mbox{ for all $z$ with } z_0=i_0
\end{equation}
By using the mixing properties of $\widehat\mu$ we get the following:

\begin{theorem} \label{cantorsimbolico} 
Let $w$ be a $\musymb$-hitting point with sequence ${\cal I}(w)=\{p_i\}$ and let 
\begin{equation}\label{lineal2}
 \overs:=\limsup_{n\to\infty} \frac 1n \, \log \frac 1{\musymb(C(\ell_n,w))}<\infty.
\end{equation}

\noindent Then,  for all $\ep>0$ with 
$$
P_G(\phi)-\int \phi \, d
{\widehat{\mu} }>\ep
$$  
there exists a  $(\widetilde{\cal J}_j,{\cal J}_j)$ pattern set ${\cal Z}_{\ep}$, with  $\widetilde{\cal J}_j$ a set of $\widetilde{d}_j$-cylinders and ${\cal J}_j$ a set of $d_j$-cylinders, 
 such that
 $$
 {\cal Z}_{\ep}\subset \{ z\in J_0: \sigma^{\widetilde{d}_j}(z)\in C(\ell_{\widetilde{d}_j},w) \mbox{ for } j\in\N\setminus\{0\}\}\subset \Wsimbolico, 
 $$
 and  verifying:

\medskip
\noindent {\rm (i)} There exists $\overline\delta>0 $ such that 
$$
{\widehat\mu} (\sigma^{d_{j-1}}(\widetilde{\cal J}_j \cap J_{j-1})):=\displaystyle\sum_{\widetilde{J}_j\in\widetilde{\cal J}_j, \,
\widetilde {J}_j\subset J_{j-1}} {\widehat\mu} (\sigma^{d_{j-1}}(\widetilde{ J}_j))\ge \overline \delta \,  {\widehat\mu}(\sigma^{d_{j-1}}(J_{j-1}))\,,
$$
\noindent {\rm (ii)} 
$$
e^{-(\widetilde{d}_j-d_{j-1})(P_G(\phi)-\int \phi \, d
{\widehat{\mu} }+2\ep)}\leq \frac{{\widehat\mu}(\sigma^{d_{j-1}}({\widetilde J}_j))}{{\widehat\mu}(\sigma^{d_{j-1}}({J}_{j-1}))}\leq e^{-(\widetilde{d}_j-d_{j-1})(P_G(\phi)-\int \phi \, d
{\widehat{\mu} }-2\ep)}
$$
\medskip
\noindent {\rm (iii)} 
$$
e^{-\ep \widetilde  d_j}\, \musymb(C(d_j-\widetilde{d}_j,w))\leq
\frac{\musymb(\sigma^{d_{j-1}}(J_j))} 
{\musymb(\sigma^{d_{j-1}}(\widetilde J_j))}
\leq e^{\ep \widetilde d_j}\, \musymb(C(d_j-\widetilde{d}_j,w))
$$
and
$$
{\musymb(C(d_j-\widetilde{d}_j,w))} \geq \, e^{-\widetilde{d}_j (\overs+\ep)}
$$
\medskip
\noindent {\rm (iv)} For some $c>0$
$$
{\widehat s}_m \leq  \frac{c} {\widehat\mu(C(m,z))^{\ep}} \qquad \mbox{ for all  } \quad  m\geq d_{0}\quad \mbox{ and }\quad  z\in{\cal Z}
$$
with $\{{\widehat s}_m\}$ the sequence in lemma \ref{chachisymb}

\medskip
\noindent {\rm (v)} 
$$
\lim_{j\to\infty}\frac{d_1+\cdots+d_{j-1}}{\widetilde { d}_{j}}=0
$$
\end{theorem}

\begin{corollary}\label{dimpicantorsimbolico} 
$$
 \Dim_{\musymb} ({\cal Z}_\ep) \ge \frac
{P_G(\phi)-\int \phi \, d
{\widehat{\mu} } -2\ep}{P_G(\phi)-\int \phi \, d
{\widehat{\mu} }+\overs+4\ep}
$$
\end{corollary}
\begin{proof} Follows from last theorem  and corollary \ref{azulesyrojosmu} with $\delta_j=\overline\delta$, $\alpha_j=({\widetilde d}_j-d_{j-1})(P_G(\phi)-\int \phi \, d
{\widehat{\mu} }+2\ep)$,
 $\beta_j=({\widetilde d}_j-d_{j-1})(P_G(\phi)-\int \phi \, d
{\widehat{\mu} }-2\ep)$ and $\gamma_j=\widetilde d_j(\overline s+\ep)+\ep\widetilde d_j$. 
Notice that, if $\widehat P\subset C_i$, then 
$$
\lim_{j\to\infty}\frac{(\alpha_1+\cdots+\alpha_j)+(\gamma_1+\cdots+\gamma_j)}{\widetilde d_j}= P_G(\phi)-\int \phi \, d
{\widehat{\mu} }+\overs +4\ep
$$
and 
$$
\lim_{j\to\infty}\frac{\beta_1+\cdots+\beta_j-(j+1)\log 1/\delta}{\widetilde d_j}= P_G(\phi)-\int \phi \, d
{\widehat{\mu} }-2\ep.
$$
\end{proof}

By letting $\ep\to 0$ in the above corollary  we obtain

\begin{corollary}\label{DimPisimbolicosinformalismo}
$$
\Dim_{\musymb} (\Wsimbolico)\ge \frac {P_G(\phi)-\int \phi \, d
{\widehat{\mu} }}{P_G(\phi)-\int \phi \, d
{\widehat{\mu} }+\overs} \,.
$$
\end{corollary}

\begin{remark} If the variational principle holds  then we have that $ P_G(\phi)-\int \phi \, d
{\widehat{\mu} }\geq h_\musymb$ and therefore
$$
\Dim_{\musymb} (\Wsimbolico)\ge \frac {P_G(\phi)-\int \phi \, d
{\widehat{\mu} }}{P_G(\phi)-\int \phi \, d
{\widehat{\mu} }+\overs} \ge  \frac {h_\musymb}{h_\musymb+\overs}\,.
$$
and if $\widehat{\mu}$ is an equilibrium measure $ P_G(\phi)-\int \phi \, d
{\widehat{\mu} }= h_\musymb$.  
\end{remark}

We will use  in section \ref{MarkovT}  the following remark to get Hausdorff dimension  results  for Markov transformations with infinite countable alphabet. His  proof is given in section \ref{remarks}.

\begin{remark} \label{paradimsin}
Under the hypothesis of theorem  \ref{cantorsimbolico}, 
if $P_G(\phi)-\int\phi \, d\widehat\mu>0$ and the collection 
 ${\cal P}=\{C(0,\sigma^{p_j}(w)): p_j\in {\cal I}(w)\}$ is $\ep$-uniformly $\widehat{\mu}$-good for all $\ep>0$ small enough, then  there  exists  a $(\widetilde{\cal J}_j,{\cal J}_j)$ pattern set ${\cal Z}_{\ep}\subset \Wsimbolico$ verifying the conditions in theorem \ref{cantorsimbolico} and with the following extra property: 

\noindent For all 
$$
\gamma\geq \frac{2\ep}{P_G(\phi)-\int\phi \, d\widehat{\mu}-\ep}>0
$$
there is a constant $c>0$ such that for all $z\in {\cal Z}$ 
$$
 \frac{\widehat\mu(\sigma^{d_j}(C(n+1,z))}{\widehat\mu\sigma^{d_j}(C(n,z))^{1+\gamma}}\geq c \qquad \mbox{ for } \quad d_j\leq n<\widetilde d_{j+1}
$$
By construction (also in  theorem \ref{cantorsimbolico}) we have that 
$$ \sigma^{\widetilde d_{j+1}}(C(n,z))=C(n-\widetilde d_{j+1},w) \quad \mbox{for } \quad \widetilde d_{j+1}\leq n\leq d_{j+1}$$

Recall that if $\#{\cal P}<\infty$ then  we  have that ${\cal P}$ is $\ep$-uniformly $\widehat{\mu}$-good (since $\widehat{\mu}$ is mixing). For $\#{\cal P}$ infinite,  if $\widehat{\mu}$  has summable uniform rate of mixing of order $3$ in ${\cal P}$ and $\phi\in L^2(\widehat{\mu})$, then  ${\cal P}$ is $\ep$-uniformly $\widehat{\mu}$-good (see corollary  \ref{mejoregorof}). 

\end{remark}

\medskip

\begin{proof}[Proof of theorem \ref{cantorsimbolico}]
We construct the $(\widetilde{\cal J},{\cal J})$-pattern set ${\cal Z}_\ep $ in a similar way as in theorem \ref{cantorsimbolicoformalismo}, but now we use the measure $\widehat\mu$ instead of  the measure $\bf \widehat m$ when we apply proposition \ref{SMsinmalos}. Since the Markov chain  is topologically mixing we have (as in the proof of theorem \ref{cantorsimbolicoformalismo}) that 
there exists a cylinder $\widetilde{J}_0\subset \widehat{P}$ such that $\widetilde{J}_0 \in \widehat{\cal P}_{\widetilde{d}_0}$ and $\sigma^{\widetilde{d}_0}(\widetilde{J}_0 )=C(0,w)$. We  define $J_0:=\widetilde J_0$ and ${\cal J}_0=\widetilde{\cal J}_0=\{\widetilde J_0\}$.
Moreover, let $i_0$ be  the first symbol of the cylinder $\widehat P$. In the case $\sup K_n=\infty$ we are assuming that $\sup\phi<\infty$, and so   we can choose $d_0$ large enough so that for all $m\geq d_0$
\begin{equation} \label{p(iv)}
\widehat s_m\leq e^{m\ep[P_G(\phi)-\sup\phi -\ep)]/2}\leq\left( \frac{1}{c(i_0)}e^{m[P_G(\phi)-\sup\phi -\ep)]}\right)^{\ep}
\end{equation}
and so by (\ref{controlsm}) we know that the condition (iv) holds  for all $m\geq d_0$. Of course, in the  case $\sup K_n<\infty$,  the condition (iv) holds for all $z$ (with $\ep=0$) and  the hypothesis $\sup\phi<\infty$ is not required.

To construct the family  $\widetilde{\cal J}_1$ we  apply proposition \ref{SMsinmalos} for the measure $\widehat\mu$ with  $P_1=P_2=C(0,w)$.  Let $N_1$ be a  natural number large enough so that ${\cal S}_{N}(M,\ep)$ with $N=M=N_1$ satisfies  (\ref{Sn}).
We define $\widetilde{\cal J}_1$ as the family of cylinders  $\inv_{J_0}(S)$
with $S\in{\cal S}_{N_1}(N_1,\ep)$.
Notice that by proposition \ref{SMsinmalos}
$$
\musymb(\sigma^{d_0}(\widetilde{\cal J}_1 \cap J_0)):= \sum_{\widetilde{J}_1\in\widetilde {\cal J}_1} \musymb(\sigma^{d_0}(\widetilde{J}_1))\geq 
 \overline\delta\, \musymb(\sigma^{d_0}(J_0)), \quad 
\text{with} \quad \overline\delta =\frac 12 \musymb(C(0,w)).
$$
and,  by  taking $N_1$ large  enough,   from property (\ref{Ergodic})  we have for all 
$\widetilde J_1\in\widetilde{ \cal J}_1$ that 
$$
 e^{-(\widetilde{d}_1-d_0)(P_G(\phi)-\int \phi \, d
{\widehat{\mu} }+2\ep)}\le 
\frac{\musymb(\sigma^{d_0}
(\widetilde J_1))}{\musymb(\sigma^{d_0}(J_0))} \le
e^{-(\widetilde{d}_1-d_0)(P_G(\phi)-\int \phi \, d
{\widehat{\mu} }-2\ep)} \quad \text{ with }\quad \widetilde{d}_1=d_0+N_1\,.
$$
In order to get the condition (v) we will also ask $N_1$ be large enough  so that $\widetilde d_1\geq 2d_0$.

Let ${\cal I}(w)$ denote the sequence given by the definition of $\musymb$-hitting point. 
We can also choose $N_1$   so that there exists $p_i\in {\cal I}(w)$ such that $ \ell_{\widetilde{d}_1} \le p_i <\ell_{\widetilde{d}_1+1} .$ We  denote this natural number $p_i$ by $p(\widetilde{d}_1)$.

In each set  $S$ in ${\cal S}_{N_1}(N_1,\ep)$ we  take the subset $L:=\inv_{S}(C(p(\widetilde{d}_1),w))$ and we denote by ${\cal L}_1$ this family of sets. We define
the family ${\cal J}_1$   as  the collection  $\inv_{J_0}(L)$ with $L\in{\cal L}_1$.
By construction for all 
$ J_1\in{ \cal J}_1$ there exists an unique, then $ {\widetilde J}_1\in{\widetilde  {\cal J}}_1$ such that $J_1\subset \widetilde J_1$,
$$
\sigma^{\widetilde{d}_1}(\widetilde J_1)=C(0,w),\quad 
\sigma^{\widetilde{d}_1}(J_1)=C(p(\widetilde{d}_1),w) \subset C(t_{\widetilde{d}_1},w),  \quad \text{ and }\quad\sigma^{d_1}(J_1)=C(0,\sigma^{p(\widetilde{d}_1)}(w)),
$$
with $\widetilde{d}_1=d_0+N_1$ and $d_1:=\widetilde{d}_1+p(\widetilde{d}_1).$
Also, we have from lemma \ref{chachisymb}, (we may assume that  $0\in{\cal I}(w)$), that 
\begin{equation}\label{hitprop}
\frac{1}{\widehat s_{N_1+p(\widetilde{d}_1)}}\frac{\musymb(C(p(\widetilde{d}_1),w))}{\musymb(C(0,w))}
\leq
\frac{\musymb(\sigma^{d_0}(J_1))} 
{\musymb(\sigma^{d_0}(\widetilde J_1))}\leq \widehat s_{N_1+p(\widetilde{d}_1)}\frac{\musymb(C(p(\widetilde{d}_1),w))}{\musymb(C(0,w))}
\end{equation}
Recall that  $p(\widetilde{d}_1)<\ell_{\widetilde{d}_1+1} $. So,
from  (\ref{lineal2})  and  by using that $\lim_{m\to\infty}\frac1m\log \widehat s_m=0,$ and  $\limsup_{n\to\infty} \ell_n/n<\infty$ if $\sup_nK_n=\infty$, we have  for $N_1$ large enough that
$$
e^{-\ep \widetilde d_1}\, \musymb(C(p(\widetilde{d}_1),w))\leq
\frac{\musymb(\sigma^{d_0}(J_1))} 
{\musymb(\sigma^{d_0}(\widetilde J_1))}
\leq e^{\ep \widetilde d_1}\, \musymb(C(p(\widetilde{d}_1),w))
$$
and 
$$
{\musymb(C(p(\widetilde{d}_1),w))} \geq \, \musymb(C(\ell_{\widetilde{d}_1+1},w))\ge \, e^{-\widetilde{d}_1 (\overs+\ep)}.
$$

The inductive step follows as  in the proof of theorem \ref{cantorsimbolicoformalismo}; just remember that   we use proposition \ref{SMsinmalos} for the  measure $\widehat\mu$ (instead of   $\bf \widehat m$). Notice that  property  (iv) holds due to (\ref {p(iv)})  and  property (v) follows because we choose $N_{m+1} $ large enough so that 
$\widetilde d_{m+1}=d_m+N_{m+1}\geq 2(m+1)(d_0+\cdots+d_m)$.

\end{proof}

\subsubsection{Proofs of Remarks    \ref{paradim} and   \ref{paradimsin}} \label{remarks}

\begin{proof}[Proof of Remark \ref{paradim}]
The construction of the pattern set ${\cal Z}_{\ep}$ is similar to the one  realized in the proof of theorem \ref{cantorsimbolicoformalismo}. The  difference is  that,  by corollary \ref{muybuenos},   we can substitute the families $S_{N_j}(N_j,\ep)$ by the families $S_{N_j}(m_0,\ep)$ with $m_0$ is a fixed natural number.

For 
 $d_j\leq n<\widetilde d_{j+1}$  with $z\in {\cal Z}_{\ep}$, we 
have that 
$$
S\subset \sigma^{d_j}(C(n,z))\subset P_1:= C(0,\sigma^{p(\widetilde d_j)}(w)) \mbox { for some } \,  S\in S_{N_{j+1}}(m_0,\ep)
$$
and so  $\sigma^{d_j}(C(n,z))=C(n-d_j,u)\subset P_1$ for some $ u \in  Good (m_0,\ep)$.

\noindent Therefore,  for $n-d_j\geq m_0$
\begin{equation}\label{eme1}
\frac{\mhat(\sigma^{d_j}(C(n+1,z)) }{\mhat(\sigma^{d_j}(C(n,z)) ^{1+\gamma}}\geq \frac{e^{-(n-d_j+1)(P_G(t\phi)-t\int\phi\, d\mhat+\ep)}}{e^{-(n-d_j)(P_G(t\phi)-t\int\phi\, d\mhat-\ep)(1+\gamma)}}  , \, 
\end{equation}
and  for $0\leq n-d_j< m_0$
\begin{equation}\label{eme2}
\frac{\mhat(\sigma^{d_j}(C(n+1,z)) }{\mhat(\sigma^{d_j}(C(n,z)) ^{1+\gamma}}\geq 
\frac{\mhat(C(m_0,u))}{\widehat\mhat(P_1) ^{1+\gamma}}\geq \mhat(C(m_0,u))\geq 
e^{-m_0(P_G(t\phi)-t\int\phi\, d\mhat+\ep)} .
\end{equation}
From lemma \ref{cocientem2}  
and by using that $w$ is a $\widehat\mu$-hitting point (recall that $\sigma^{d_j}(C(n,z))\subset P_1= C(0,\sigma^{p(\widetilde d_j)}(w))\, $), the condition (\ref{localm}) holds and $\lim_{n\to\infty}\frac 1n \log \widehat k_n=0$,   we get  the following estimation with $c$ a positive constant 
\begin{equation}\label{eme3}
\left(\frac{\widehat\mu(\sigma^{d_j}(C(n+1,z)) }{\widehat\mu(\sigma^{d_j}(C(n,z)) ^{1+\gamma}}\right)^t\geq 
c\,{ e^{-(n-d_j)[\gamma(P_G(t\phi)-tP_G(\phi))+\ep]} } \, \frac{\mhat(\sigma^{d_j}(C(n+1,z)) }{\mhat(\sigma^{d_j}(C(n,z)) ^{1+\gamma}}
\end{equation}
Let $v:=\sigma^{\widetilde d_j}(z)$. 
Since $\sigma^{p(\widetilde d_j)}(v)=\sigma^{d_j}(z) \in C(0,\sigma^{p(\widetilde d_j)}(w))$,  from lemma \ref{chachisymb}  we have that  
\begin{align}
\widehat\mu(\sigma^{\widetilde d_j}(C(n+1,z))=\widehat\mu(C(n+1-\widetilde d_j,v))&\geq \frac{1}{\widehat s_{n+1-\widetilde d_j}}
 \frac{\widehat\mu(\sigma^{p(\widetilde d_j)}(C(n+1-\widetilde d_j,v)))}{\widehat\mu(\sigma^{p(\widetilde d_j)}(C(p(\widetilde d_j),v)))} \, \widehat\mu(C(p(\widetilde d_j),v))  \notag \\
&=\frac{1}{\widehat s_{n+1-\widetilde d_j}} \,  \frac{\widehat\mu(C(p(\widetilde d_j),v)}{\widehat\mu(\sigma^{p(\widetilde d_j)}(C(p(\widetilde d_j),v)))} \, \widehat\mu(\sigma^{d_j}(C(n+1,z)) \notag
\end{align}
and  also 
$$
\widehat\mu(\sigma^{\widetilde d_j}(C(n,z))=\widehat\mu(C(n-\widetilde d_j,v))\leq {\widehat s_{n-\widetilde d_j}} \,  \frac{\widehat\mu(C(p(\widetilde d_j),v)}{\widehat\mu(\sigma^{p(\widetilde d_j)}(C(p(\widetilde d_j),v)))} \, \widehat\mu(\sigma^{d_j}(C(n,z)) 
$$
Hence,
$$
\left(\frac{\widehat\mu(\sigma^{\widetilde d_j}(C(n+1,z)) }{\widehat\mu(\sigma^{\widetilde d_j}(C(n,z)) ^{1+\gamma}}\right)^t\geq 
     \frac{1}{(\widehat s_{n+1-\widetilde d_j} \widehat s_{n-\widetilde d_j}^{\, 1+\gamma} )^t} 
         \left(  \frac{\widehat\mu(\sigma^{p(\widetilde d_j)}(C(p(\widetilde d_j),v)))}{\widehat\mu(C(p(\widetilde d_j),v)}\right)^{\gamma t}
\left(\frac{\widehat\mu(\sigma^{d_j}(C(n+1,z)) }{\widehat\mu(\sigma^{d_j}(C(n,z)) ^{1+\gamma}}\right)^t
$$
We use again lemma \ref{chachisymb},   
and also that   $C(p(\widetilde d_j)-1,v)=C(p(\widetilde d_j)-1,w)$, since $v\in C(p(\widetilde d_j), w)$, and
we obtain 
$$
\left(\frac{\widehat\mu(\sigma^{\widetilde d_j}(C(n+1,z)) }{\widehat\mu(\sigma^{\widetilde d_j}(C(n,z)) ^{1+\gamma}}\right)^t \geq 
 \frac{1}{[\widehat s_{n+1-\widetilde d_j} \widehat s_{n-\widetilde d_j}^{\, 1+\gamma} \,
  \widehat s_{d_j-\widetilde d_j}^{\gamma} ]^t}\, \frac{1}
 {  (\widehat\mu(C(p(\widetilde d_j)-1,w)))^{\gamma t}}
\left(\frac{\widehat\mu(\sigma^{d_j}(C(n+1,z)) }{\widehat\mu(\sigma^{d_j}(C(n,z)) ^{1+\gamma}}\right)^t \notag \\
$$
Since $p(\widetilde d_j)\geq \ell_{\widetilde d_j}$, $\liminf_{n\to\infty} \widehat\mu(C(n+1,w))/ \widehat\mu(C(n,w))>0$ and $\unders=\liminf_{n\to\infty}-(\log \widehat\mu(C(\ell_n,w))/n$ we have
$$
\widehat\mu(C(p(\widetilde d_j)-1,w))\leq \widehat\mu(C(\ell_{\widetilde d_j}-1,w))\leq c_1e^{-\widetilde d_j(\unders -\ep)}
$$
for some positive constant $c_1$. And as 
 $\lim_{n\to\infty}\frac 1n \log \widehat s_n=0$ we get that  
$$
\left(\frac{\widehat\mu(\sigma^{\widetilde d_j}(C(n+1,z)) }{\widehat\mu(\sigma^{\widetilde d_j}(C(n,z)) ^{1+\gamma}}\right)^t\geq
c_2 e^{-(n-\widetilde d_j)\ep(1+\gamma)t} e^{\widetilde d_j(\unders -\ep)\gamma t}
\left(\frac{\widehat\mu(\sigma^{d_j}(C(n+1,z)) }{\widehat\mu(\sigma^{d_j}(C(n,z)) ^{1+\gamma}}\right)^t 
$$
for some $c_2>0$.
We recall that $p(\widetilde d_j)\leq \ell_{\widetilde d_j+1}$ and $\limsup_{n\to\infty}\ell_n/n<\infty$, and so  
$$n-\widetilde d_j=n-d_j+ p(\widetilde d_j)< n-d_j+\ell_{\widetilde d_j+1}\leq  n-d_j+c\, \widetilde d_j
$$
for some positive constant $c$. Therefore since $\unders>0$ for $\ep>0$ small enough we have that
$$
\left(\frac{\widehat\mu(\sigma^{\widetilde d_j}(C(n+1,z)) }{\widehat\mu(\sigma^{\widetilde d_j}(C(n,z)) ^{1+\gamma}}\right)^t\geq
c_2 e^{-(n- d_j)\ep(1+\gamma)t} 
\left(\frac{\widehat\mu(\sigma^{d_j}(C(n+1,z)) }{\widehat\mu(\sigma^{d_j}(C(n,z)) ^{1+\gamma}}\right)^t 
$$
By using (\ref{eme1}), (\ref{eme2})  and (\ref{eme3}) in the above inequality we get the following:

\noindent For $n-d_j\geq m_0$
$$
\left(\frac{\widehat\mu(\sigma^{\widetilde d_j}(C(n+1,z)) }{\widehat\mu(\sigma^{\widetilde d_j}(C(n,z)) ^{1+\gamma}}\right)^t\geq c\, e^{(n-d_j)[\gamma t(P_G(\phi)-\int\phi\, d\mhat)-(3+\gamma-(1+\gamma)t)\ep]}\geq c \, , 
$$
and for $0\leq n-d_j< m_0$
$$
\left(\frac{\widehat\mu(\sigma^{\widetilde d_j}(C(n+1,z)) }{\widehat\mu(\sigma^{\widetilde d_j}(C(n,z)) ^{1+\gamma}}\right)^t\geq c\, 
e^{-m_0[(\gamma+1)P_G(t\phi)-\gamma tP_G(\phi)-t\int\phi\, d\mhat+2\ep+\ep(1+\gamma)t]}=\mbox{ctte} 
$$

Finally, notice  that for 
 $\widetilde d_{j+1}\leq n <d_{j+1}$  with $z\in {\cal Z}_{\ep}$, we
 have that  
$$
C(p(\widetilde d_{j+1}),w) \subset \sigma^{\widetilde d_{j+1}}(C(n,z))\subset  C(0,w)
$$
and so  $ \sigma^{\widetilde d_{j+1}}(C(n,z))=C(n-\widetilde d_{j+1},w)$ and  $ \sigma^{\widetilde d_{j+1}}(C(n+1,z))=C(n+1-\widetilde d_{j+1},w)$ .

\end{proof}

\begin{proof}[Proof of Remark \ref{paradimsin}]
The construction of the pattern set ${\cal Z}_{\ep}$ is similar to the one  realized in the proof of theorem \ref{cantorsimbolico}. 
The  difference is  that,  by   corollary \ref{muybuenos} ,   we can substitute the families $S_{N_j}(N_j,\ep)$ by the families $S_{N_j}(m_0,\ep)$ with $m_0$ is a fixed natural number.  

For $d_j\leq n<\widetilde d_{j+1}$  with $z\in {\cal Z}_{\ep}$   
we have that 
$$
S\subset \sigma^{d_j}(C(n,z))\subset P_1:= C(0,\sigma^{p(\widetilde d_j)}(w)) \mbox { for some } \,  S\in S_{N_{j+1}}(m_0,\ep)
$$
and so  $\sigma^{d_j}(C(n,z))=C(n-d_j,u)\subset P_1$ for some $ u \in  Good (m_0,\ep)$.

Therefore, for $n-d_j\geq m_0$
$$
\frac{\widehat\mu(\sigma^{d_j}(C(n+1,z)) }{\widehat\mu(\sigma^{d_j}(C(n,z)) ^{1+\gamma}}\geq \frac{1}{{c(w_{p(\widetilde d_j)})}^{2+\gamma}}\frac{e^{-(n+1-d_j)(P_G(\phi)-\int \phi \, d
{\widehat{\mu} }+\ep)}}{e^{-(n-d_j)(P_G(\phi)-\int \phi \, d
{\widehat{\mu} }-\ep)(1+\gamma)}} \geq c\, e^{(n-d_j)[\gamma(P_G(\phi)-\int \phi \, d
{\widehat{\mu} })-\ep(2+\gamma)]}\geq c 
$$
and for $0\leq n-d_j< m_0$
$$
\frac{\widehat\mu(\sigma^{d_j}(C(n+1,z)) }{\widehat\mu(\sigma^{d_j}(C(n,z)) ^{1+\gamma}}\geq 
\frac{\widehat\mu(C(m_0,u))}{\widehat\mu(P_1) ^{1+\gamma}}\geq 
e^{-m_0(P_G(\phi)-\int \phi \, d
{\widehat{\mu} }+\ep)} 
$$
Finally, notice that  for
$\widetilde d_{j+1}\leq n <d_{j+1}$  with $z\in {\cal Z}_{\ep}$   
we have that 
$$
C(p(\widetilde d_{j+1}),w) \subset \sigma^{\widetilde d_{j+1}}(C(n,z))\subset  C(0,w)
$$
and so  $ \sigma^{\widetilde d_{j+1}}(C(n,z))=C(n-\widetilde d_{j+1},w)$ and  $ \sigma^{\widetilde d_{j+1}}(C(n+1,z))=C(n+1-\widetilde d_{j+1},w)$ .

\end{proof}

\subsection{Special cases } \label{special}
Along this section  we assume that
$$
s:=\lim_{n\to\infty}  \frac 1n \log \frac{1}{ {\musymb(C(\ell_n,w))}}<\infty.
$$
If $\widehat\mu$ is  weak (i.e $\sup_n K_n=\infty$), then we  also require  $\limsup_{n\to\infty} {\ell_n}/{n}<\infty$.

\medskip
\noindent We recall  (see remark \ref{s=infty}) that if $s=\infty$, then the target-ball set has zero $\widehat\mu$-dimension. Also if $\lim_{n\to\infty}\ell_n/n=\infty$ and $\widehat\mu$ is ergodic, $\phi\in L^1(\widehat\mu)$ and $\sum_{n\geq 1}V_n(\phi)<\infty$, then by  (\ref{geq0}) we have   that $s=\infty$  for $w$  $\widehat\mu$-a.e.

\subsubsection{Markov chain with finite alphabet}

For $\phi$ continuous (with $P_{top}(\phi)<\infty$)
 the function $G(t):=P_{top}(t\phi)-P_{top}(\phi)t$ is continuous, convex, decreasing and $G(1)=0$. Therefore  for the results in the previous section we have for all $w\in \Sigma_A^{\cal I}$ the following:

\begin{theorem}\label{introA} If the alphabet is finite and for all  $0<t\leq 1$ there is a mixing weak $\sigma$-Gibbs measure with continuous potential $t\phi$,
then 
$$
\Dim_{\widehat\mu}(W_{\sigma}(\widehat P,\ell_n,w)))=T
$$
with $T$ the root of the equation
$$
P_{top}(t\phi)-P_{top}(\phi)t=st
$$
If $\widehat\mu$ is  mixing, then
$$T\geq \frac {P_{top}(\phi)-\int \phi \, d
{\widehat{\mu} }}{P_{top}(\phi)-\int \phi \, d
{\widehat{\mu} }+s} \geq  \frac {h_\musymb}{h_\musymb+s}\,.
$$
and moreover, if $\widehat\mu$ is the (unique) equilibrium measure of the potential $\phi$ (which is exact and Gibbs), then
$P_{top}(\phi)-\int \phi \, d{\widehat{\mu}}=h_{\widehat\mu}$. By Shannon-MacMillan-Breiman theorem if $v=\lim_{n\to\infty}\ell_n/n$ then 
$s=vh_{\widehat\mu}$ for   $w$ $\widehat\mu-$a.e.

\end{theorem}

\begin{remark}
From classical  results  of Bowen and Ruelle (see \cite{Bo}, \cite{Ru1}) one knows that  for finite alphabet  there is a unique equilibrium measure  for any potential 
with the Walters property. Moreover,  this measure is a $\sigma$-Gibbs measure and it is exact (whence ergodic and strong mixing). So the hypothesis of above theorem holds for these potentials.
\end{remark}

We refers to section \ref{NoH}  where we consider non-H\"older potentials  $\phi$ such that there exists a mixing weak $\sigma$-Gibbs measure with  potential $t\phi$.

\subsubsection{Markov chain with infinite alphabet but BI property}

\begin{proposition} \label{localGibb}Let $\Sigma_A^{\cal I}$ be  a topologically mixing Markov chain with infinite countable alphabet  satisfying the BI property and
$\phi:\Sigma_A^{\cal I}\longrightarrow\R$  be a positive recurrent potential with $P_G(\phi)<\infty$ and such that  $V_1(\phi)<\infty$ and  $\phi$ is weak H\"older continuous. Then    the Ruelle-Perron-Frobenius measure $\mhat$ 
verifies
$$
0<\frac{1}{c}\inf\{h(x) : x\in C(0,z)\}\leq \frac{\mhat(C(n-1,z))}{\exp(-nP_G(\phi)+\sum_{j=0}^{n-1}\phi\circ\sigma^j(z))}\leq c
$$
with $c>1$ depending on the finite collection of symbols given by the BI property and $h$  a positive continuous function such that $\log h$ and $\log h\circ \sigma$  are  weak H\"older continuous and $V_1(\log h)<\infty$.
In particular $\mhat$ is local $\sigma$-Gibbs.
\end{proposition}

\begin{proof} By Sarig results (see \cite{Sa4} or \cite{Sa}) we have that $d\mhat=hd\nu$  with $h$ a positive continuous function and $\nu$ a consevative measure such that  $L_{\phi}h=\lan h$ and $L^{*}_{\phi}\nu=\lan \nu$
with $\lan=\exp P_G(\phi)$ and $L_{\phi}$ de Ruelle operator. The measure $\nu$ is  finite and positive  on cylinders and $\log h$ and $\log h\circ \sigma$  are  weak H\"older continuous and $V_1(\log h)<\infty$.

Proceeding as in \cite{Sa} (or in \cite{SaNotes} p. 98)  by using that $h$ is bounded away from  zero and infinity on $0$-cylinders and $\sum_{n=1}^{\infty}V_n(\phi)<\infty$ we have that 
$$
\mhat(C(n-1,z)\asymp \frac{1}{\lan^n}\int L^n_{\phi}\1_{C(n-1,z)} d\,\nu\asymp  \frac{1}{\lan^n} \exp(\sum_{j=0}^{n-1}\phi\circ\sigma^j(z))\, \nu(\sigma^n(C(n-1,z)))
$$
with comparability constants depending on the bounds  of $h$ in $C(0,z)$. And since we have the BI property, then
$$
\nu(\sigma^n(C(n-1,z)))=\nu(\sigma(C(0,\sigma^{n-1}(z))))\geq \inf\{\nu(C_k)\, :\, k\in{\cal I}_0\}>0
$$
where ${\cal I}_0$ is the finite collection of symbols given by the BI property.
(Sarig also  proved that the BI property implies that $\sup h<\infty$)

\medskip

To conclude that $\mhat$ is local Gibbs we will see that $\nu(\sigma(C(0,\sigma^{n-1}(z))))$ is upper bounded by a  positive constant (depending on ${\cal I}_0$ and $z_0$).  By the BI property we know that there exists $k\in {\cal I}_0$ such that $C_{z_{n-1}k}\neq\varnothing$, also since the Markov chain is topologically mixing  given $k\in {\cal I}_0$ and $z_0$ there exists two  finite sequences $u\in \Sigma_A^{\cal I}|_m$, $v\in \Sigma_A^{\cal I}|_s$ such that  $u_0=v_0=v_s=k$, $u_m=z_0$, and $u_i, v_j\neq k$ for $1\leq i\leq m-1$, $1\leq j\leq s-1$.

\medskip
For all  $w\in \Sigma_A^{\cal I}$  and $\ell>m+n$ we have that 
$$
\lan^{n+m}\nu(C(\ell+s-1,w))=\int L^{n+m}_{\phi}\1_{C(\ell+s-1,w)} d\,\nu\asymp  \exp(\sum_{j=0}^{m+n-1}\phi\circ\sigma^j(w))\nu(\sigma^{m+n}(C(\ell +s-1,w)))
$$
and  also
$$
\lan^{n+m}\nu(C(m+n-1,w))=\int L^{n+m}_{\phi}\1_{C(m+n-1,w)} d\,\nu\asymp  \exp(\sum_{j=0}^{m+n-1}\phi\circ\sigma^j(w))\nu(\sigma^{m+n}(C(m+n-1,w)))
$$
with absolute constants of comparability. 
Hence,  if $w\in \Sigma_A^{\cal I}$ has the property
\begin{equation}\label{pa}
w|_{m+n}=(k,u_1,\, \cdots\, , u_{m-1}, z_0,\, \cdots\,, z_{n-1}, k)
\end{equation}
then  $\nu(\sigma(C(0,\sigma^{n-1}(z)))=\nu(\sigma^{m+n}(C(m+n-1,w)))$  and we have 
$$
\nu(\sigma(C(0,\sigma^{n-1}(z)))\asymp\frac{ \nu(C( \ell +s-n-m-1,\sigma^{m+n}(w)))\, \nu(C(m+n-1,w))}{\nu(C(\ell+s-1,w))}.
$$
Therefore 
\begin{equation}\label{conm}
\nu(\sigma(C(0,\sigma^{n-1}(z)))\asymp\frac{ \mhat(C( \ell +s-n-m-1,\sigma^{m+n}(w)))\, \mhat(C(m+n-1,w))}{\mhat(C(\ell+s-1,w))}
\end{equation}
with  comparability constants depending on the bounds of $h$ in  the finite collection of the $0$-cylinders $C_k$ with $k\in{\cal I}_0$. 

Now fix $p, \ell \in\N$ with $\ell>m+n+p$ and consider the points $w\in \Sigma_A^{\cal I}$ verifying (\ref{pa}) and with the properties:
\begin{itemize}
\item[(a)] $\sigma^{\ell}(w)|_{s-1}=v|_{s-1}=(k, v_1,\, \cdots, \, v_{s-1})$
\item[(b)] $\#\{w_i : w_i=k \, \mbox{ with }\,  m+n<i<\ell\}=p$
\end{itemize}
Note that we can write  (\ref{conm}) as 
$$
\nu(\sigma(C(0,\sigma^{n-1}(z)))\, {\mhat(C(\ell+s-1,w))}\asymp{ \mhat(\sigma^{n+m}(C( \ell +s-1,w)))\, \mhat(C(m+n-1,w))}
$$
and for all $w$ verifying (\ref{pa})  $C(m+n-1,w))$ is the cylinder $A:=C_{ku_1\ldots u_{m-1}z_0\ldots z_{n-1}}$. 
By adding all over $w\in \Sigma_A^{\cal I}$  with $\sigma^{\ell+s}(w)=w$ and  verifying (\ref{pa}) and properties (a) and (b)
 we get that  
$$
\nu(\sigma(C(0,\sigma^{n-1}(z)))\asymp\frac{ \mhat [B_{\ell}\cap\sigma^{-(\ell-(m+n))}(C(s-1,v)) ]\, \mhat(A)}{\mhat[A\cap \sigma^{-(m+n)}(B_{\ell})\cap \sigma^{-\ell}(C(s-1,v))]}
$$
with 
$B_{\ell}:=\{(y_0,y_1, \cdots )\in \Sigma_A^{\cal I}: y_0=k  \mbox{ and }\#\{y_i :y_i=k \, \mbox{ with } \,  0<i<\ell-m-n \}=p \}$.
Hence for $p$ fixed
\begin{equation}\label{cf0}
\nu(\sigma(C(0,\sigma^{n-1}(z)))\asymp\frac{ \mhat [\bigcup_{\ell>m+n}B_{\ell}\cap\sigma^{-(\ell-(m+n))}(C(s-1,v)) ]\, \mhat(A)}{\mhat[A\cap\bigcup_{\ell>m+n} \sigma^{-(m+n)}(B_{\ell})\cap \sigma^{-\ell}(C(s-1,v))]}
\end{equation}

Let $\overline \sigma:C_k\longrightarrow C_k$  denote the induced first return map on the   $0$-cylinder $C_k$, i.e.
$$\overline\sigma(x)=\sigma^{\varphi(x)}(x) \quad \mbox{ with } \quad \varphi(x)=\1_{C_k}\inf\{ n>0: \sigma^n(x)\in C_k\}
$$
and let   ${\cal S}:=\{w|_n : n\geq 0, w_0=k, w_i\neq k \mbox{ for } 0<i\leq n, (w_0,w_1,\ldots,w_n,k)\in \Sigma_A^{\cal I}|_{n+1}\}$ the set  of  induced symbols.  Then $\overline v:=v|_{s-1}\in {\cal S}$, and  $A$ is a cylinder in the  topological Markov chain $\Sigma^{\cal S}$. If we denote the cylinders in $\Sigma^{\cal S}$ by $\overline C(n,.)$, and $A$ is the cylinder $\overline C(r-1,x)$ (for some $r\in\N$ and $x\in \Sigma^{\cal S}$),  then  we can re-write (\ref{cf0}) as
$$
\nu(\sigma(C(0,\sigma^{n-1}(z)))\asymp\frac{ \overline {\bf m} (\overline\sigma^{-(p+1)}(\overline C(0,\overline v)))\, \overline {\bf m}(\overline C(r-1,x))}{\overline {\bf m}[\overline C(r-1,x)\cap \overline\sigma^{-(r+p+1)}(\overline C(0,\overline v))]}
=\frac{ \overline {\bf m}(\overline C(0,\overline v))\, \overline {\bf m}(\overline C(r-1,x))}{\overline {\bf m}[\overline C(r-1,x)\cap \overline\sigma^{-(r+p+1)}(\overline C(0,\overline v))]}
$$
where $\overline {\bf m}$ is the $\overline\sigma$-invariant measure $\mhat_{k}\circ\overline\pi$ with $\mhat_{k}$ the normalized restriction of $\mhat$ to $C_k$ and $\overline\pi: \Sigma^{\cal S}\longrightarrow C_k\subset \Sigma_A^{\cal I}$ the natural injection.
But in  \cite{Sa4} Sarig obtained  (by using  results of Aaronson, Denker and Urbanski \cite{ADU} on the Schweiger property) that $\overline\sigma$ is exponentially continued fraction mixing (see remark \ref{fracmixing}), and therefore for $p$ large 
$$
\overline {\bf m}[\overline C(r-1,x)\cap \overline\sigma^{-(r+p+1)}(\overline C(0,\overline v))]\geq (1-c\,  \theta^{p+1}) \overline {\bf m} (\overline C(0,\overline v))\, \overline {\bf m}(\overline C(r-1,x))
$$
for some constant $c>1$ and $\theta\in(0,1)$. So, we obtain that 
$\nu(\sigma(C(0,\sigma^{n-1}(z)))$ is upper bounded by a constant depending  on  $z_0$ (and on the finite family ${\cal I}_0$).

\end{proof}

\begin{theorem} If $\Sigma_A^{\cal I}$  satisfies the BI property,  
$\phi:\Sigma_A^{\cal I}\longrightarrow\R$   is weak H\"older continuous such that  $V_1(\phi)<\infty$    and  $\phi$ is  positive recurrent then,  for $w$  $\widehat\mu$-a.e$$
 T^-\leq \Dim_{\widehat\mu}(W_{\sigma}(\widehat P,\ell_n,w)))\leq T^+
$$
with 
$$
\begin{aligned}
T^-=&\sup\{t\in(0,1]: \infty>P_{G}(t\phi)-P_{G}(\phi)t>st\,,  t\phi \mbox{ positive recurrent} \mbox{ and  } \phi\in L^1({\mhat}_t)\} \\
T^+=&\inf\{t\in(0,1]: P_{G}(t\phi)-P_{G}(\phi)t<st\} \\
\end{aligned}
$$
Here $\mhat_t$ is the RPF measure for the potential $t\phi$. 
\end{theorem}

\begin{remark} \label{abajoo}
By results in section 5.1.2 we also know that if $\widehat\mu$ is mixing with potential $\phi$ such that $\sum_{n=1}^{\infty}V_n(\phi)<\infty$,  $\phi\in L^1(\mu)$, and $\sup\phi<\infty$ in the case $\widehat\mu$ is not  Gibbs (i.e $\sup_n K_n=\infty$),  then 
\begin{equation}\label{abajo}
\Dim_{\widehat\mu}(W_{\sigma}(\widehat P,\ell_n,w)))\geq  \frac {P_{G}(\phi)-\int \phi \, d
{\widehat{\mu} }}{P_{G}(\phi)-\int \phi \, d
{\widehat{\mu} }+s}
\end{equation}
In particular, 
if $\widehat\mu$ is the RPF measure of a positive recurrent potential $\phi$ with summable variations such that $\sup\phi<\infty$ and $P_G(\phi)<\infty$, and $h_{\widehat\mu}<\infty$, then (see Sarig's results in section 3.1) $\widehat\mu$ is an equilibrium measure which is exact (whence ergodic and strong mixing). Hence ${P_{G}(\phi)-\int \phi \, d
{\widehat{\mu} }} =h_{\widehat\mu}$.
\end{remark}

\begin{proof} By Sarig's results  (see \cite{Sa4}) we know that  there is a finite RPF measure ($\mhat_t$)  for the (positive recurrent) potential $t\phi$  which is  $\sigma$-invariant and exact and by lemma \ref{localGibb} it is  a local $\sigma$-Gibbs measure. Also,  we know that condition (\ref{localm}) holds in a set of full $\widehat\mu$-measure  (see remark \ref{R}). This allow us to get the result from the theorems on  section 5.1.
\end{proof}

If $\Sigma_A^{\cal I}$ satisfies the stronger property  BIP, then any  potential  $\phi$ with  the Walters condition  and $V_1(\phi)<\infty$ and $P_G(\phi) <\infty$ is positive recurrent, and the RPF measure is a $\sigma$-Gibbs measure exact (see \cite{Sa2}, \cite{SaNotes}).  Notice that since  measures are not local, then (\ref{localm}) always holds.
Hence in this case we get the following:

\begin{theorem}\label{BIP1}
If $\Sigma_A^{\cal I}$  satisfies the BIP property,  
  the Walters condition  holds for $\phi$, $V_1(\phi)<\infty$ and $P_G(\phi) <\infty$,  then  for all $w$ (with $s<\infty$)
$$
 T^-\leq \Dim_{\widehat\mu}(W_{\sigma}(\widehat P,\ell_n,w)))\leq T^+
$$
with 
$$
\begin{aligned}
T^-=&\sup\{t\in(0,1]: \infty>P_{G}(t\phi)-P_{G}(\phi)t>st\, \mbox{ and  } \phi\in L^1({\mhat}_t)\} \\
T^+=&\inf\{t\in(0,1]: P_{G}(t\phi)-P_{G}(\phi)t<st\} \\
\end{aligned}
$$
Here $\mhat_t$ is the RPF measure for the potential $t\phi$. 
\end{theorem}

\begin{remark} 
If 
\begin{equation}\label{entropiafinita}
-\sum_{i\in{\cal I}}{\mhat_t(C_i)}\log \mhat_t(C_i)<\infty 
\end{equation}
then  $\mhat_t$ has finite entropy.  But, if BIP holds and  $\phi$ has summable variations, $\sup\phi<\infty$, $P_G(t\phi)<\infty$,  and $\mhat_t$ has finite entropy, then $\phi\in L^1({\mhat}_t)$ and
$\mhat_t$ is an equilibrium measure for $t\phi$ (see section \ref{TForm} for references).  
Hence, under these hypothesis, if there exists $0<t_1\leq 1$ such that
 $\infty>P_{G}(t_1\phi)-P_{G}(\phi)t_1>st_1$ and (\ref{entropiafinita}) holds for all $t_1\leq t\leq 1$, then 
$$T^-=\sup\{t\in [t_1,1]: P_{G}(t\phi)-P_{G}(\phi)t>st\}$$
If $\widehat\mu$ is $\sigma$-Gibbs (for example if it is the RPF measure for $\phi$) then the condition (\ref{entropiafinita}) is equivalent to
\begin{equation}\label{entropiafinita2}
 -\sum_{i\in{\cal I}}{\widehat\mu(C_i)^t}\log \widehat\mu(C_i)<\infty 
\end{equation}
 \end{remark}

\begin{remark}\label{igualdad} 
For infinite alphabet  the function 
$$
G(t):=P_{G}(t\phi-P_G(\phi)t)=P_G(t\phi)-P_G(\phi)t
$$ 
is  not  necessarily continuous, this was  first notice by Mauldin and Urbanski  in \cite{MU} with their pressure. Hence,  
the equation $P_{G}(t\phi)-P_{G}(\phi)t=st$ might not have a root.  However, Sarig proved that if BIP holds, $\phi$ is weakly H\"older continuous and $P_G(t_1\phi)<\infty$, then $t\longrightarrow P_G(t\phi)$ is  real analytic for all $t>t_1$. See   \cite{Sa3}   and \cite{Sa5} for further results  on the non-analyticity of $G(t)$ in relation with  the properties of recurrence of $t\phi$. 
We recall that $G(t)$ is convex and $G(1)=0$ and so if $0<G(t_1)<\infty$, then $G(t_2)<G(t_1)$ for all $t_2>t_1$. 
\end{remark}

From theorem \ref{BIP1} and  the last two remarks follow that:

\begin{theorem}\label{intro}
If $\Sigma_A^{\cal I}$  satisfies the BIP property,  $\widehat\mu$ is a $\sigma$-Gibbs measure for $\phi$ with 
$ \sum_{n\geq 1}V_n(\phi)<\infty$, $\sup \phi<\infty$, 
and there exists $0<t_1\leq 1$ such that
 $\infty>P_{G}(t_1\phi)-P_{G}(\phi)t_1>st_1$ and (\ref{entropiafinita2}) holds for  $t=t_1$, then
 $$
 \Dim_{\widehat\mu}(W_{\sigma}(\widehat P,\ell_n,w)))=\sup\{t\in[t_1,1]: P_{G}(t\phi)-P_{G}(\phi)t>st\}=\inf\{t\in(0,1]: P_{G}(t\phi)-P_{G}(\phi)t<st\} .
$$
And  if moreover $\phi$ is weakly H\"older continuous, then  
$$ \Dim_{\widehat\mu}(W_{\sigma}(\widehat P,\ell_n,w)))=T\quad  \mbox{ with } \quad 
P_{G}(T\phi)-P_{G}(\phi)T=sT
$$
\end{theorem}

\begin{remark} \label{Gaussmu}
Theorem C of introduction  follows from above theorem and remark \ref{abajoo}. Notice that $\phi(w)=-\log|g'\circ\pi(w)|$ with $g(x)=\frac 1x - \lfloor \frac{1}{x }\rfloor$ the Gauss map.
\end{remark}

\subsubsection{Non-H\"older potentials} \label{NoH}

In the previous two sections  we have mainly considered potentials $\phi:\Sigma_A^{\cal I}\longrightarrow\R$  with some H\"older regularity property  on a topologically mixing Markov chain $\Sigma_A^{\cal I}$ verifying  BI or BIP. These conditions were used to guarantee the existence of mixing local $\sigma$-Gibbs measures for $t\phi$; however, theorems in section \ref{mudim} only required mixing local  weak $\sigma$-Gibbs measures.

The existence of weak Gibbs measures for non-H\"older potentials have been studied for several authors.
For example,  Hu considered  in \cite{Hu} potentials  $\phi :\Sigma_A^{\cal I}\longrightarrow\R$, with  finite alphabet ${\cal I}=\{0,1, \ldots r^*\}$ and $a_{0,0}=1$,  which satisfy the H\"older condition everywhere  except at the  fixed point $0$ and its preimages. He defined a new metric $\widetilde d$ on $\Sigma_A^{\cal I}$,  the definition of this metric involves two fixed parameters $\kappa\in(0,1)$ and $\gamma>0$, and in particular, he  considered  potentials $\phi$ such that:
\begin{itemize}
\item [(i)] $\phi$ is continuous.
\item [(ii)] There exists $\theta\in(0,1]$ and $\alpha\in[0,\theta(1+\gamma))$ such that  for all $w\in{\cal B}_k$ and $v\in{\cal B}_{k-1}\cup{\cal B}_k\cup{\cal B}_{k+1}$
$$
|\phi(w)-\phi(v)|\leq C k^{\alpha-1} {\widetilde d}(w,v)^{\theta} \quad \mbox{ for some constant  } \quad C>0.
$$
\item [(iii)] There exists $\beta>-1$ and $K_1>0$ such that for all $k\geq K_1$ if $w\in{\cal B}_k$ and $0=(0,0,\ldots)$, then 
$$
|\phi(0)-\phi(w)-\frac{\beta+1}{k+1}|\leq C \frac{1}{(k+1)^{1+\delta}}
$$
with $0<\delta\leq 1$ and $C=C(\delta)>0$  independent of $k$ and $w$.  
\end{itemize}
 \noindent  In properties (ii) and (iii) ${\cal B}_k$  denotes the set  of words $w=(w_0,w_1,\ldots)\in \Sigma_A^{\cal I}$ with $w_0=w_1=\cdots=w_{k-1}=0$ and $w_k\neq0$.

Due to definition of $\widetilde d$ and the restrictions on the values of $\alpha$, the condition (ii) implies a H\"older condition. Also,  even though the condition (iii)  implies that $\phi$ is not  a H\"older function with the usual metric, the potential $\phi$  satisfies  H\"older condition with respect to the metric $\widetilde d$.

Let us assume that $\widehat\mu$ is a weak $\sigma$-Gibbs measure for the potential $\phi$ with $P=P_{top}(\phi)=0$ and $\phi$ verifying properties (i)-(iii).  The measure $\widehat\mu$ is an ergodic equilibrium measure for $\phi$ (see lemma 7.4  and theorem D in  \cite{Hu}). Moreover, for all $0\leq t\leq 1$ there exists a
 exact weak $\sigma$-Gibbs  measure for  the potential $t\phi$ with $P=P_{top}(t\phi)$, see \cite{Hu}.
  Notice that since $P_{top}(0)>0$ and $P_{top}(\phi)=0$, follows from the convexity of $G(t):=P_{top}(t\phi)$ that $G(t)$ is decreasing.

 Hence,  we have  that
 $$
\Dim_{\widehat\mu}(W_{\sigma}(\widehat P,\ell_n,w)))=T\geq \frac {h_\musymb}{h_\musymb+s}
$$
with $T$ the root of the equation
$
P_{top}(t\phi)=st
$.

\medskip

We would like to mention that M.  Yuri  (\cite{Yu1}, \cite {Yu2}) has  established  a thermodynamic formalism for  {\it countable to one transitive Markov systems with finite range structure (FRS) satisfying  a property of local exponential instability},  
and he has proved the existence  of  equilibrium and weak $\sigma$-Gibbs measures  for some  non-H\"older potentials. See also \cite{FFY}.  Examples of such systems are  some non-hyperbolic system which exhibit intermittent phenomena, as the Manneville-Pomeau map,   we will  discuss them with more detail in section \ref{intermittent}.

\section{Markov Transformations with BI property} \label{MarkovT}

Let $\lan$ be Lebesgue measure in $[0,1]$. A map
$f:[0,1]\longrightarrow [0,1]$ is a {\it Markov transformation} if
there exists a  finite or numerable family ${\cal P}_0=\{P^0_i\}_{i\in {\cal I}}$ of disjoint
open intervals in $[0,1]$ such that

\begin{itemize}
\item[(a)] $\lan([0,1]\setminus \cup_j P^0_j)=0$.

\item[(b)] For each $j$, there exists a set $K$ of indices such
that $f(P^0_j)=\cup_{k\in K} P^0_k$ (mod $0$).

\item[(c)] $f$ is derivable in $\cup_j P^0_j$ and there exists
$\sigma>0$ such that $|f'(x)|\ge \sigma$ for all $x\in\cup_j P^0_j$.

\item[(d)] There exists $\gamma>1$ and a non zero natural number
$n_0$ such that if $f^m(x)\in \cup_j P^0_j$ for all $0\le m \le
n_0-1$, then $|(f^{n_0})'(x)|\ge\gamma$.

\item[(e)] Given $P_i, \, P_j \in {\cal P}_0$ there is $n_0\in\N$ such that
$\lambda (f^{-n}(P_i)\cap P_j) > 0$, for all $n\ge n_0$\,.

\item[(f)]  For some $P\in {\cal P}_0$, \
$\sum_n\lan(f^{-n}(P))=\infty$ and $\sum_n n\,\lan([\varphi_P=n])<\infty$, with $\varphi_P : P\longrightarrow P$ the first return time to $P$, i.e. $\varphi_P(x) =\inf\{k : f^k(x)\in P\}$.

\item[(g)] There exist constants $c>0$ and $0<\alpha\le 1$ such that, for all $x,y \in P^0_j$,
$$
\Big| \frac{f'(x)}{f'(y)}-1 \Big| \le c|x-y|^\alpha \,.
$$
\end{itemize}

Let us define  the following collections $\{{\cal P}_i\}$ of partitions:
$$
{\cal P}_{n}=\bigcup_{P^0_i\in{\cal
P}_0}\{(f\big|_{P^0_i})^{-1}(P_j)\, :\, P_j\in{\cal P}_{n-1} \,, \
P_j \subset f(P^0_i)\}=\bigvee_{j=0}^n f^{-j}({\cal P}_0);
$$
 from (a) and (b) we have that
$
 \bigcap_{n=0}^\infty \bigcup_{P\in {\cal P}_n} \mbox{cl}(P)
$
has full Lebesgue measure. 
Let us  recall some well known properties of Markov transformation (see e.g  \cite{M}, \cite{Aa}). 
From properties (d) and (g)  we get the following expansiveness property:  there exists $c>0$ and $0<\beta<1$ such that for all $x,y$ in the same interval of ${\cal P}_n$ 
\begin{equation}\label{expansive}
|x-y|\leq c\,  \beta^n |f^n(x)-f^n(y)|
\end{equation}
In particular
 we have that 
 $
 \displaystyle\sup_{P\in{\cal P}_n} \diam(P) \to 0 \mbox{ as } n\to\infty.
 $

Also, it is easy to check that 
there exists an absolute constant $c>0$ such that for all natural number $n$
we have that if $y,z$ belong to  the same interval of ${\cal P}_n$  then
$$
\frac {(f^s)'(y)}{(f^s)'(z)} \le c \,, \qquad \hbox{for $s=1,\dots,n+1$}\,.
$$
And therefore we have the following  bounded distortion property:

\begin{proposition} \label{elchachi}
If $P$ is an element of ${\cal P}_n$ and $A$ is a measurable subset of $P$ then
$$
\frac{\lambda(A)}{\lambda(P)} \asymp
\, \frac {\lambda(f^{j}(A))}{\lambda
(f^{j}(P))}\, , \qquad \hbox{ for } j=1,\dots, n+1\,.
$$
\end{proposition}

There exists  
 a $f$-invariant  probability measure  $\mu$ which is absolutely
continuous with respect to Lebesgue  measure $\lambda$, usually called the ACIP measure. The measure $\mu$  is 
comparable to $\lambda$ in the initial intervals of ${\cal P}_0$, it is exact and has entropy $h_{\mu}>0$.

Any Markov transformation is a shift modeled transformation in $[0,1]$ (see section \ref{SMT}). Let us denote by  $(\Sigma_A^{\cal I},d,\sigma)$
the  symbolic representation  of $f$. Notice that the Markov chain   $\Sigma_A^{\cal I}$   is { topologically mixing} by  property (e).

\medskip

 Let  $-\log|\widehat f' |$ denote the { continuous extension to  $\Sigma_A^{\cal I}$ of  the } potential $-\log|f'\circ\pi |$  in $\pi^{-1}(X_{\Pibf}^{1-1})$. Property (f) says that the potential $-\log|\widehat f' |$ is positive recurrent  and 
by property (g)  and (\ref{expansive}) follow that there exists $c>0$ and $0<\theta<1$ such that 
\begin{equation}\label{regularidadpotencial}
V_n(-\log|\widehat f' |)\leq c\, \theta^{n}  \qquad \mbox{for} \quad n\geq 1
\end{equation}
Also, by property (d) we have that there exists $\eta>0$  such that 
$
-\log|{(\widehat f^{n_0})'} | \leq -\eta 
$
and therefore
for all $\sigma$-invariant probability $\bf \widehat m$ we have that
$
\int \log|{\widehat f^\prime} | \,d{\bf \widehat m}>0.
$
}

\medskip 

We  denote by $\musymb$ the measure in $\Sigma_A^{\cal I}$ corresponding to de ACIP measure $\mu$, i.e $\musymb=\mu\circ\pi$. 
 If $f$ satisfies  the BI property, i.e. $\inf\{\lan(f(P_i^0): P_i^0\in{\cal P}_0\}>0$,  then the measure $\widehat\mu$ is a  local   $\sigma$-Gibbs measure for the potential $-\log|\widehat f' |$ with {$P=P_G(-\log|\widehat f' |)=0$ } (see proposition \ref{localGibb}).

   \medskip
     
Hereafter, we will assume that  $f$ (and therefore $ \Sigma_A^{\cal I}$)  satisfies the BI property.
We recall that if we have BIP property, then  $\widehat\mu$ is a $\sigma$-Gibbs measure.

\medskip

Notice that since $\mu$ is comparable to $\lan$ in  each initial block $P^0_i\in {\cal P}_0$, then  from lemma \ref{gridmu} we have that 
$$
\Dim_{\Pibf}(\pi(\Sigma))=\Dim_{\widehat\mu}(\Sigma) \quad \mbox{ for all } \Sigma\subset \Sigma_A^{\cal I}
$$

\subsection{Target-block sets for Markov transformations with BI}\label{ShrinkingM}

We say that $x\in X_{\Pibf}$ is a $\mu$-hitting point iff  there exists  a $\musymb$-hitting point  $w\in \Sigma_A^{\cal I}$ such that $\pi(w)=x$. Notice that the set of $\mu$-hitting points  has $\mu$-full measure (see remark \ref{R}), and moreover  if 
 BIP property holds then $\mu$ is $\sigma$-Gibbs and so 
all  point in $X_{\Pibf}$  is a  $\mu$-hitting point. 

Given  a $\mu$-hitting point $x\in[0,1]$ with $\pi(w)=x$,  a block $P\in{\cal P}_N$, and a sequence $\{\ell_n\}\subset \N$ such that
$
\limsup_{n\to\infty}{\ell_n}/{n}<\infty
$
we want to estimate the size of the set
$$
W_f(P,\ell_n,x)=\{y\in{P}: \ f^k(y) \in
P(\ell_k,x) \ \hbox{for infinitely many }k\}.
$$
In a similar way that in the proof of proposition \ref{cotasupformalismo} by using that $\mu$ and $\lan$ are comparable in the initial blocks and proposition \ref{elchachi} we get the following upper bound for the Hausdorff dimension.

\begin{proposition}\label{arriba} Let
$
{\underline s}= \liminf_{n\to\infty} -[\log {\lan(P(\ell_n,x))}]/n<\infty.
$
Then
$$
\Dim(W_f(P,\ell_n,x))  \le T^+:=\inf\{t>0: \, P_G(-t\log|\widehat f^\prime|)<\unders t\}
$$ 
\end{proposition}

\begin{proof} 
For each $N\in\N$  we have the following covering of $W_f(P,\ell_n,x)$
$$
\bigcup_{n=N}^\infty \{P(n+\ell_n,y): \ f^n(y)=x, \quad y\in P^0_i\}\,.
$$
with $P^0_i$  the $0$-block such that $P\subset P^0_i$. Given $\ep>0$ and $t>0$, 
from   proposition \ref{elchachi}, definition of  ${\underline s}$ and since  $f^n(y)=x$, we have for $n$ large that
$$
\frac{\lan(P(n+\ell_n,y))}{\lan(P(n,y))}\leq c \, \frac{\lan(P(n,x))}{\lan (P(0,x))}\leq c \, \frac{e^{-n(\underline s-\ep/(2t))}}{\lan(P(0,x))}, 
$$
and since $\mu$ is comparable to $\lan$ in $P^0_i$
$$
\lan(P(n+\ell_n,y))\leq c\, e^{-n(\underline s-\ep/(2t))}\widehat\mu(C(n,z))
$$
with $\pi(z)=y$ and $c>0$ depending on the initial block $P^0_i$.
Hence for $t>0$ and $N$ large
$$
\sum_{n\ge N} \sum_{f^n(y)=x, y\in P^0_i} \diam( P(n+\ell_n,y))^t \le 
 c\,  \sum_{n\ge N} e^{-n(\underline s t-\ep/2)} \sum_{\sigma^n(z)=w, z_0=i}   \, \widehat\mu ( C(n,z))^t\,.
$$
But 
$$
{\musymb (C(n,z))}^t\leq  c(i)^{t}  \,  \exp[-t\sum_{j=0}^n \log|\widehat  f^\prime\circ\sigma^j(z)|]
$$
and so for $n$ large 
$$
 \sum_{\sigma^n(z)=w, z_0=i}  \, {\musymb (C(n,z))}^t\leq  c(i)^{t}   \sum_{\sigma^n(z)=w, z_0=i} e^{\sum_{j=0}^n (-t\log|\widehat f^\prime\circ\sigma^j(z)|)} \leq c\, e^{n[P_G(-t\log|\widehat f^\prime|)+\ep/2]}
 $$
Therefore,
$$
\sum_{n\ge N} \sum_{f^n(y)=x, y\in P^0_i} \diam( P(n+\ell_n,x))^t 
  \leq c \sum_{n\ge N}  e^{-n[\unders t-P_G(-t\log|\widehat f^\prime|)-\ep]} \,.
$$
For $\ep>0$ and $t>0$ such that $\unders t-P_G(-t\log|\widehat f^\prime|)-\ep>0$ we have that 
$$
\sum_{n\ge N} \sum_{f^n(y)=x, y\in P^0_i} \diam( P(n+\ell_n,x))^t \to 0 \qquad \mbox{ as } \qquad N\to\infty
$$
\end{proof}
\begin{remark} \label{uppermodeled}
A similar result holds for $f$ 
 a shift modeled transformation in $[0,1]$  (not necessarily a Markov transformation) with    a $f$-invariant  measure  $\mu$ such that  $\mu\asymp\lan $ in the block $P$ and $\widehat\mu$   a weak $\sigma$-invariant Gibbs measure with potential $-\log |\widehat f^{\prime}|$  and $P_G(-\log |\widehat f^{\prime}|)=0$.  In this case we should change $\lan$ by $\mu$ in the definition of ${\underline s}$ (or ask for $\mu\asymp\lan $ in $P(\ell_n,x)$ for all $n$ large enough).
\end{remark}

In this section let us assume the following condition on the center $x$  of the target-block
\begin{equation} \label{hipx}
\liminf_{n\to\infty}\frac{\lan(P(n+1,x))}{\lan(P(n,x))^{1+\gamma}}>0  \qquad \mbox{ for all } \quad \gamma>0
\end{equation}
Notice that this condition holds for all $x$  with $\gamma=0$ in the case of finite alphabet (follows  from proposition \ref{elchachi}).
Moreover, in the infinite case,  since $\mu\asymp \lan$ in $P(0,x)$ and $\mu$ is ergodic, then by Shannon-MacMillan-Breiman theorem follows  that the set of points $x$ such that  (\ref{hipx}) holds  has full $\lan$-measure. We recall that $h_\mu>0$.

\noindent We also assume that 
\begin{equation} \label{hipx2}
{\overline s}:= \limsup_{n\to\infty} \frac 1n \log \frac 1{\lan(P(\ell_n,x))}<\infty 
\end{equation}

\begin{theorem} \label{markovT1}
$$
\Dim(W_f(P,\ell_n,x))  \ge T^-
$$ 
where $T^-$ is the supremum of all $t>0$ such that:
$$
 P_G(-t\log|{\widehat f^\prime} |)>{\overline s} t
$$
and  there exists a  mixing  local weak $\sigma$-Gibbs measure ${\mhat}$ for the potential $-t\log|{\widehat f^\prime} |$
verifying 
(\ref{localm}),  $-t\log|{\widehat f^\prime} |\in L^1(\mhat)$, and such that
 the collection $\{C(0, \sigma^{p_j}(w)) : p_j\in{\cal I}(w) \}$
  is $\ep$-uniformly $\widehat{\bf m}$-good for all $\ep>0$ small enough.
 Here $\pi(w)=x$  and $w$ is a $\widehat\mu$-hitting point with sequence ${\cal I}(w)$.
 \end{theorem}

\begin{proof}
Let  $\widehat P$ be the cylinder  such that $\pi(\widehat P)=P$ and $\phi=-\log|{\widehat f^\prime} |$. {We recall that $P_G(\phi)=0$ and $- \int \phi \, d\mhat>0$.}
 Let us suppose that there exist
 $t>0$ such that 
 $
P_G(t\phi)>t\overs+\ep
 $ for some $\ep>0$
and  there exists  a mixing weak $\sigma$-Gibbs measure $\mhat$ for the potential $t\phi$.  From theorem \ref{cantorsimbolicoformalismo} we know that  there exists a pattern like set $ {\cal Z}\subset  W_{\sigma}(\widehat P,\ell_n,w)$ with $\Dim_{\widehat\mu}({\cal Z}) \geq t$ and  verifying also remark \ref{paradim}.  
Hence $\pi({\cal Z})$ is contained in $W_{f}(P,\ell_n,x)$  and by lemma \ref{gridmu} $\Dim_{\bf\Pi}(\pi({\cal Z}))\geq t$.  Next, we want to use 
remark \ref{comparacioncongrid} to conclude that $\Dim(\pi({\cal Z}))=\Dim_{\bf\Pi}(\pi({\cal Z}))$. Notice that in the case of finite alphabet we have (\ref{hipx}) for any point in $[0,1]$ and then we are done.  Now in the infinite countable case, for all $\gamma>0$ we get  the following by using (twice) proposition \ref{elchachi}:

\medskip

For $\widetilde d_{j+1}\leq n < d_{j+1}$  
\begin{align}
\frac{\diam(P(n+1,\pi(z)))}{\diam(P(n,\pi(z)))^{1+\gamma}}
&\asymp  
\frac{\lan(f^{\widetilde d_{j+1}}(P(n+1,\pi(z))))}
{\lan(f^{\widetilde d_{j+1}}(P(n,\pi(z)))) ^{1+\gamma}}
 \, \left (   \frac{\lan(f^{ \widetilde d_{j+1}}(P(n,\pi(z)))}{\lan(P(n,\pi(z))}\right )^{\gamma}
\notag \\
&\asymp 
\frac{\lan(f^{\widetilde d_{j+1}}(P(n+1,\pi(z))))}{\lan(f^{\widetilde d_{j+1}}(P(n,\pi(z))))^{1+\gamma}}  \,
 \, \left (   \frac{\lan(f^{ \widetilde d_{j+1}}(P( \widetilde d_{j+1},\pi(z)))}{\lan(P( \widetilde d_{j+1},\pi(z))}\right )^{\gamma} 
   \notag
\end{align}
Notice that  $f^{\widetilde d_{j+1}}(P(n,\pi(z))=P(n-\widetilde d_{j+1},x)$ for all  $\widetilde d_{j+1}\leq n \leq  d_{j+1}$,  since $\sigma^{\widetilde d_{j+1}}(C(n,z))=C(n-\widetilde d_{j+1},w)$ and $\pi(w)=x$, (see remark \ref{paradim}). Hence, by using
 (\ref{hipx}) we get  that
\begin{equation}\label{primera}
\frac{\diam(P(n+1,\pi(z)))}{\diam(P(n,\pi(z)))^{1+\gamma}}
\asymp
\frac{\lan(P(n+1-\widetilde d_{j+1},x))}{\lan(P(n-\widetilde d_{j+1},x))^{1+\gamma}}  \,
 \, \left (   \frac{\lan(P( 0,x)}{\lan(P( \widetilde d_{j+1},\pi(z))}\right )^{\gamma} 
 \geq c_1
\end{equation}
for some  constant $c_1>0$ (depending on $x$). 

\medskip

For $d_j\leq n<\widetilde d_{j+1}$, we get in a similar way that 
$$
\frac{\diam(P(n+1,\pi(z)))}{\diam(P(n,\pi(z)))^{1+\gamma}}
\asymp 
 \frac{\lan(f^{ \widetilde d_j}(P(n+1,\pi(z))))}{\lan(f^{\widetilde d_j}(P(n,\pi(z))))^{1+\gamma}} \, \left (   \frac{\lan(f^{ \widetilde d_j}(P( \widetilde d_j,\pi(z)))}{\lan(P( \widetilde d_j,\pi(z))}\right )^{\gamma} 
 \geq c_2\,  \frac{\lan(f^{ \widetilde d_j}(P(n+1,\pi(z))))}{\lan(f^{  \widetilde d_j}(P(n,\pi(z))))^{1+\gamma}} 
$$
for some constant $c_2>0$ (depending on $x$). We recall that 
 $\sigma^{ \widetilde d_{j}}(z)\in C(\ell_{d_{j}},w)\subset C(0,w)$ for all $z\in{\cal Z}$  with $\pi(w)=x$,  and therefore $f^{ \widetilde d_j}(P( \widetilde d_j,\pi(z))=P(0,x)$.
Moreover, by using that 
 $\mu$ is comparable to $\lan$ in the block $P(0,x)$ (and  $\sigma^{\widetilde d_j}(z)\in C(0,w)$) we  obtain from    remark \ref{paradim} that 
\begin{equation}\label{segunda}
\frac{\diam(P(n+1,\pi(z)))}{\diam(P(n,\pi(z)))^{1+\gamma}}\geq c_1\,  \frac{\lan(f^{ \widetilde d_j}(P(n+1,\pi(z))))}{\lan(f^{  \widetilde d_j}(P(n,\pi(z))))^{1+\gamma}}\asymp 
\frac{\widehat\mu(\sigma^{\widetilde d_j}(C(n+1,z)))}{\widehat\mu(\sigma^{\widetilde d_j}(C(n,z)))^{1+\gamma}}
\geq c
\end{equation}
Therefore, from (\ref{primera}), (\ref{segunda}) and  remark \ref{comparacioncongrid}  we have that $\Dim(\pi({\cal Z}))=\Dim_{\bf\Pi}(\pi({\cal Z}))\geq t$.

\end{proof}

If the potential $-t\log|{\widehat f^\prime} |$ is positive recurrent, then there exists the RPF measure $\mhat_t$  (see section \ref{WG}) which is exact; since we have BI property and (\ref{regularidadpotencial}) holds, the measure $\mhat_t$  is  local $\sigma$-Gibbs. Moreover,  if there exists an increasing  sequence $\{p_j \}\subset\N$  such that
\begin{equation}\label{recurrent}
 \#\{ P(0,f^{p_j}(x) )\}<\infty
\end{equation}
for $x$ the 
center of the target-block set (by Poincare recurrence theorem this happen $x$ $\mu$-a.e), then (\ref{localm}) holds and 
 the collection $\{C(0, \sigma^{p_j}(w)) : p_j\in{\cal I}(w) \}$
  is $\ep$-uniformly $\widehat{\bf m}$-good for all $\ep>0$ (since it is a finite collection).
Hence we have the following corollary:
  
\begin{corollary} For all $x\in X_{\bf\Pi}$ verifying  (\ref{hipx}) and  (\ref{recurrent}) 
$$
\Dim(W_f(P,\ell_n,x))  \ge T^- \quad \mbox{with}
$$  
$$T^-:=
\sup\{t\in(0,1]: \infty>P_G(-t\log|{\widehat f^\prime} |)>{\overline s}t\,,  -t\log|{\widehat f^\prime} | \mbox{ positive recurrent} \mbox{ and  }-t\log|{\widehat f^\prime} |\in L^1({\mhat}_t)\} \\
$$ 
Here ${\mhat}_t$ is the RPF measure for  the potential $ -t\log|{\widehat f^\prime} |$.
\end{corollary}

\begin{remark} If $\overline s=0$, then $\Dim(W_f(P,\ell_n,x)) =1$ 
\end{remark}

Sarig proved 
that if the  BIP property holds, then any  potential  $\phi$ with  the Walters condition  and $V_1(\phi)<\infty$ and $P_G(\phi) <\infty$ is positive recurrent. And moreover, the RPF measure is a $\sigma$-Gibbs measure which is exponentially continued fraction mixing. So the collection of all $0$-cylinders is $\ep$-uniformly $\widehat{\bf m}$-good for all $\ep>0$ (see section 3.3). Therefore

\begin{corollary} \label{BIPmenos} If  $\Sigma_A^{\cal I}$ satisfies the BIP property, then for all $x\in X_{\bf\Pi}$ verifying  (\ref{hipx}) 
$$
\Dim(W_f(P,\ell_n,x))  \ge T^-
$$  with
$$ \quad T^-:=
\sup\{t\in(0,1]: \infty>P_G(-t\log|{\widehat f^\prime} |)>{\overline s}t\,,  \mbox{ and  }-t\log|{\widehat f^\prime} |\in L^1({\mhat}_t)\} \\
$$ 
Here ${\mhat}_t$ is the RPF measure for the potential $-t\log|{\widehat f^\prime} |$.
\end{corollary}

Since $\sup -t\log|{\widehat f^\prime} |<\infty$ (due to the property (c) of the Markov transformation), by Sarig results we also know that, under BIP property, if  the  RPF measure ${\mhat}_t$ has finite entropy, then $-t\log|{\widehat f^\prime} |\in L^1({\mhat}_t)$ and 
${\mhat}_t$ is an equilibrium measure.
Notice that  if $-\sum_{i\in{\cal I}}{\mhat}_t(C_i)\log {\mhat}_t(C_i)<\infty$, then ${\mhat}_t$ has finite entropy, but 
$$
-\sum_{i\in{\cal I}}{\mhat}_t(C_i)\log {\mhat}_t(C_i)\asymp -t\sum_{i\in{\cal I}}\widehat\mu(C_i)^t\log \widehat\mu(C_i)
$$
Hence, from   proposition \ref{arriba} and corollary \ref{BIPmenos} (see also remark \ref{igualdad}) follows that:

\begin{corollary} \label{paraBIP0}If  $\Sigma_A^{\cal I}$ satisfies the BIP property,  
$
s=\lim_{n\to\infty} -[\log{\lan(P(\ell_n,x))}]/n<\infty, 
$
and 
$$
\sum_{i\in{\cal I}}\widehat\mu(C_i)^{t_1}\log \widehat\mu(C_i)<\infty \quad \mbox{ and } \quad { \infty>} P_G(-t_1\log|{\widehat f^\prime} |)>st_1 \quad   \mbox{ for some }  0<t_1\leq 1
$$
 then  for all $x\in X_{\bf\Pi}$ verifying  (\ref{hipx}) 
$$
\Dim(W_f(P,\ell_n,x))=\sup\{t\geq t_1: P_G(-t\log|{\widehat f^\prime} |)>st\}= \inf\{t>0: \, P_G(-t\log|\widehat f^\prime|)<s t\}=T
$$  
with $T$ such that  $P_G(-T\log|{\widehat f^\prime} |)=sT$
\end{corollary}

As in the proof of theorem \ref{markovT1}, but  using  now theorem \ref{cantorsimbolico} and remark \ref{paradimsin}, we  also get the following  lower bound from the Hausdorff dimension

\begin{theorem}\label{markovT2} 
If  $\int \log|f^\prime|\, d{{\mu}} <\infty$, then
$$
\Dim(W_f(P,\ell_n,x))  \ge 
\frac{\int \log|f^\prime|\, d{{\mu}}}{\int \log|f^\prime|\, d{{\mu} }+\overs} \geq \frac {h_\mu}{h_\mu+\overs}\,.
$$
for  all  $\widehat\mu$-hitting point  $x$ verifying (\ref{hipx}) and such that  the collection $\{C(0, \sigma^{p_j}(w)) : p_j\in{\cal I}(w) \}$
  is $\ep$-uniformly $\widehat\mu$-good for all $\ep>0$ small enough.
\end{theorem}
Recall that  conditions on $x$ in the above theorem  hold $\lan$-a.e.,   and if BIP holds then any collection of $0$-cylinders is  $\ep$-uniformly $\widehat\mu$-good.

Notice that by Sarig results we know that the  variational principle holds for  $-\log|{\widehat f^\prime} |$ because it  has summable variations and $\sup -\log|{\widehat f^\prime} |<\infty$. Moreover,  If $-\sum_{i\in{\cal I}}\mu(P^0_i)\log \mu(P^0_i)<\infty$, then  $h_\mu= \int \log|f^\prime|\, d{{\mu}} <\infty$ and 
$\widehat\mu$ is an equilibrium measure.

\subsection {Target-ball sets  for Markov transformations}\label{ball}

 Let $x\in X_{\bf\Pi}$  be a $\mu$-hitting point verifying  (\ref{hipx}), $P$ a block  in ${\cal P}_N$, and  $\{r_n\}$ be a decreasing sequence of positive numbers. We define 
$$
\widetilde W_f(P,r_n,x)= \{y\in P: \ |f^n(y)-x|\le r_n \ \hbox{for infinitely many $n$}\}.
$$
Recall that from Borel-Cantelli lemma  and 
 theorem 4.1 in \cite{FMP2}, we have the following dichotomy :
\begin{equation} \label{BC+}
\begin{aligned}
\quad \mbox{If} \qquad &\sum_nr_n<\infty &\implies 
\liminf_{n\to\infty} \frac {|f^n(y)- x|}{r_n} = \infty
 \quad \mbox { for } \lan-a.e.\;  y \in [0,1]
\\
\quad \mbox{If} \qquad &\sum_nr_n^{\alpha}=\infty \mbox{ for some } \alpha>1 &\implies 
\liminf_{n\to\infty} \frac {|f^n(y)- x|}{r_n} = 0
 \quad \,\mbox { for } \lan-a.e.\;  y \in [0,1]
 \end{aligned}
\end{equation}
If  condition  (\ref{hipx}) holds for $\gamma=0$, then we can take $\alpha=1$ in the second implication.
\medskip

Notice that
\begin{equation}\label{liminf=0}
\widetilde W_f(P,r_n/n,x)\subset \left\{y\in P: \liminf_{n\to\infty} \frac {|f^n(y)- x|}{r_n} = 0\right\}\subset \widetilde W_f(P,r_n,x)
\end{equation}
and our interest is the case $\sum_n r_n<\infty$.
We will assume that 
 $
u:=\lim_{n\to\infty}-[ \log{r_n}]/n<\infty . 
 $

For  $x\not\in\partial{\cal P}_0$ 
let $I_n:=[x-r_n,x+r_n]$; if $x\in\partial{\cal P}_0$ then $I_n$ is half interval, i.e. an interval  the length $r_n$ and $x$ as one of the boundary points, for example if $x=0$, $I_n:=[0,r_n]$ and if $x=1$, $I_n:=[1-r_n,1]$. For all $n$ large enough $I_n\subset P(0,x)\cup\{x\}$, and
therefore  $f^{-n}(I_n)$ is a disjoint union of the intervals $\mbox{cl}(f^{-n}(I_n)\cap P(n,y))$ with $f^n(y)=x$. For each $N$ large we have the following covering of $\widetilde W_f(P,r_n,x)$
$$
\bigcup_{n=N}^\infty \{\mbox{cl}(f^{-n}(I_n)\cap P(n,y)): \ f^n(y)=x, \quad y\in P\}\,.
$$
But we have (see e.g. proposition 3.1 in \cite{FMP2})
$$
\diam(f^{-n}(I_n)\cap P(n,y))=\lan(f^{-n}(I_n)\cap P(n,y))\leq C\, r_n\lan(P(n,y))
$$
Hence, proceeding as in proposition \ref{arriba} we get that
$$
\Dim(\widetilde W_f(P,r_n,x))  \le \inf\{t>0: \, P_G(-t\log|\widehat f^\prime|)<u\, t\}
$$

On the other hand, to look for  lower bounds for the  dimension of this set we should notice the following: Let us define $s_k:=\text{diam\,}(P(p_k,x))$ with ${\cal I}(x)=\{p_i\}$ 
the increasing sequence of the $\mu$-hitting point $x$, and for each $s_k$ let $n(k)$ be  the greatest natural number such that $s_k\le r_{n(k)}$. We denote by  $\widetilde {\cal D}$  this sequence of indexes $\{n(k): \; k\in\N\}$. Since $s_k \to 0$ as $k\to\infty$,  and $r_n\to 0$ as $n\to\infty$ by hypothesis, we have that $n(k)\to\infty$ as $k\to\infty$.
We will write $\widetilde{{\cal D}}=\{\widetilde d_i\}$ with $\widetilde d_i< \widetilde d_{i+1}$ for all $i$. Therefore, we have that if $\widetilde d\in \widetilde {\cal D}$ then there exists $\ell(\widetilde d)\in {\cal I}(x)$ (maybe there is more than one then we choose one) such that
$$
r_{\widetilde d+1} < \text{diam\,} (P(\ell(\widetilde d),x)) \le r_{\widetilde d }\qquad \text{and} \qquad
P(\ell(d),x) \subset I_{\widetilde d}
$$
We define the sequence $\{\ell_n\}$ by $\ell_n:= \ell(\widetilde d_i)$ for $\widetilde d_i\leq n< \widetilde d_{i+1}$.  Since $r_n\to 0$  as $n\to\infty$ we have that $\ell_n\to\infty$ as $n\to\infty$. Notice that
$$
 V:= \{y\in P: \ f^{\widetilde d_i}(y)\in P(\ell_{\widetilde d_i},x)\ \hbox{for infinitely many $i$}\} \subset  \widetilde W_f(P,r_n,x) \,,
 $$
 the sequence $\{a_n:= -[\log \lan(P(\ell_n,x))]/n\}$ is strictly decreasing for $\widetilde d_i\leq n<\widetilde d_{i+1}$,  and so
$$
\limsup_{n\to\infty} \frac 1n \log \frac 1{\lan(P(\ell_n,x))}= \limsup_{i\to\infty} \frac 1{\widetilde d_i} \log \frac 1{\lan(P(\ell(\widetilde d_i),x))}=\lim_{n\to\infty}\frac 1n \log\frac{1}{r_n}:=u
$$
We recall that the lower bounds of the Hausdorff dimension in previous section  came from the $\widehat\mu$-dimension of the $(\widetilde{\cal  J}_j,{\cal J}_j)$ pattern set $\cal Z$ constructed in theorems \ref{cantorsimbolicoformalismo} and \ref{cantorsimbolico}. Moreover, it is clear from the proofs of the mentioned theorems that we can construct these pattern sets with the property that $\widetilde d_j\in\cal D$, and therefore $\cal Z\subset V$ and  again $\Dim_{\widehat\mu}(\cal Z)$ give us the lower bounds for the Hausdorff dimension of $\widetilde W_f(P,r_n,x) $. More precisely, we have similar lower bounds for $\Dim (\widetilde W_f(P,r_d,x))$ as  the ones in section  \ref{ShrinkingM} but with $\overline s=u$.
In particular, we get the following corollary

\begin{corollary}\label{anterior}
 If (\ref{hipx}) and  the BIP property holds and 
there exists $0<t_1\leq 1$ such that 
$$
-\sum_{i\in{\cal I}}\widehat\mu(C_i)^{t_1}\log \widehat\mu(C_i)<\infty \qquad \mbox{ and } \qquad { \infty >}P_G(-t_1\log|{\widehat f^\prime} |)>{ u}t_1,
$$
then 
\begin{align}
\Dim(\widetilde W_f(P,r_n,x))&=\sup\{t\in[t_1,1]: P_G(-t\log|{\widehat f^\prime} |)>ut\}= \inf\{t\in(0,1]: \, P_G(-t\log|\widehat f^\prime|)<u t\} \notag \\ 
&=T\geq 
 \frac{h_{\mu}}{h_{\mu}+u} \notag
\end{align}
with  $T$ such that $P_G(-T\log|{\widehat f^\prime} |)=uT$ 
\end{corollary}

\begin{remark}
Notice that it follows from (\ref{liminf=0}) that under the hypothesis of corollary \ref{anterior}
$$
\Dim(\widetilde W_f(P,r_n,x))=\Dim \left\{y\in P: \liminf_{n\to\infty} \frac {|f^n(y)- x|}{r_n} = 0\right\}
$$
\end{remark} 

\subsection{Some examples}
\subsubsection{Gauss transformation and not Bernoulli modification of Gauss map}
The Gauss transformation is the map $\phi:[0,1]\longrightarrow [0,1]$
defined by $\phi(0)=0$ and  $\phi(x)=\frac 1x -\lfloor{\frac 1x}\rfloor$ for $x\neq 0$, where  $\lfloor{x}\rfloor$ denotes the integer part of $x$. 
The map $\phi$
is a Markov transformation with partition ${\cal P}_0=\{P^0_i:=(1/(i+1),1/i) : \, i \in \N\setminus\{0\} \}$ and such that $\phi(P^0_i)=(0,1)$ for all $i$. 
The  symbolic representation of $\phi$ is the full shift $(\Sigma^{\Z^+}, \sigma)$,  and   if $w=(w_0,w_1,\ldots)\in\Sigma^{\Z^+}$, then the corresponding point in $[0,1]$ is the  irrational point with continued fraction expansion $[\,w_0,w_1,\dots \,]$. Notice that $\Sigma^{\Z^+}$ satisfies Bernoulli property and therefore the BIP property. The ACIP measure is the 
Gauss measure which is defined by
$$
\mu(A)= \frac 1{\log 2} \int_A \frac 1{1+x} \, d\lambda .
$$
If $x$ is an irrational in $[0,1]$  with continued fraction expansion $[\,w_0,w_1,  \dots \,]$,  then  ${\lan(P(n,x))}\asymp1/ |(\phi^{n+1})^{\prime}(x)|$ and 
$$
 \frac{1} {((w_0+1)(w_1+1)\ldots (w_n+1))^2}\leq  \frac{1}{|(\phi^{n+1})^{\prime}(x)|}\leq \frac{1} {(w_0w_1\ldots w_n)^2}
$$
and  in order to have (\ref{hipx}) we require for all $\gamma>0$
$$
\liminf_{n\to\infty}\frac{(w_0 \,w_1\ldots w_{n-1})^{\gamma}}{w_n}>0
$$
for the point $x=[\,w_0,\, w_1,  \dots \,]$ which will be the center of the target. Now, let 
 ${\ell_n}\subset\N$ be a sequence such that $\limsup_{n\to\infty} \ell_n/n<\infty$, we ask for 
$$
\infty>s=s(\ell_n,x):=\lim_{n\to\infty}\frac{1}{n}\log\frac{1}{ \lan(P(\ell_n,x))}\asymp \lim_{n\to\infty}\frac 1n \log \, (w_0w_1\ldots w_{\ell_n})
$$
If $\lim_{n\to\infty}\ell_n/n=:\tau$, then by Shannon-MacMillan-Breiman theorem  $s=\tau h_{\mu}$ for $\lan$-almost all $x$.

For all $t>1/2$ we have that
$$
-\sum_{i=1}^{\infty}\mu(P^0_i)^t\log \mu(P^0_i)\asymp \sum_i^\infty \frac{\log i}{i^{2t}}<\infty.
$$
and also 
for $t>1/2$,  we have that the function $G(t)=P_G(-t\log|\widehat{\phi^{\prime}})$ is analytic, strictly convex and has a logarithmic singularity at $t=1/2$, see e.g  \cite{PW}. Hence, since $\lim_{t\to (1/2)^+}G(t)=+ \infty$ and $G(1)=0$, there exists $1/2<T\leq 1$ such that $G(T)=sT$

\medskip

Given an irrational $x\in[0,1]$ and a sequence $\{\ell_n\}$ with the above properties, we have for all block $P\in{\cal P}_N$ (for some $N$) the following result for recurrent patterns in continued fraction expansions (follows from corollary  \ref{paraBIP0}, 
and theorem \ref{markovT2})

\begin{theorem} {\sc [Continued Fractions]} \label{gausscode}
$
\mbox{If } \quad \sum_n \dfrac 1{(w_0\cdots w_{\ell_n})^2} <\infty \implies \lan(W_\phi(P,\ell_n,x))=0 \quad \mbox{ but  } \quad 
 \Dim (W_\phi(P, \ell_n,x))=T \ge \frac{h_{\mu}}{h_{\mu}+s} $

\medskip

 \noindent  with  $1/2<T\leq 1$ the unique solution of   $P_G(-t\log|\widehat{\phi^{\prime}}|)=ts$, and 
 $h_{\mu}=\frac{\pi^2}{6\log2}$, the entropy of $\phi$.
\end{theorem}

\begin{remark} If 
$
 \sum_n \frac 1{(w_0\cdots w_{\ell_n})^2} =\infty \ ,
$
then $\lan(W_\phi(P,\ell_n,x))=1$, see theorem 3.1 in \cite{FMP2}.
\end{remark}

\begin{remark} If $s\geq h_{\mu}$ then $h_{\mu}/(h_{\mu}+s)\leq 1/2$ and the lower bound is irrelevant.  We recall that
$s=\tau h_{\mu}$ for $\lan$-almost all $x$, with  
 $\tau:=\lim_{n\to\infty}\ell_n/n$, and so $h_{\mu}/(h_{\mu}+s)=1/(1+\tau)$.
\end{remark}

For 
$\{r_n\}$ be a decreasing sequence of positive numbers the dichotomy (\ref{BC+}) holds for $\phi$, 
and 
moreover,  if 
$
 u:=\lim_{n\to\infty}-[\log{r_n}]/n<\infty,
$
 we have the following dimension result for any block $P\in{\cal P}_N$

\begin{theorem}\label{Gauss2} Let $1/2<T\leq 1$ be  the unique solution of  $P_G(-t\log|\widehat{\phi^{\prime}}|)=tu$, then 
$$
\Dim \left\{y\in P\subset [0,1]: \ \liminf_{n\to\infty} 
\frac{|\phi^n(y)-x| }{r_n}=0 \right\} = T  \ge \frac{h_{\mu}}{h_{\mu}+u}=\frac{\pi^2}{\pi^2+(6\log2)u}
$$
\end{theorem}
\noindent In particular for $u=0$ we have that $T=1$.

\medskip

We can modify the Gauss map and consider the following Markov transformation which is not Bernoulli

$$
f(x)=
\begin{cases}
\phi(x)\,, \quad &\text{if } \frac 12<x\leq 1
\,, \vspace{3pt}
\\
\left(1-\frac{1}{\lfloor{\frac 1x}\rfloor}\right)\phi(x)+\frac{1}{\lfloor{\frac 1x}\rfloor}\,, \quad &\text{if } 0<x\leq \frac 12 \,.
\end{cases}
$$
The initial partition is the same that the one for the Gauss  transformation, i.e ${\cal P}_0$, but in this case 
$
\phi(P^0_1)=(0,1)$ and $\phi(P^0_i)=(1/i,1)$ for $i\geq 2$.
The symbolic representation of $f$  is  $(\Sigma_A^{\Z^+},\sigma)$ with transition matrix $A=(a_{i,j})$ such that $a_{1,j}=1$ for all $j\in\Z^+$ and for $i>1$ we have that $a_{i,j}=1$ for $1\leq j<i$ and $a_{i,j}=0$ otherwise. 
We would like to remark  that $f$ satisfies 
 BIP property but it is not Bernoulli. Also, it is easy to check that if $w=(w_0,w_1,\ldots)\in \Sigma_A^{\Z^+}$,  then for the corresponding point $x=\pi(w)\in[0,1]$ we have that  ${\lan(P(n,x))}\asymp1/ |(f^{n+1})^{\prime}(x)|$ 
 $$
 \prod_{j=0}^n \frac{1}{g(w_j)(w_j+1)^2}\leq   \frac{1}{|(f^{n+1})^{\prime}(x)|}\leq \prod_{j=0}^n \frac{1}{g(w_j)w_j^2}
 $$
 with $g(w_j)=1$ if $w_j=1$  and $g(w_j)=1-1/w_j$ otherwise.
We have the following similar result for $f$ with  $x$ and $u$ as in the Gauss case and $x=\pi(w)$ such that
for all $\gamma>0$
$$
\liminf_{n\to\infty}\frac{(g(w_0)^{1/2}w_0 \,g(w_1)^{1/2}w_1\ldots g(w_{n-1})^{1/2}w_{n-1})^{\gamma}}{g(w_n)^{1/2}w_n}>0
$$
This inequality  implies  condition (\ref{hipx}) which we recall holds  for $x$ $\lan$-a.e.

\begin{theorem} Let $1/2<T\leq 1$ be  the unique solution of  $P_G(-t\log|\widehat{f^{\prime}}|)=tu$, then 
$$
\Dim \left\{y\in P\subset  [0,1]: \ \liminf_{n\to\infty} 
\frac{|f^n(y)-x| }{r_n}=0 \right\} = T  \ge \frac{\int \log|f^\prime|\, d{{\mu}}}{\int \log|f^\prime|\, d{{\mu} }+u} 
$$
with $\mu$ the ACIP measure.
\end{theorem}

\subsubsection{Luroth expansion}
Consider the piecewise linear Markov transformation $f:[0,1]\longrightarrow [0,1]$
defined by $f(0)=0$, $f(1)=1$ and 
$$
f(x)=
n(n+1)x - n \,, \qquad \mbox{if  } \quad x\in\left[\frac{1}{n+1},\frac 1n\right) 
$$
The  initial partition is ${\cal P}_0=\{P^0_i:=(1/(i+1),1/i) \, :  i \in \N\setminus\{0\} \}$, and the symbolic representation  of $f$ is the 
full shift $(\Sigma^{\Z^+},\sigma)$. If $w=(w_0,w_1,\ldots)\in\Sigma^{\Z^+}$, then the corresponding point  $\pi(w)=x\in [0,1]$ is the  irrational point with {\it L\"uroth  expansion} $[\,w_0,w_1,\dots \,]$, i.e. (see e.g. \cite{DK})
$$
x= \frac{1}{w_0+1}+\frac{1}{(w_0+1)w_0(w_1+1)}+\frac{1}{(w_0+1)w_0(w_1+1)w_1(w_2+1)}+\cdots =\sum_{i=0}^{\infty}\frac{w_i}{\prod_{k=0}^n(w_k+1)w_k}
$$
Moreover,
$$
\lan(P(n,x))=\frac{1}{(w_0+1)w_0(w_1+1)w_1\cdots (w_n+1)w_n}
$$
and   we require (\ref{hipx})
for  $x$ the center of the target.
 We have that
$$
G(t)=P_G(-t\log|\widehat{f^{\prime}}|)=\log \sum_{n=1}^{\infty}\left(\frac{1}{n(n+1)}\right)^t,
$$
and so $G(t)=\infty$ for $0\leq t\leq 1/2$ and for $t>1/2$ we have that $G(t)$ is real analytic, strictly decreasing and convex and has a unique zero at $t=1$. 
Hence, we obtain for $f$ similar results  to theorems \ref{gausscode} and \ref{Gauss2}, we just should write $\log \sum_{n=1}^{\infty}{1}/{[n(n+1)]^t}$
instead of   $P_G(-t\log|\widehat{\phi^{\prime}}|)$, and now $h_{\mu}$ denotes  the entropy of $f$ with respect to $\lan$, in fact 
$h_{\mu}=\sum_{n=1}^{\infty}{\log[n(n+1)]}/{n(n+1)}$.

\medskip

We refers to \cite{Ba} and \cite{FLMW} for some  dimension results for other class of sets  defined in terms of digit frequencies in L\"uroth expansion.

\subsubsection{Inner functions}
The classical Fatou's theorem asserts that a bounded holomorphic function $F:\D\longrightarrow \C$, from the unit disc $\D$ into the complex plane $\C$, has radial limits almost everywhere. An 
holomorphic function $F$
defined on $\D$ and with values in $\D$ is called an {\it inner function} if the radial limits
\begin{equation}\label{radiallimit}
F^*(\xi):=\lim_{r\to 1^-}  F(r\xi)
\end{equation}
\noindent (which exists for almost every $\xi$ by Fatou's theorem) have
modulus $1$ for almost every $\xi\in\p\D$.
It is well known that any inner function can be written as
$$
F(z) = e^{i\phi} B(z) \, \exp
\left( -\int_{\p\D} \frac{\xi+z}{\xi-z} \,d\nu(\xi) \right)
$$
where $B(z)$ is a Blaschke product and $\nu$ is a finite positive singular measure on $\p\D$. We recall that given a sequence $\{a_n\}$ in $\D$ such that $\sum_{n=1}^\infty (1-|a_n|)<\infty$, 
a Blaschke product  $B(z)$ is a complex function of the type
$$
B(z)= z^m\prod_{a_n\neq 0} \frac {|a_n|}{a_n} \frac{z-a_n}{1-\overline{a_n}z}\,,
$$
where $m$ is the number of elements of the sequence $\{a_n\}$ equal to zero.
If $F$ is inner  with a fixed point $p$ in $\D$ (but $F$  not
conjugated to a rotation),  Aaronson \cite{Aa0} and Neuwirth \cite{N}
proved, independently, that $F^*$ is exact with respect to
harmonic measure $\omega_p$, (see \cite{DM}). 
The mixing properties of inner functions are stronger, 
 Pommerenke \cite{P} proved   that if 
$F$ is an inner function with $F(0)=0$,  but
not a rotation,  then there exists a positive absolute constant
$K$ such that
$$
\left| \frac{\lan[B\cap (F^*)^{-n}(A)]}{\lan(A)} - \lan(B) \right| \le
K\,e^{-\alpha n} \,,
$$
for all $n\in \N$, for all arcs $A,B\subset\p\D$. Here
$\alpha=\max\{1/2, |F'(0)|\}$ and $\lan$ denotes the normalized Lebesgue measure.
If  $F$ is inner with $F(p)=p$ and  $p\in\D$ the above mixing result holds for the harmonic measure $\omega_p$. In the terminology of \cite{FMP1} this imply that $F$
is uniformly mixing at any point $\eta$  of $\p\D$ with respect to the harmonic measure $\omega_p$. 
 From Borel-Cantelli lemma  and 
 theorem 3 in \cite{FMP1}, we have the following dichotomy for any decreasing sequence $\{r_n\}$ of positive numbers:

$$
\mbox{If} \qquad \sum_nr_n<\infty \implies 
\liminf_{n\to\infty} \frac {d((F^*)^n(\xi), \eta)}{r_n} = \infty
 \qquad \mbox { for } \lan-a.e.\; \xi\in\p\D
$$
$$
\mbox{If} \qquad \sum_nr_n=\infty \implies  \liminf_{n\to\infty} \frac {d((F^*)^n(\xi), \eta)}{r_n} =0\qquad \,  \mbox { for } \lan-a.e.\;  \xi\in\p\D
$$
Here, $d$ denotes the choral distance in $\partial\D$.

Let  $F$ be an  inner function with $F(p)=p$ and  $p\in\D$ and 
  denote  denote $F^*(e^{2\pi i t})=e^{2\pi i S(t)}$ and   $f(t)=S(t)\mod 1$. If  $f$ is a Markov transformation, then
 the dynamic of $F^*$ on $\p\D$ is isomorphic to
the dynamic of  $f$ and 
we will inherit for $F^*$ the  dimension results obtained  in section \ref{MarkovT} for Markov transformations and sequences $\{r_n\} $ with 
$
u:=\lim_{n\to\infty}-[\log{r_n}]/n<\infty.
$

\

\noindent{\it Example 1:} For $B:\D \longrightarrow \D$ a  finite Blaschke product with 
 a fixed point $p\in \D$, 
but not an automorphism which is conjugated
to a rotation, we have 
\begin{equation} \label{dime}
\Dim \left\{ \xi\in\p\D: \liminf_{n\to\infty} \frac
{d((B^*)^n(\xi), \eta)}{r_n} =0\right\} =T\ge \frac {h}{h+u}
\end{equation}
where $T$ is the unique root of the equation $P_{top}(-t\log|f^\prime|)=ut$ and $h=
\int_{\p\D} \log |B'(z)|\, d\lan(z)$. For $B(z)=z^N$ we have that $T= {h}/{(h+u)}$ and $h=\log N$.

\medskip

The dynamic of $B^*$ on $\p\D$ is isomorphic to
the dynamic of a Markov transformation $f$ with a finite partition
${\cal P}_0$ (the number of  elements of ${\cal P}_0$ is  the number of factors of $B(z)$) and $f(P)=[0,1]$ for all $P\in {\cal P}_0$. The harmonic measure $\omega_p$ is  the ACIP measure.

\medskip

Next we will consider other two examples where the associate transformation  $f$ is  Markov with countable (but not finite) initial partition and having  BIP property. 

\medskip

\noindent{\it Example 2:} Consider the  infinite Blaschke product
$$
B(z)= z \prod_{k=1}^\infty \frac{z-a_k}{1-a_kz}\,, \qquad \hbox{$a_k=1-2^{-k}$}\,.
$$
\noindent Notice that $B$ is defined and $C^\infty$ in $\p\D\setminus\{1\}$ 
$$
|B'(z)| = \sum_{k=0}^\infty \frac {1-a_k^2}{|z-a_k|^2} \,,
\qquad \hbox{if $|z|=1, \ z\ne1$}\,,
$$
and  $S'(t)=|B'(e^{2\pi it})|>C>1$. Moreover,
it follows from Phragm\'en-Lindel\"of theorem that the image of $S(t)$ is $(-\infty,\infty)$
and so for $j\in\Z$ we can define the intervals $P_j= \{t\in (0,1): \ j<S(t)<j+1\}$. The transformation
$f:[0,1]\longrightarrow [0,1]$ given by $f(t)=S(t)$ (mod $1$), $f(0)=f(1)=0$, is a Markov transformation with partition ${\cal P}_0=\{P_j\}$  such that $f(P_j)=(0,1)$ and $\lan(P_j)\asymp 2^{-|j|}. $
 The dimension result   (\ref{dime})  holds ($\lan$-a.e $\eta\in \p\D$)  for  $B$ with   $T>0$  the unique root of the equation $P_{G}(-t\log|\widehat f^\prime|)=ut$.

\medskip

\noindent{\it Example 3:}
Consider the singular inner functions
$$
F(z)=e^{c\frac {z+1}{z-1}} \,, \qquad \hbox{for $c>2$}.
$$
These inner functions have only
one singularity at $z=1$ and its Denjoy-Wolff point $p$ is real
and it verifies $0<p<1$. We have that 
 $S(t) = -\frac
{c}{2\pi} \, \cot \pi t$ for $t\in [0,1]$,  
and $f(t)=S(t)$
(mod 1) is a Markov transformation with countable initial partition 
 ${\cal P}_0=\{P_j: j\in\Z\}$ where $P_j = \{t\in (0,1):
\ j<S(t)<j+1\}$ such that 
$f(P_j)=(0,1)$  and $\lan(P_j)=\arctan(2\pi c/(c^2+4\pi^2j(j+1)))$. We refers to  \cite{Mar}.
 The dimension result   (\ref{dime})  holds ($\lan$-a.e $\eta\in \p\D$)  for  $F$ with   $T>1/2$  the unique root of the equation $P_{G}(-t\log|\widehat f^\prime|)=ut$
and $h= \log \left( \frac {1}{1-p^2} \log \frac 1{p^2} \right)$.

\section{Induced Markov transformation. Intermittent systems} \label{intermittent}

We say that the map $F:[0,1]\longrightarrow [0,1]$ has {\it an induced Markov transformation } if there exist a finite or countable partition ${\cal P}_0$ of the interval $[0,1]$ and a {\it return time function} $R:\bigcup_j P^0_j\longrightarrow \Z^+$ which is constant on each block $P^0_j$, such that $R$ is not constant almost everywhere in $[0,1]$ and:
\begin{itemize}
\item[(i)] The induced map $f:[0,1]\longrightarrow [0,1]$, defined by $f(y)=F^{R(y)}(y)$ (with $R(y)=0$ in $[0,1]\setminus\cup_j P^0_j$) is a Markov transformation with partition ${\cal P}$
\item[(ii)] The return time function satisfies $\int R d\mu<\infty$ with $\mu$ the ACIP measure associated to the Markov transformation $f$.
\end{itemize}
If the ACIP measure $\mu$  is comparable to the Lebesgue measure $\lan$ in the whole interval $[0,1]$ (for example if BIP property holds), then we can write property (ii) as $\int R d\lan<\infty$. 
Above  definition includes some   transformations  modelled by a Young tower or  intermittent interval maps as  the Manneville-Pommeau transformation or the Liverani-Saussol-Vaienti transformation.

Let $\{r_n\}$ be a decreasing sequence   of positive numbers. For the Markov transformation $f$ 
 Borel Cantelli results state in (\ref{BC+}) hold. Moreover,
we also have  (see  \cite{FMP2} )  the following result  for $F$ (instead of $f$) for sequences such that $r_n\leq C\, r_{2n}$ for all $n$ with $C$ a positive constant (as $r_n=1/n^{1/\alpha})$:
$$
\mbox{If} \qquad \sum_nr_n^{\alpha}=\infty  \mbox{ for some } \alpha >1\implies  \liminf_{n\to\infty}\frac{|F^n(y)-x|}{r_n}=0\qquad \,  \mbox { for } \lan-a.e.\;  y\in [0,1]
$$
If  $\lan(P(n,x))\asymp\lan(P(n+1,x)$ for all $n$ (with constants depending on $x$), then the last implication also holds for $\alpha=1$.

If we  denote by $R_n(y)=\sum_{k=0}^{n-1} R(f^k(y))$, then $f^n(y)=F^{R_n(y)}(y)$. 
We can state all the previous dimension results on  target  sets for the Markov transformation $f$ in terms of the map $F$. More precisely, we have 
 dimension results for the set
 $$
 \{ y\in[0,1] : |F^{R_n(y)}(y)-x|\leq r_n \quad \mbox{for infinitely many } n\}
 $$
 However, if we want to get  dimension results for  the set
$$
 \{ y\in[0,1] : |F^{n}(y)-x|\leq r_n \quad \mbox{for infinitely many } n\}, 
 $$
then we will need  to study the existence of weak $\sigma$-Gibbs measures for the potentials $-t\log|\widehat F^{\prime}|$  instead of $-t\log |\widehat f'|$, and also the mixing properties of these measures.   For some intermittent systems 
the existence  of equilibrium measure  which are weak $\sigma$-Gibbs (in the symbolic context) for these potentials is well know.  So, we would be able to use our results for Markov shift and weak Gibbs measures  in these settings.  
For the existence of weak Gibbs measure for  intermittent system see \cite{PY}, \cite{Yu2}, \cite{MRT},  \cite{Ke}, \cite{Hu}.

 \subsection{The Manneville-Pomeau model for intermittency }
 
 Let us consider the family of 
 transformations $F:[0,1] \longrightarrow [0,1] $ with $$F(x)=x+x^{1+\alpha} \mod 1\quad \mbox{ with } \quad 0<\alpha<1$$
 These maps are uniformly expanding out of all neighborhood of the fixed point $0$.
 We know (see \cite{PW}) that  for the countable partition $\{ P^0_j\}$
 $$
 P_0^0=(a_0,1) \,  \quad   P_j^0=(a_j,a_{j-1}) \quad
 \mbox{with } \quad  a_0+a_0^{1+\alpha}=1\quad  \mbox{ and  }\quad  F(a_{j+1})=a_j
 $$
 and the return time function defined by $R(y)=j$ iff $y\in P_{j+1}^0=(a_{j+1},a_{j})$, 
 the induced map $f(y)=F^{R(y)}(y)$  satisfies conditions (i) and (ii). Notice that $f$ has BIP property, since $f(P_j^0)=(0,1)$ for all $j$, and  the condition (ii) is $\sum_j j \, \lan(P_j^0)<\infty$ but
 $$\lan(P_j^0)=a_{j-1}-a_j=a_j^{1+\alpha}\asymp \frac{1}{(\alpha (j+1))^{1+1/\alpha}}
 $$
  If  we denote, as usual, the symbolic representation of $f'$ as  $\widehat f'$, then  it is easy to check that $G(t)=P_G(-t\log |\widehat f'|)$  for $t\in(\alpha,1]$  is continuous, convex, $G(1)=0$ and $\lim_{t\to\alpha^+}G(t)=\infty.$ 
  Also, for $t\in(\alpha,1]$ 
 $$
 -\sum_j \lan(P_j^0)^t\log \lan(P_j^0)\asymp \sum_j \frac{1}{j^{t(1+1/\alpha)}}\log (\alpha j)  <\infty
 $$
We ask $x\in X_{\bf\Pi}$ be a point such that (\ref{hipx}) holds,
 and also we require the decreasing sequence $\{r_n\}$ verifies $\sum_n r_n<\infty$ and  $0\leq u:=\lim_{n\to\infty} -(\log {r_n})/n<\infty$. 
For any  block $P\in{\cal P}_N$ for some $N$, the set 
 $$\widetilde W_f(P,r_n,x)=\left\{ y\in P :\,  \liminf_{n\to\infty}\frac{|F^{R_n(y)}(y)-x|}{ r_n }=0 \right\}$$
with $f(y)=F^{R(y)}(y)$ verifies the following:
  \begin{theorem} Let $\alpha<T\leq 1$  be 
 the unique solution of $P_G(-t\log|\widehat{f^{\prime}}|)=tu$, then 
 $$
 \lan(\widetilde W_f(P,r_n,x))=0\quad \mbox{and}
 \quad 
 \Dim (\widetilde W_f(P,r_n,x))=T \, 
 $$ 
 \noindent Moreover,  for $u> (1+\frac {1}{\alpha})h_{\mu}$ 
 $$T\geq h_{\mu}/(h_{\mu}+u)$$
 Here $h_{\mu}$ is the entropy of  the ACIP measure of $f(y)=F^{R(y)}(y)$.
\end{theorem}
\noindent  If $\sum_n r_n=\infty$ we have  from Borel-Cantelli lemma that  $\lan(\widetilde W_f(P,r_n,x))=1$.

 \medskip

 Next, we will look for a Hausdorff dimension result for $F$ instead of $f(y)=F^{R(y)}(y)$.
Let us  consider now the initial partition ${\cal P}^0=\{ (0,a_0), (a_0,1)\}$. Notice that the full shift $(\Sigma^{\{0,1\}},\sigma)$ gives a symbolic representation of $F$ and
we have that $P_{top}(-\log |F'|)=0$. It is known  (see \cite{P}) that there exists an ACIP measure $\mu$. The  density function $h(x)$  of $\mu$ verifies 
$h(x)\asymp x^{-\alpha} $ for all $x>0$, and there exist constants $C>C'>0$ and $n_0>0$ such that for any measurable set $E\subset[0,1]$ and any $m$-block $P\in{\cal P}_m=\bigvee_{j=0}^m f^{-j}({{\cal P}^{0}})$
$$
|\mu(F^{-k}(E)\cap P)-\mu(E)\mu(P)|\leq C\, \frac{m^{\frac1\alpha-1}}{(k-m-n_0)^{\frac1\alpha-1}}\mu(E)\mu(P) \quad \mbox{ for all } k>n_0+m
$$
and if $P=P(m,0)$ with $m\geq n_0$
$$
|\mu(F^{-k}(E)\cap P)-\mu(E)\mu(P)|\geq C\, \frac{m^{\frac1\alpha-1}}{k^{\frac1\alpha-1}}\mu(E)\mu(P) \quad \mbox{ for all } k>m
$$
We refers to  \cite{Hu0} and \cite{Hu}, an also \cite{Y}. Notice that in particular,
$$
\mu(F^{-k}(E)\cap P)\leq (1+C\, \frac{m^{\frac1\alpha-1}}{(k-m-n_0)^{\frac1\alpha-1}}) \, \mu(E)\mu(P) \quad \mbox{ for all } k>n_0+m
$$
and so for any $k$-block $P(k,z)$ 
\begin{equation}\label{quasi}
\mu(F^{-k}(E)\cap P(k,z))\leq \mu(F^{-k}(E)\cap P(k-n_0-1,z))\leq C' \mu(E)\mu(P(k-n_0-1,z)
\end{equation}
with $C'=1+C(k-n_0-1)^{\frac1\alpha-1}$.

Moreover, the measure $\widehat\mu$  in $\Sigma^{\{0,1\}}$ is a weak $\sigma$-Gibbs measure for the  potential $\phi:=-\log|\widehat F^{\prime}|$ and an equilibrium measure (see \cite{Hu}). 
 In fact,  for all $0\leq t\leq 1$ there exists an exact  weak $\sigma$-Gibbs measure $\bf\widehat m_t$ for the potential $t\phi$ with $P=P_{top}(t\phi)$ which is an equilibrium measure (see theorem F in \cite{Hu}).  
Notice that, since $\widehat\mu$ is weak Gibbs  it follows from (\ref{quasi})  that
given $\ep>0$ for all $k$ large enough  (depending on $\ep$, $\alpha$ and $n_0$)

\begin{equation}\label{quasi2}
\mu(F^{-k}(E)\cap P(k,z))\leq \mu(F^{-k}(E)\cap P(k-n_0-1,z))\leq e^{\ep k} \mu(E)\, \widehat\mu(C(k,w))
\end{equation}
with $w\in \Sigma^{\{0,1\}}$ such that $\pi(C(k,w))=\mbox{cl}(P(k,z))$

Let $P$ denote a $N$-block of ${\cal P}_N$ such that $0\not\in P$ (this condition allows to use that $\lan\asymp \mu$ in $P$),  $x$ be (any) point in $[0,1]$, and  let us suppose again that $\sum_n r_n<\infty$ and 
\begin{equation}\label{defutilde}
\widetilde u:=\lim_{n\to\infty}-(\log \mu(I_n))/n<\infty
\end{equation}
with $I_n:=[x-r_n,x+r_n]$ if $x\neq 0,a_0,1$,  and $I_n:=[0,r_n]$ if $x=0$, $I_n=[a_0-r_n,a_0]$ if $x=a_0$ and $I_n:=[1-r_n,1]$  if $x=1$. 
Notice that if $u:=\lim_{n\to\infty}-(\log r_n)/n$, then since the density function of $\mu$ verifies $h(x)\asymp x^{-\alpha}$ we have that
 $$
  \widetilde u=
 \begin{cases}
u,\qquad \qquad \quad \text{ if } x\neq 0 \\
(1-\alpha)u, \qquad \text{ if } x=0.
 \end{cases}
 $$ 
 Then we have the following result:

  \begin{theorem} \label{buenoMP} Let $0\leq T\leq 1$  be 
 the unique solution of $P_{top}(-t\log|{F^{\prime}}|)=t\widetilde u$, then 
 $$
 \Dim \left\{ y\in P :\,  \liminf_{n\to\infty}\frac{|F^{n}(y)-x|}{ r_n }=0 \right\}=T \geq h_{\mu}/(h_{\mu}+\widetilde u)
 $$ 
 with $h_{\mu}$ is the entropy of  the ACIP measure of $F$.
\end{theorem}
\noindent Notice that the Hausdorff dimension is bigger for $x=0$.

\medskip

We will need the following distortion estimates whose proof is similar to the one of proposition 2.3 in \cite{Hu}. Recall $F(a_0)=a_0+a_0^{1+\alpha}=1$ and $F(a_{j+1})=a_j$; we define $a_{-1}=1$.
\begin{lemma}\label{distorsion2}
There exist $C_1,C_2>0$ and  $j_0\geq 0$ such that  if
$$
z,y\in (a_j,a_{j-1} ] \quad  \mbox{ and } \quad  F(z), F(y)\in (a_k,a_{k-1} ] \quad  \mbox{ for some } \quad j,k\geq 0
$$
then 
$$
\Delta(z,y)\,  \frac{F^{\prime}(z)}{F^{\prime}(y)}\leq \Delta(F(z),F(y)) 
$$
with
$$
\Delta(u,v)=
\begin{cases}
1+C_1\dfrac{|u-v|}{u}, \quad \mbox{ if } u,v\in (a_j,a_{j-1} ] \, \mbox{ with }  j>j_0\\
1+C_2|u-v|, \, \quad \mbox{ if } u,v\in (a_j,a_{j-1} ] \, \mbox{ with }  j\leq j_0
\end{cases}
$$
\end{lemma}

\

Notice that  since $F((a_{j+1},a_{j} ] )=(a_j,a_{j-1} ] $ for $j\geq 0$ we have the following :

\noindent Let  $\pi(z_0,z_1,\cdots))=z$ and   $\pi(y_0,y_1,\cdots))=y$  (with $\pi:\Sigma^{\{0,1\}}\longrightarrow [0,1]$ the natural projection) and ${\cal Q}=\{ (a_j,a_{j-1}] : j\geq 0\}$
\begin{itemize}
\item[(a)]If $z_m=1$ and $y_n=z_n$ for $0\leq n\leq m$, then 
$F^{n}(z), F^{n}(y)$ belong to the same element of ${\cal Q}$
\item[(b)] If $y_n=0$ for $0\leq n\leq m-1$, then $y\in[0,a_{m-1}]$ and $F^{n}(y)\in[F^{n-1}(y), a_{m-1-n}]$. Moreover, if $y_m=1$, then $y\in[a_m,a_{m-1}]$

\end{itemize}

From this observation and lemma \ref{distorsion2} we have that:

\begin{lemma}  \label{jacobiano}
\begin{itemize}
\item[\rm (i)] If $y\in P(m,z)$ and $z_m=1$, then  for all $0<s\leq m$
$$
{(F^{s})^{\prime}(z)}\asymp {(F^{s})^{\prime}(y)}
$$
\item[ \rm (ii)] If $y\in P(m,z)\neq P(m,0)$ with $z_r=1$ and $z_i=0$ for $r+1\leq i\leq m$, then $(F^{r})^{\prime}(y)\asymp (F^{r})^{\prime}(z)$  and  for all $r<s\leq m$
$$
{(F^{s})^{\prime}(y)}\asymp {(F^{r})^{\prime}(z)}{(F^{s-r-1})^{\prime}(a_{k(y)-r})}
$$
with $k(y)=\min\{k: k>m \mbox { and } y_k=1\}$. If $y\in P(m,0)$, then ${(F^{s})^{\prime}(y)}\asymp {(F^{s})^{\prime}(a_{k(y)})}$
\end{itemize}
\end{lemma}
\begin{proof}
Part (i) follows from lemma \ref{distorsion2} (as in  proposition 2.3 (ii) in \cite{Hu}) and property (a) since
$$
\frac{(F^{s})^{\prime}(z)}{(F^{s})^{\prime}(y)}=\prod_{j=0}^{s-1}\frac{(F)^{\prime}(F^j(z))}{(F)^{\prime}(F^j(y))}\leq \frac{\Delta(F^s(z),F^s(y)) }{\Delta(z,y)}
$$
We recall that $F^{n}(z), F^{n}(y)$ (for $0\leq n\leq m$) belongs to the same element of ${\cal Q}$, say $(a_j,a_{j-1}]$, and so
$$
\Delta(F^s(z),F^s(y))-1\leq \max\{C_1,C_2\}\frac{a_{j-1}-a_j}{a_j}=\max\{C_1,C_2\} a_j^{\alpha}\leq ctte
$$
From part (ii) notice that  $y_r=y_k=1$ and $y_i=0$ for $r<i<k$, then from (i) we  have that 
$(F^{r})^{\prime}(y)\asymp (F^{r})^{\prime}(z)$ and 
$(F^{s-r-1})^{\prime}(F^{r+1}(y))\asymp (F^{s-r-1})^{\prime}(a_{k-r})$. Recall $(F)^{\prime}(F^r(y))\asymp 1$.
\end{proof}

\begin{corollary} \label{MPdistorsion}
Let $1\leq s\leq n$.
\begin{itemize}
\item[\rm  (i)] If  $z_i=1$ for some $s\leq i\leq n$, then  
$$
\frac{\lan(F^s(P(n+1,z))}{\lan(F^s(P(n,z)))}\asymp \frac{\lan(P(n+1,z)}{\lan(P(n,z))}
$$

\item[\rm (ii)] If $z_r=1$  for some $0\leq r<s$  and $z_i=0$ for $r+1\leq i\leq n$, then  for all $r<s\leq n$
$$
\frac{\lan(P(n+1,z)}{\lan(P(n,z))}\asymp \left ( \frac{n-r}{n+1-r}\right)^{1/\alpha} \geq c\, \, \frac{\lan(F^s(P(n+1,z))}{\lan(F^s(P(n,z)))} \quad \text{ if } z_{n+1}=0
$$
and 
$$
\frac{\lan(P(n+1,z)}{\lan(P(n,z))}\asymp \frac{(n-r)^{1/\alpha}}{(n+1-r)^{1+1/\alpha}} \geq c\, \, \frac{\lan(F^s(P(n+1,z))}{\lan(F^s(P(n,z)))} \quad \text{ if } z_{n+1}=1
$$
for some  positive constant $c$.
\item[ \rm (iii)] If $z_i=0$ for $0\leq i\leq n$, then  for all $0\leq s\leq n$  the estimates in {\rm (ii) }holds with $r=-1$.
\end{itemize}

\end{corollary}
\begin{remark} \label{facil}
Notice that in particular we have that 
$$
\inf_x\liminf_{n\to\infty}\lan(P(n+1,x))/\lan(P(n,x)>0.
$$
\end{remark}

\begin{proof}
From  lemma  \ref{jacobiano} (i) we have $\lan(F^s(P(k,z)))\asymp {(F^{s})^{\prime}(z)}\lan(P(k,z))$  (for $k=n,n+1$), and we get part (i).
 From lemma \ref{jacobiano} (ii)  we know that if $y\in P(n,z)$ then $(F^{r+1})^{\prime}(y)\asymp (F^{r+1})^{\prime}(z)$.
 Therefore if $z_{n+1}=0$
\begin{equation}
\begin{aligned}
\frac{\lan(P(n+1,z))}{\lan(P(n,z))}&=\frac{\lan(P(n+1,z)) (F^{r+1})^{\prime}(z)}{\lan(P(n,z))(F^{r+1})^{\prime}(z)}\asymp \frac{\lan(F^{r+1}(P(n+1,z))}{\lan(F^{r+1}(P(n,z))}=
\frac{\lan(P(n-r,0)}{\lan(P(n-r-1,0))}=
\frac{a_{n-r}}{a_{n-r-1}} \\ \notag
&\asymp \left ( \frac{n-r}{n+1-r}\right)^{1/\alpha}\geq  \left ( \frac{n-(s-1)}{n+1-(s-1)}\right)^{1/\alpha}\asymp \frac{\lan(P(n+1-s,0)}{\lan(P(n-s,0))}=\frac{\lan(F^{s}(P(n+1,z))}{\lan(F^{s}(P(n,z))}
\end{aligned}
\end{equation}
In a similar way, if $z_{n+1}=1$, then $\lan(F^{r+1}(P(n+1,z))= \lan(P(n-r-1,0)-\lan(P(n-r,0)=a_{n-r-1}-a_{n-r}=a_{n-r}^{1+\alpha}$ and we get 
$$
\frac{\lan(P(n+1,z))}{\lan(P(n,z))}\asymp \frac{a_{n-r}^{1+\alpha}}{a_{n-r-1}} \asymp \frac{(n-r)^{1/\alpha}}{(n+1-r)^{1+1/\alpha}}.
$$
By using that $\lan(F^{s}(P(n+1,z))= \lan(P(n-s,0)-\lan(P(n+1-s,0)=a_{n-s}-a_{n+1-s}=a_{n+1-s}^{1+\alpha}$ we get (ii). Finally (iii) follows from taking $r=-1$  in the proof of (ii)
\end{proof}

\begin{proof}{ \it of theorem \ref{buenoMP}. } Since (\ref{liminf=0}) holds for $F$, an upper bound for the dimension follows from  the upper bound for $\Dim (\widetilde W_F(P,r_n,x))$
with 
$$
\widetilde W_F(P,r_n,x)=\left\{ y\in P :\, {|F^{n}(y)-x|}\leq { r_n }\, \mbox{ for infinitely many } n\right\}
$$
For all $N$ large enough we have the following covering of $\widetilde W_F(P,r_n,x)$
$$
\bigcup_{n=N}^\infty \{\mbox{cl}(f^{-n}(I_n)\cap P(n,y)): \ f^n(y)=x, \quad y\in P\}\,.
$$
with $I_n:=[x-r_n,x+r_n]$ if $x\neq 0,a_0,1$,  and $I_n:=[0,r_n]$ if $x=0$, $I_n=[a_0-r_n,a_0]$ if $x=a_0$ and $I_n:=[1-r_n,1]$  if $x=1$.
 Notice that for $n$ large enough $I_n\subset P(0,x)\cup\{x\}$.

From (\ref{quasi2}) and since $\lan\asymp\mu$ in $P$, we have that the diameter of each interval $F^{-n}(I_n)\cap P(n,y)$ verifies

$$
\diam(F^{-n}(I_n)\cap P(n,y))=\lan(F^{-n}(I_n)\cap P(n,y))\leq e^{\ep n}\,  \mu(I_n)\, \widehat\mu(C(n,w))
$$
with $w\in \Sigma^{\{0,1\}}$ such that $\pi(C(n,w))=\mbox{cl}(P(n,y))$
From this inequality and by definition of $\widetilde u$, see (\ref{defutilde}),
we can proceed 
as in proposition \ref{arriba}  (see remark \ref{uppermodeled}) and to get
$$
\Dim(\widetilde W_F(P,r_n,x))  \le \inf\{t>0: \, P_{top}(-t\log| F^\prime|)<\widetilde u\, t\}
$$ 
Again, since (\ref{liminf=0}) holds for $F$ we can get a lower bound for the dimension  of our set  by obtaining a lower bound for $\Dim(\widetilde W_F(P,r_n/n,x)) $. Notice that the value of $\widetilde u $ for the sequence $\{r_n/n\}$ coincides with the corresponding value for the  sequence $\{r_n\}$.  We  apply the same arguments that in the case of  target-ball sets for Markov transformations (see section \ref{ball}), we  define (in the same way) a  target-block set $W_F(P,\ell_n,x)\subset \widetilde W_F(P,r_n/n,x)) $. 
As in the proof of theorems \ref{markovT1}  and \ref{markovT2}, we use the symbolic patterns sets $\cal Z$ given by theorems
\ref{cantorsimbolicoformalismo}  and  \ref{cantorsimbolico}  . Notice that we have remark \ref{facil} and  so,   from  proposition \ref{parareg1},  we have equality between the Hausdorff dimension and the grid dimension of $\pi(\cal Z)$ 
\end{proof}

\

\begin{tabular}{cccc}
\small
\parbox{7cm}{
Mar\'{\i}a Victoria Meli\'an P\'erez\\
Departamento de Matem\'aticas\\
Universidad Aut\'onoma de Madrid \\
Campus de Cantoblanco\\
28049 Madrid, Spain  \\
E-mail: mavi.melian@uam.es 
}
\end{tabular}

\end{document}